\documentclass[11pt,reqno]{amsart}
\usepackage[a4paper,margin=1.08in]{geometry}
\usepackage[T1]{fontenc}
\usepackage{lmodern}
\usepackage{amsmath,amssymb,mathtools,mathrsfs,microtype,enumitem}
\usepackage[colorlinks=true,linkcolor=blue,citecolor=blue,urlcolor=blue]{hyperref}
\hypersetup{pdftitle={Arithmetic regularity and rational-time singularities for the periodic Benjamin--Ono equation},pdfauthor={Xi Chen}}
\numberwithin{equation}{section}
\newtheorem{theorem}{Theorem}[section]
\newtheorem{proposition}[theorem]{Proposition}
\newtheorem{lemma}[theorem]{Lemma}
\newtheorem{corollary}[theorem]{Corollary}
\theoremstyle{definition}

\newtheorem{remark}[theorem]{Remark}
\newcommand{\T}{\mathbb T}
\newcommand{\R}{\mathbb R}
\newcommand{\C}{\mathbb C}
\newcommand{\Z}{\mathbb Z}
\newcommand{\N}{\mathbb N}
\newcommand{\Q}{\mathbb Q}
\newcommand{\eps}{\varepsilon}
\newcommand{\e}{\mathrm e}
\title[Talbot effect for Benjamin--Ono]{Talbot effect for the periodic Benjamin--Ono equation}
\author{Xi Chen}
\address{Department of Mathematics and Computer Science, University of Basel, Spiegelgasse 1, 4051 Basel, Switzerland}
\email{xi01.chen@unibas.ch}
\thanks{ORCID: \href{https://orcid.org/0009-0001-8710-4499}{0009-0001-8710-4499}.}
\date{}
\begin{document}
\begin{abstract}
We study the dependence on time of the spatial regularity of solutions
to the periodic Benjamin--Ono equation with real, mean-zero initial data of bounded
variation. For every datum with a nonzero jump, we prove that continuity
and essential boundedness at an irrational time $t=2\pi\alpha$ are
equivalent to the arithmetic condition
$\sum_j\log(q_{j+1})/\sqrt{q_j}<\infty$, where $q_j$ are the
continued-fraction denominators of $\alpha$. We also determine the exact
H\"older regularity and the Besov scale for spatial exponents
$2<p\leq\infty$. We derive a high-frequency energy formula and
prove the same sharp local criteria on every spatial interval at every
irrational time for an open dense set of BV data. At rational times,
piecewise smooth data give finite sums of logarithmic and jump kernels
modulo a continuous remainder. For arbitrary BV data with a jump, the
nonlinear and linear solutions have different logarithmic Poisson
coefficients at every sufficiently large reduced denominator, uniformly
in the numerator. The proofs combine the smoothing of Tao's gauge
transform with estimates for quadratic Fourier series. A finite
averaging argument isolates an atom of the initial gauge measure while
retaining the real nonlinear reconstruction.
\end{abstract}
\maketitle
\setcounter{tocdepth}{1}
\tableofcontents

\section{Introduction}\label{sec:introduction}

We consider the periodic Benjamin--Ono equation
\begin{equation}\label{eq:BO}
 \partial_tu=H\partial_x^2u-\partial_x(u^2),
 \qquad u(0,x)=u_0(x),\qquad (t,x)\in\R\times\T,
 \quad \T=\R/(2\pi\Z),
\end{equation}
where $u$ is real-valued and $H$ is the periodic Hilbert transform,
with Fourier multiplier $-i\operatorname{sgn}(n)$ and
$\operatorname{sgn}(0)=0$. We write $S(t)u_0$ for the nonlinear
solution and $E(t)u_0$ for the linear solution, whose Fourier
multiplier is $e^{itn|n|}$. Throughout the main statements the data
have mean zero. Nonzero means are handled by the translation formula
in Proposition~\ref{cor:nonzero-mean}.
Global well-posedness in the real Sobolev spaces $H^s(\T)$,
$s>-1/2$, is established in \cite{GKT-wp,KLV}.
G\'erard \cite{Gérard-explicit} obtained an explicit formula for
the Benjamin--Ono flow on both the circle and the line.

The Talbot effect describes the different behavior of periodic
dispersive evolutions at rational and irrational times. In the optical
setting, the phenomenon is the reappearance of a diffraction grating
at distinguished propagation distances, as observed by Talbot
\cite{Talbot1836} and subsequently analyzed by Rayleigh
\cite{Rayleigh1881}. Integer, fractional, and fractal versions are
discussed by Berry and Klein \cite{BerryKlein1996}. Wen, Zhang,
and Xiao \cite{WenReview} survey developments in classical,
nonlinear, and quantum optics. One nonlinear optical realization is
the second-harmonic self-imaging in photonic crystals demonstrated
by Zhang, Wen, Zhu, and Xiao \cite{ZhangNonlinearTalbot}.

The connection with the Schr\"odinger equation is made precise by
Eceizabarrena \cite{Eceizabarrena}, who justified the paraxial
limit from the Helmholtz model to the free Schr\"odinger evolution
of a Dirac comb in the sense of distributions. For the
Schr\"odinger equation on the circle, the distinction between
rational and irrational times is governed by quadratic exponential
sums. Oskolkov \cite{Oskolkov} studied the underlying
Vinogradov-type series and their applications to dispersive equations.
Kapitanski and Rodnianski \cite{KapRod} related spatial regularity
to the arithmetic properties of time, and Rodnianski
\cite{Rodnianski} studied the fractal dimension of graphs arising
from BV data. Taylor \cite{Taylor2003} obtained explicit
rational-time formulas for the Schr\"odinger evolution on spheres.

The corresponding linear Benjamin--Ono evolution has
an additional feature: rational-time profiles can contain logarithmic
cusps as well as jumps. Boulton, Olver, Pelloni, and Smith
\cite{BOPS} identified such cusp revivals in linear nonlocal
equations. Boulton, Macpherson, and Pelloni
\cite{BMP} give the rational-time representation and prove
continuity, with fractal behavior under appropriate roughness
hypotheses, at almost every time.

Nonlinear smoothing can carry such information from a linear flow
to a nonlinear equation. Erdo\u{g}an and Tzirakis used this approach
for periodic KdV and cubic NLS \cite{ET-KdV,ET-NLS}, with
the conserved mean transport included for KdV; for cubic
NLS the comparison includes the usual mass-dependent phase.
Banica and Vega \cite{BanicaVega} also proved a Talbot effect for
cubic NLS solutions on the line associated with sums of Dirac
masses, in their study of polygonal lines evolving under the
binormal flow.

For Benjamin--Ono, the relevant smoothing occurs after Tao's gauge
transform \cite{Tao}. Isom, Mantzavinos, Oh, and Stefanov
\cite{IMOS} obtained smoothing for the periodic gauged evolution.
The theorem of G\'erard, Kappeler, and
Topalov \cite{GKT-smoothing} reconstructs the solution from a
quadratically evolving gauge profile multiplied by
$e^{i\partial_x^{-1}u(t)}$, with a smoother remainder. This
spatially varying phase is part of the leading term. It may change
the logarithmic and jump coefficients, and taking a real part may
conceal an unbounded complex contribution. Both issues must be
resolved to obtain a classification for the original real solution.
The same work already gives almost-everywhere $C^{1/2-}$ regularity
for data in $B^1_{1,1}$ \cite[Lemma~4.4]{GKT-smoothing}.
The results below allow general BV data with jumps and identify the
exact arithmetic conditions, including the continuity endpoint.

We prove that a single nonzero jump determines the set of irrational
times at which the solution is continuous. The classification allows
infinitely many jumps and an arbitrary atomless part of the
initial derivative. We then determine the positive regularity scale
and establish a quantitative separation between nonlinear and linear
singularities at rational times.

\subsection{Main results}

Let $BV_{r,0}(\T)$ denote the real, mean-zero functions of bounded
variation, equipped with the norm $\|f\|_{L^1}+|Df|(\T)$.
Here $Df$ is the distributional derivative and $|Df|(\T)$ its
total variation. A jump at $a$ has size
$[f]_a=f(a+)-f(a-)$; the one-sided limits exist for BV functions.
Continuity or H\"older regularity always means the existence of a
representative with that property. For an irrational $\alpha$, let
$p_j/q_j$ be its regular continued-fraction convergents, and set
\begin{equation}\label{intro:brjuno}
 \mathcal B(\alpha)=\sum_{j\geq0}
                \frac{\log q_{j+1}}{\sqrt{q_j}}.
\end{equation}
An initial repeated denominator may be omitted. This does not affect
any condition below. The convergence of \eqref{intro:brjuno} is
the $1/2$-Brjuno condition; see \cite{LMPS} for the general
$k$-Brjuno functions.

\begin{theorem}[Continuity at irrational times]\label{main:continuity}
Let $u_0\in BV_{r,0}(\T)$ have at least one nonzero jump, and
write $u(t)=S(t)u_0$. If $\alpha=t/(2\pi)$ is irrational, then
\begin{equation}\label{intro:fullendpoint}
 u(t)\in C^0(\T)
 \quad\Longleftrightarrow\quad u(t)\in L^\infty(\T)
 \quad\Longleftrightarrow\quad \mathcal B(\alpha)<\infty.
\end{equation}
If $u_0$ has only finitely many nonzero jumps, divergence of
$\mathcal B(\alpha)$ makes $u(t)$ essentially unbounded on every
nonempty open arc. There is also an open dense set
$\mathcal V\subset BV_{r,0}(\T)$ such that every member has a
nonzero jump and this local conclusion holds at every irrational
time simultaneously.
\end{theorem}

The finite-jump assertion imposes no restriction on the atomless
part of $Du_0$. For a general jump configuration, the same local
conclusion holds whenever one class of jump points under the
rotation $x\mapsto x+2t$ is finite, or contains one jump sufficiently
larger than the others in that class. Only an explicit countable set
of irrational times can fall outside the local classification for a
fixed datum.

The sufficient condition in Theorem~\ref{main:continuity} holds
for every BV datum, including continuous data. Its necessity uses
the atom corresponding to a jump. It is stronger than an almost-everywhere
statement: the sum in \eqref{intro:brjuno} converges at some Liouville
times and diverges at others. In particular, irrationality alone does
not imply continuity of the nonlinear solution.

\begin{theorem}[Sharp spatial regularity]\label{main:regularity}
Under the hypotheses of Theorem~\ref{main:continuity}, for every
irrational $\alpha=t/(2\pi)$ and $0<\gamma\leq1/2$,
\begin{equation}\label{intro:holder}
 u(t)\in C^\gamma(\T)
 \quad\Longleftrightarrow\quad
 \sup_j\frac{q_{j+1}}{q_j^{1/\gamma-1}}<\infty.
\end{equation}
Thus $u(t)\in C^{1/2}(\T)$ exactly when $\alpha$ has bounded
continued-fraction coefficients. At almost every time the graph
of $u(t)$ has upper Minkowski dimension $3/2$.
For every $u_0$ in the open dense set $\mathcal V$ from
Theorem~\ref{main:continuity}, criterion \eqref{intro:holder}
holds on every nonempty open arc at every irrational time.
\end{theorem}

The full Besov statement, including the summability endpoints, is
given in Proposition~\ref{bs:spectrum}. For example, if
$2<p\leq\infty$, $1/p<s\leq1/2$, and $1\leq r<\infty$, then
\[
 u(t)\in B^s_{p,r}(\T)
 \quad\Longleftrightarrow\quad
 \sum_j\left(\frac{q_{j+1}^{s-1/p}}
                   {q_j^{1-1/p-s}}\right)^{r/2}<\infty.
\]
These statements concern the real solution. Their lower bounds do
not follow merely by proving that its complex gauge profile is rough.
We compute the weak limit of the localized high-frequency energy
and show how its positive mass gives the necessary local estimates.
The total mass is proportional to the sum of the squared initial
jumps. The density may depend on the spatial point, through resonances
between jump locations.

For the rational-time statement, let $\mathsf P_r$, $0<r<1$,
denote the Poisson operator on $\T$, with Fourier multiplier
$r^{|n|}$. For $v\in L^2(\T)$, when the limit exists, define
\begin{equation}\label{intro:poissoncoefficient}
 \mathscr L_\xi(v)=\lim_{r\uparrow1}
 \frac{(\mathsf P_rv)(\xi)}{\log(1/(1-r))}.
\end{equation}
For a logarithmic singularity
$-A\log|2\sin((x-\xi)/2)|$ plus a locally bounded remainder,
this coefficient is $A$. Definition~\eqref{intro:poissoncoefficient}
continues to make sense when infinitely many singular centers
accumulate and no isolated expansion is available.

\begin{theorem}[Rational-time singularities]\label{main:rational}
Let $u_0\in BV_{r,0}(\T)$ have a nonzero jump.
There are constants $c_{u_0}>0$ and $Q_{u_0}<\infty$ such that,
for every reduced fraction $p/q$ with $q>Q_{u_0}$ and
$t=2\pi p/q$, the limits
$\mathscr L_\xi(S(t)u_0-E(t)u_0)$ exist at every $\xi\in\T$,
and at least $c_{u_0}q$ distinct points satisfy
\[
 \big|\mathscr L_\xi(S(t)u_0-E(t)u_0)\big|
       \geq c_{u_0}q^{-1/2}.
\]
The constants are independent of $p$. Consequently, the nonlinear
correction is essentially unbounded near each of these points.

If $u_0$ is piecewise $C^{1,\eta}$ on a finite partition, with
$0<\eta\leq1$, then at every rational time $t/(2\pi)$ the
solution is a finite sum of logarithmic and jump kernels modulo a
continuous remainder. Its singular part is nonzero at every such
time, including denominators below $Q_{u_0}$.
\end{theorem}

We give the coefficients and their squared-sum asymptotics in
Section~\ref{sec:rational-times}. For the square wave we also
determine the entire rational-time singular set and its cardinality.
The comparison $S(t)u_0-E(t)u_0$ has no continuous representative at any
nonzero rational time when $u_0=\lambda\operatorname{sgn}(x)$
and $0<|\lambda|\leq1$. Further results for amplitude families are
collected in Appendix~\ref{sec:amplitudefamilies}.

\subsection{The argument and its relation to earlier work}

The global BO flow and its conservation laws are supplied by
\cite{GKT-wp}. The gauge smoothing theorem \cite{GKT-smoothing}
reduces the analysis, up to a continuous remainder, to a real part
of a quadratic Fourier series multiplied by a H\"older phase.
For BV initial data, the Fourier coefficients of that series are
the coefficients of a finite measure divided by the frequency.
Its atoms are precisely the initial jumps, multiplied by explicit
nonzero phases.

The arithmetic kernel is
\[
 F_\alpha(\beta)=\sum_{n\geq1}
             \frac{e^{2\pi i(\alpha n^2+\beta n)}}n,
 \qquad \beta\in\R/\Z.
\]
The uniform quadratic-sum estimate recorded in \cite{BMP} gives
uniform convergence
under \eqref{intro:brjuno}. To prove necessity, we choose nested
spatial intervals on which successive rational approximations
contribute with a common sign. Controlling unfinished frequency
blocks gives one growth bound for the kernel and all its translates.
This control makes divergence stable under a sufficiently small
measure contribution.

The remaining step is to isolate an atom without discarding the
nonlinear phase or the real projection. Translation by $2\alpha$
changes $F_\alpha$ by an explicit phase, modulo a continuous
function. A finite average of these translated observations selects
one atom of the gauge measure. A quadratic average removes the
conjugate contribution introduced by the real part. We choose a
finite averaging length and then apply the stability estimate.
For local observations, a spatial cutoff restricts the translations
to the prescribed arc. A finite Gram matrix, or a strict inequality
between the jump sizes in one rotation class, prevents cancellation.

The positive regularity criteria use a different feature of the
same rational approximations. At suitable intermediate frequency scales between consecutive
convergent denominators, the quadratic sum is concentrated near a
finite spatial grid. We show that an atomic derivative contributes
a definite amount of energy there, whereas an atomless derivative
has negligible energy at the same scale. The real projection preserves
the leading energy after localization. These estimates give the
H\"older and Besov lower bounds and the local energy formula.

At rational times, finite Fourier inversion produces translated
logarithmic kernels. The nonlinear phase rotates each pair of real
logarithmic and jump coefficients without changing its combined
magnitude. To compare the nonlinear and linear coefficients uniformly
in time, we also use compactness of the BO orbit in $L^2$, which
follows from the positive Sobolev bounds in \cite{GKT-wp}. A
nonzero mean-zero orbit cannot approach a constant gauge phase.
This gives the positive gap needed in Theorem~\ref{main:rational}.

The distinction between a one-sided series and a symmetric
principal-value series matters at the continuity endpoint.
Rivoal and Seuret \cite{RS} study convergence of Hardy--Littlewood
series using continued fractions. The two-sided series considered
by Arkhipov and Oskolkov \cite{AO} has cancellation between positive
and negative frequencies that is absent from $F_\alpha$.
We state the precise kernel result and its proof in
Section~\ref{en:section}; the $1/2$-Brjuno condition is used here
as a classification for the actual real BO solution with a jump.

Alama Bronsard and Laurens \cite{BL} prove convergence of finite
dimensional approximations derived from explicit solution formulas
and display linear and nonlinear BO profiles for square-wave data.
Our rational-time calculations identify their singular locations
analytically and explain why the nonlinear correction can retain
singularities. Section~\ref{sec:compatibility} compares the
statements with those plots. Their $L^2$ convergence result is not
used as pointwise evidence for a singularity.

\subsection{Organization}

Section~\ref{sec:foundation} fixes the notation and gives the gauge
representation. Sections~\ref{ar:section} and~\ref{gm:section}
prove the kernel and nonlinear continuity criteria.
Sections~\ref{fe:section} and~\ref{sec:localization} establish the
sharp regularity estimates and their local versions.
Section~\ref{sec:rational-times} treats rational-time singularities,
and Section~\ref{sec:square} gives the square-wave applications.
The appendices contain the extension to finite gauge expansions,
the amplitude-family results, and a fixed-point Liouville example.

\section{Notation and the gauge representation}\label{sec:foundation}

\subsection{Function spaces and Fourier conventions}
We use normalized Lebesgue measure on $\T$:
\[
 \widehat f(n)=\frac1{2\pi}\int_\T f(x)e^{-inx}\,dx,
 \qquad \langle f,g\rangle=\frac1{2\pi}\int_\T f\overline g\,dx.
\]
All $L^p(\T)$ norms use this measure. For $s\in\R$,
\[
 \|f\|_{H^s}^2=\sum_{n\in\Z}(1+n^2)^s|\widehat f(n)|^2.
\]
The subscripts $r$ and $0$ indicate real-valued functions and zero
mean. We write
$H^{1/2-}=\bigcap_{0\leq s<1/2}H^s$.
The notation $C^0$ means continuous functions. For $0<\rho<1$,
$C^\rho$ is the H\"older space with seminorm
$[f]_{C^\rho}=\sup_{x\ne y}|f(x)-f(y)|/d(x,y)^\rho$,
where $d$ is the distance on the circle. When used as a H\"older
exponent, $\rho=1$ means Lipschitz continuity. A piecewise
$C^{1,\eta}$ function is $C^1$ on each arc of a finite partition,
with derivative extending to a $\eta$-H\"older function on the
closed arc; jumps are allowed at the partition points.

The space $BV(\T)$ consists of the functions in $L^1$ whose
distributional derivative is a finite Radon measure. We set
$\operatorname{Var}(f)=|Df|(\T)$ and
$\|f\|_{BV}=\|f\|_{L^1}+\operatorname{Var}(f)$.
For a finite complex measure $\mu$, its total variation norm is
$\|\mu\|_{\mathrm{TV}}=|\mu|(\T)$, and its Fourier coefficients
are unnormalized:
\[
 \widehat\mu(n)=\int_\T e^{-inx}\,d\mu(x).
\]
In particular, $\widehat{Df}(n)=2\pi in\widehat f(n)$.
We write $\mathcal J(f)=\{a\in\T:[f]_a\ne0\}$ for the jump set.
The atomic part of $Df$ is
$\sum_a[f]_a\delta_a$, an absolutely summable series; $\delta_a$
is the unit point mass at $a$. The remaining part is atomless and
need not be absolutely continuous. On $\T_1=\R/\Z$ we use
Lebesgue measure of mass one and the characters
$e(n\beta)$, where $e(z)=\exp(2\pi iz)$. Thus physical positions
$x,a\in\T$ become $\beta=x/(2\pi)$ and $a/(2\pi)$ on $\T_1$.

The Wiener algebra is
$\mathcal A(\T)=\{f:\sum_{n\in\Z}|\widehat f(n)|<\infty\}$,
with norm $\|f\|_{\mathcal A}=\sum_n|\widehat f(n)|$.
For a sequence $z=(z_n)_{n\geq1}$ we write
$\|z\|_{\mathfrak h^s}^2=\sum_{n\geq1}n^{2s}|z_n|^2$.
A rational time means $t/(2\pi)\in\Q$. The parameter
$\alpha=t/(2\pi)$ is used throughout; translation by $2t$
on $\T$ corresponds to translation by $2\alpha$ on $\T_1$.
We write $\|y\|_{\R/\Z}=\min_{k\in\Z}|y-k|$.
The analogous Wiener algebra on $\T_1$ uses the characters $e(n\beta)$.
The symbols $\lesssim$, $O$, and $\asymp$ denote bounds up to
constants independent of the variable under consideration;
subscripts indicate permitted dependence. A statement local on an
open arc means the corresponding statement after multiplication by
every smooth function compactly supported in that arc. In particular,
$C^\gamma_{\mathrm{loc}}(I)$ means H\"older regularity on compact
subarcs of $I$. Essential unboundedness on $I$ means failure of
membership in $L^\infty(I)$.

\subsection{The flow and Tao's gauge}
The global flow $S(t)$ is a group on $L^2_{r,0}(\T)$, preserves
$M=\|u_0\|_2^2$, and agrees with the flow on each positive Sobolev
space \cite{GKT-wp}. The mean-zero primitive $\partial_x^{-1}$
has multiplier $(in)^{-1}$ for $n\ne0$ and zero multiplier at
$n=0$. The projections $\Pi$ and $\Pi_{>0}$ retain,
respectively, the nonnegative and the strictly positive Fourier modes.
Set
\[
 \mathcal G(v)=\partial_x\Pi(e^{-i\partial_x^{-1}v}),
 \quad V(t)=\partial_x^{-1}S(t)u_0,
 \quad a_t=e^{iV(t)},\quad w_0=\mathcal G(u_0),
\]
and define the linear gauge profile by
\[
 w_L(t,x)=\sum_{n\geq1}e^{it(n^2-M)}\widehat w_0(n)e^{inx}.
\]
For $u_0\in H^s_{r,0}$, $0\leq s<1/2$, the smoothing theorem
\cite[Theorem~1.1]{GKT-smoothing} gives
\begin{equation}\label{eq:GKT}
 S(t)u_0=2\operatorname{Re}(i a_t w_L(t))+r_t,
 \qquad r_t\in H^{3s}(\T).
\end{equation}
If $u_0\in H^{1/2-}$, then $a_t,r_t\in C^\rho$ for every
$0<\rho<1$: apply Sobolev embedding to $V(t)\in H^{1+s}$
and $r_t\in H^{3s}$ with $s$ sufficiently close to $1/2$.
Unless explicitly stated, these are assertions at a fixed time.
Every BV function belongs to $H^{1/2-}$, since
$|\widehat u_0(n)|\leq\operatorname{Var}(u_0)/(2\pi|n|)$
for $n\ne0$.

For BV data the gauge can be expressed through a finite measure.
Let $V_0=\partial_x^{-1}u_0$, $g_0=e^{-iV_0}$,
$f_0=u_0g_0$, and
\begin{equation}\label{pre:gaugemeasure}
 \nu=-\frac1{2\pi}Df_0.
\end{equation}
The function $V_0$ is Lipschitz, so $f_0$ is BV. The identity
$g_0'=-if_0$ gives
\begin{equation}\label{pre:gaugecoeff}
 \widehat w_0(n)=-i\widehat f_0(n)=\frac{\widehat\nu(n)}n,
 \qquad n\geq1.
\end{equation}
An initial jump at $a$ therefore gives the atom
\begin{equation}\label{pre:atom}
 \nu(\{a\})=-\frac{[u_0]_a}{2\pi}e^{-iV_0(a)}.
\end{equation}
This coefficient never vanishes when the jump is nonzero.

\begin{proposition}[Nonzero means]\label{cor:nonzero-mean}
For real $u_0\in L^2(\T)$, let $m=\widehat u_0(0)$ and
$v_0=u_0-m$. Then the BO solution is
\begin{equation}\label{Galilean-intro}
 u(t,x)=m+S(t)v_0(x-2mt).
\end{equation}
Thus the spatial regularity conclusions and the nonlinear singular
coefficients transfer by translation. The comparison with the linear
flow in Theorem~\ref{main:rational} remains stated for mean-zero
data; for nonzero means its translated counterpart compares
$u(t,x)$ with $m+E(t)v_0(x-2mt)$.
\end{proposition}
\begin{proof}
For smooth data, substitution into \eqref{eq:BO} gives
$u_t=v_t-2mv_x=Hv_{xx}-\partial_x((m+v)^2)$, evaluated
at $x-2mt$. Approximation by smooth mean-zero data and continuity
of the $L^2$ flow give the same identity for $L^2$ data.
\end{proof}

\subsection{Finite expansions of the initial gauge}
Some of the rational-time and appendix results apply under a Fourier
hypothesis rather than bounded variation. Suppose $u_0\in L^2_{r,0}$
and, for distinct $a_1,\ldots,a_J\in\T$,
\begin{equation}\label{eq:FE}
 \widehat{\mathcal G(u_0)}(n)
   =\sum_{j=1}^J\frac{c_je^{-ina_j}}n+\rho_n,
 \qquad \sum_{n\geq1}n^{1+2\eps}|\rho_n|^2<\infty
\end{equation}
for some $\eps>0$. We do not require that the coefficients $c_j$
come from jumps of a BV function.

\begin{proposition}[Regularity implied by the gauge expansion]
\label{prop:auto-regularity}
Every real mean-zero $L^2$ datum satisfying \eqref{eq:FE} belongs
to $H^{1/2-}(\T)$.
\end{proposition}
\begin{proof}
Let $\Phi$ be the BO Birkhoff map and $\Phi_0$ its leading gauge
approximation. At the base regularity $s=0$,
\cite[Theorem~1.6 and Remark~1.7(ii)]{GKT-smoothing} give
\[
 \Phi_0(u_0)_n=-\frac{i}{\sqrt n}\widehat{\mathcal G(u_0)}(n),
 \qquad \Phi(u_0)-\Phi_0(u_0)\in\mathfrak h^1.
\]
The finite sum in \eqref{eq:FE}, after division by $\sqrt n$,
belongs to $\mathfrak h^r$ for every $r<1$, because
$\sum_n n^{2r-3}<\infty$. The divided remainder belongs to
$\mathfrak h^{1+\eps}$. Hence $\Phi(u_0)\in\mathfrak h^r$
for every $r<1$. The scale property
\cite[Proposition~A.1]{GKT-wp},
$\Phi^{-1}(\mathfrak h^{s+1/2})=H^s_{r,0}$ for $s\geq0$,
now gives the assertion by taking $r=s+1/2<1$.
\end{proof}

\begin{lemma}[Piecewise smooth data]\label{lem:piecewise-tail}
If $h$ is piecewise $C^{1,\eta}$ on a finite partition of $\T$,
where $0<\eta\leq1$, then
\begin{equation}\label{piecewise-tail}
 \widehat h(n)=\frac1{2\pi in}\sum_a[h]_a e^{-ina}
                         +O(n^{-1-\eta}),\qquad n\to+\infty.
\end{equation}
Consequently every real mean-zero piecewise $C^{1,\eta}$ datum
satisfies \eqref{eq:FE}, with
\[
 c_j=-\frac{[u_0]_{a_j}}{2\pi}e^{-iV_0(a_j)},
 \qquad 0<\eps<\eta.
\]
\end{lemma}
\begin{proof}
Integration by parts on the smooth arcs expresses $in\widehat h(n)$
as the sum of the jump contributions divided by $2\pi$ and the
Fourier coefficient of the classical, piecewise $C^\eta$ derivative
$v$. For a translation by $\delta$, the interfaces occupy total
length $O(|\delta|)$, and on their complement the H\"older bound
gives $\|v(\cdot+\delta)-v\|_1=O(|\delta|^\eta)$.
Taking $\delta=\pi/n$ in
$(e^{in\delta}-1)\widehat v(n)
 =\widehat{v(\cdot+\delta)-v}(n)$ proves
$\widehat v(n)=O(n^{-\eta})$. This proves \eqref{piecewise-tail}.
Apply it to $f_0=u_0e^{-iV_0}$ and use
$[f_0]_a=[u_0]_ae^{-iV_0(a)}$ and \eqref{pre:gaugecoeff}.
The remainder is in $\mathfrak h^{1/2+\eps}$ because
$\sum_n n^{-1-2(\eta-\eps)}<\infty$.
\end{proof}

For later use, let $s(x)=\operatorname{sgn}(x)$ on $(-\pi,\pi)$,
periodically extended. It satisfies \eqref{eq:FE} simply because
it is piecewise constant. Its mean-zero primitive is
$|x|-\pi/2$. The explicit coefficients needed for the square-wave
applications will be computed in Section~\ref{sec:square}.

\section{Quadratic Fourier series}\label{ar:section}

The quadratic kernel governs both the continuity endpoint and positive
regularity. We first prove uniform upper bounds, then show that divergence
of the arithmetic sum forces local unboundedness after a nontrivial real
projection. All series in this section are on $\T_1=\R/\Z$.

For an irrational number $\alpha$, write $p_j/q_j$ for its regular
continued-fraction convergents, omitting an initial repeated denominator
if necessary.  Set
\begin{equation}\label{ar:CFcondition}
  \mathcal B(\alpha):=
  \sum_{j\geq 0}\frac{\log q_{j+1}}{\sqrt{q_j}}.
\end{equation}
All assertions involving this sum are unchanged by adding an integer to
$\alpha$.  We use the convention $e(z)=\exp(2\pi i z)$ and write
\[
  F_\alpha(\beta)=\sum_{n\geq 1}
  \frac{e(\alpha n^2+\beta n)}{n}.
\]
Initially this series defines an $L^2$ function of
$\beta\in\R/\Z$.

\begin{lemma}[Uniform quadratic sums]\label{ar:quadratic}
There is an absolute constant $C$ with the following property.  If
$p/q$ is reduced and $|\alpha-p/q|\leq q^{-2}$, then
\begin{equation}\label{ar:quadraticbound}
 \sup_{\beta\in\R}
 \left|\sum_{n=1}^{N}e(\alpha n^2+\beta n)\right|
 \leq C\left(\frac{N}{\sqrt q}+\sqrt q\right).
\end{equation}
The same bound holds on any interval of $N$ consecutive integers.
\end{lemma}

The estimate, including uniformity in the linear parameter, is recorded
in the proof of \cite[Lemma~1]{BMP}; see the quadratic-sum estimate in that proof.
Related continued-fraction estimates appear in
\cite[Proposition~3, equation~(3.1)]{RS}, in the phase normalization
$\exp(\pi i x n^2)$ with $x=2\alpha$. The bound used here
has the normalization of \cite{BMP}. The interval version follows
by translating the summation variable: the quadratic coefficient stays
fixed and only the linear coefficient changes.

\begin{proposition}[A continued-fraction criterion]\label{ar:uniform}
If $\mathcal B(\alpha)<\infty$, the ordinary partial sums defining
$F_\alpha$ converge uniformly in $\beta$.  In particular,
$F_\alpha$ has a continuous representative.
\end{proposition}

\begin{proof}
Suppose $q_j\leq N<q_{j+1}$.  By Lemma~\ref{ar:quadratic}, every
partial sum over an interval contained in $[N,2N]$ is bounded by
$CN/\sqrt{q_j}$.  Summation by parts therefore gives
\begin{equation}\label{ar:weightedblock}
 \sup_{\beta\in\R}
 \left|\sum_{n=N}^{M}\frac{e(\alpha n^2+\beta n)}n\right|
 \leq \frac{C}{\sqrt{q_j}},\qquad N\leq M<2N.
\end{equation}
There are at most $C(1+\log(q_{j+1}/q_j))$ powers of two in
$[q_j,q_{j+1})$.  Consequently, the sum of the sup norms of all
dyadic blocks is finite if
\[
 \sum_j\frac{1+\log(q_{j+1}/q_j)}{\sqrt{q_j}}<\infty.
\]
This follows from \eqref{ar:CFcondition}, since the denominators grow
at least geometrically every two steps.  The same estimate
\eqref{ar:weightedblock} controls an unfinished final block, and tends
to zero as its initial frequency tends to infinity.  Thus all ordinary
partial sums, not just the sums at dyadic endpoints, converge
uniformly.  Their limit is continuous.
\end{proof}

\begin{remark}\label{ar:liouvillecontinuous}
The criterion includes Liouville numbers.  To see this directly, choose
the continued-fraction coefficients recursively, after a sufficiently
long initial segment, so that
\[
 a_{j+1}=\left\lfloor\frac{\exp(q_j^{1/4})}{q_j}\right\rfloor.
\]
Then $q_{j+1}=a_{j+1}q_j+q_{j-1}\sim\exp(q_j^{1/4})$.
It follows that
$\log(q_{j+1})/\sqrt{q_j}\sim q_j^{-1/4}$, which is summable,
whereas $\log(q_{j+1})/\log(q_j)\to\infty$.  The usual convergent
estimate $|\alpha-p_j/q_j|<(q_jq_{j+1})^{-1}$ shows that the resulting
number is Liouville.  The initial finite segment is arbitrary, so
these examples are dense.
\end{remark}

\begin{proposition}[H\"older regularity at finite-type times]
\label{ar:holderkernel}
Suppose that, for some $c>0$ and $\kappa\geq 1$,
\begin{equation}\label{ar:finite-type}
 \|m\alpha\|_{\R/\Z}\geq c m^{-\kappa},\qquad m\geq 1.
\end{equation}
Then $F_\alpha\in C^{1/(\kappa+1)}(\R/\Z)$.
\end{proposition}

\begin{proof}
The convergent inequalities and \eqref{ar:finite-type} imply
$q_{j+1}\leq C_\alpha q_j^\kappa$.  Given a sufficiently large $N$,
choose $j$ so that
\[
 \sqrt{q_{j-1}q_j}\leq N<\sqrt{q_jq_{j+1}}.
\]
The upper inequality gives $q_j\geq c_\alpha
N^{2/(\kappa+1)}$.  The lower inequality, together with
$q_j\leq C_\alpha q_{j-1}^\kappa$, gives
$q_j\leq C_\alpha N^{2\kappa/(\kappa+1)}$.  Applying
Lemma~\ref{ar:quadratic} with denominator $q_j$ now yields
\begin{equation}\label{ar:typeweyl}
 \sup_\beta\left|\sum_{n=1}^{N}e(\alpha n^2+\beta n)\right|
 \leq C_\alpha N^\theta,
 \qquad \theta=\frac{\kappa}{\kappa+1}.
\end{equation}
The same bound holds on intervals of length $N$.

For a dyadic integer $N$, let
\[
 P_N(\beta)=\sum_{N\leq n<2N}
 \frac{e(\alpha n^2+\beta n)}n.
\]
Summation by parts and \eqref{ar:typeweyl} give
$\|P_N\|_\infty\leq C_\alpha N^{\theta-1}$.  Differentiating this
finite sum and using \eqref{ar:typeweyl} once more gives
$\|P_N'\|_\infty\leq C_\alpha N^\theta$.
For $0<|h|<1/2$, split the dyadic blocks at a scale $N_0$ comparable
to $|h|^{-1}$.  The low-frequency increments have total size at most
\[
 C_\alpha |h|\sum_{N\leq N_0}N^\theta
 \leq C_\alpha |h|N_0^\theta,
\]
and the high-frequency increments have total size at most
\[
 C_\alpha\sum_{N>N_0}N^{\theta-1}
 \leq C_\alpha N_0^{\theta-1}.
\]
Both sums are over dyadic $N$.  This proves the required estimate
$|F_\alpha(\beta+h)-F_\alpha(\beta)|
\leq C_\alpha |h|^{1/(\kappa+1)}$.
\end{proof}

\subsection{Poisson means and multiplication}
We use the periodic Poisson operators \(\mathsf P_r\), \(0<r<1\), with Fourier
multipliers \(r^{|n|}\).  Their kernels are normalized to have integral
one with respect to \(dx/(2\pi)\).

\begin{lemma}\label{li:commutator}
If \(1/2<\gamma<1\), \(b\in C^\gamma(\T)\), and \(f\in L^2(\T)\),
then, for \(1/2\le r<1\),
\[
 \|\mathsf P_r(bf)-b\mathsf P_rf\|_{L^\infty}
 \le C_\gamma [b]_{C^\gamma}\|f\|_{L^2}
       (1-r)^{\gamma-1/2}.
\]
Here \([b]_{C^\gamma}\) is the H\"older seminorm with respect to the
periodic distance.
\end{lemma}
\begin{proof}
Writing \(\epsilon=1-r\), the Poisson kernel satisfies
\(\mathsf P_r(y)\le C\epsilon/(\epsilon^2+y^2)\) for \(|y|\le\pi\).
The commutator is the convolution integral with integrand
\(\mathsf P_r(y)(b(x-y)-b(x))f(x-y)\).  Cauchy--Schwarz and the H\"older
estimate reduce the assertion to
\[
 \int_{-\pi}^{\pi}
 \frac{\epsilon^2|y|^{2\gamma}}{(\epsilon^2+y^2)^2}\,dy
 \le C_\gamma\epsilon^{2\gamma-1},
\]
which follows by the substitution \(y=\epsilon z\).
\end{proof}

\subsection{The continuity criterion for the quadratic kernel}
\label{en:section}

For an irrational \(\alpha\), let \(p_j/q_j\) be its regular
continued-fraction convergents, omitting an initial repeated denominator,
and put
\[
 \mathcal B(\alpha)=\sum_j\frac{\log q_{j+1}}{\sqrt{q_j}},
 \qquad
 F_\alpha(\beta)=\sum_{n\ge1}\frac{e(\alpha n^2+\beta n)}n,
 \qquad e(z)=e^{2\pi iz}.
\]
The spatial variable \(\beta\) belongs to \(\T_1=\R/\Z\), and the
series initially defines an \(L^2(\T_1)\) function.  The sufficient
condition \(\mathcal B(\alpha)<\infty\) is given by
Proposition~\ref{ar:uniform}.  We prove its necessity, together with a
local statement that survives multiplication by a variable complex phase.

The condition \(\mathcal B(\alpha)<\infty\) is the
\(1/2\)-Brjuno condition in the terminology of
Lee--Marmi--Petrykiewicz--Schindler \cite[Section~1.2]{LMPS}.
Rivoal and Seuret \cite[Theorem~1(ii)]{RS} obtained related
sufficient conditions for two-parameter Hardy--Littlewood series,
using the phase $\exp(\pi i x n^2)$; here $x=2\alpha$.
The uniform sufficiency proof in Proposition~\ref{ar:uniform}
uses the present normalization directly. Proposition~\ref{en:classification} identifies the
corresponding criterion for boundedness and continuity in the spatial
variable and proves local unboundedness when the arithmetic sum diverges.
The profile considered here also differs from the symmetric
principal-value kernel studied by Arkhipov--Oskolkov
\cite[Theorem~1]{AO}, since
\(\operatorname{p.v.}\sum_{n\in\Z\setminus\{0\}}
 e(\alpha n^2+\beta n)/n=F_\alpha(\beta)-F_\alpha(-\beta)\)
has the cancellation between positive and negative frequencies that
underlies their uniform boundedness result.

\begin{proposition}\label{en:classification}
For every irrational \(\alpha\), the following conditions are equivalent:
\[
 \mathcal B(\alpha)<\infty,\qquad
 F_\alpha\in C^0(\T_1),\qquad
 F_\alpha\in L^\infty(\T_1).
\]
More precisely, suppose \(\mathcal B(\alpha)=\infty\),
\(b\in C^\rho(\T_1)\), and \(1/2<\rho<1\).  Then
\(\operatorname{Re}(bF_\alpha)\notin L^\infty(I)\) for every open arc
\(I\subset\T_1\) on which \(b\) is not identically zero.
\end{proposition}

The proof uses a spatial choice at each long rational resonance.  A
neighborhood wide enough to retain a fixed proportion of one logarithmic
peak contains many points of the next rational grid.  Among them, a Gauss
sum has the required real projection.  Nesting these neighborhoods makes
the contributions of all relevant scales accumulate at a single point.

\subsubsection{A local choice of Gauss phases}

Let \((p,q)=1\), and define
\[
 \mu_l=\frac1q\sum_{n=0}^{q-1}e\bigl((pn^2+ln)/q\bigr),
 \qquad d=\gcd(2,q),\qquad z_l=\sqrt{q/d}\,\mu_l.
\]
The usual squared-modulus calculation gives
\begin{equation}\label{en:gauss-magnitude}
 |\mu_l|^2=
 \begin{cases}
 q^{-1},&q\text{ odd},\\
 q^{-1}\bigl(1+(-1)^{pq/2+l}\bigr),&q\text{ even}.
 \end{cases}
\end{equation}
Thus the nonzero \(z_l\) have modulus one; their indices are the whole
grid when \(q\) is odd and one parity subgrid when \(q\) is even.

\begin{lemma}\label{en:positive-phase}
There is an absolute constant \(C\) such that the following holds.
Let \(J\subset\T_1\) be an interval satisfying
\( |J|\sqrt q\ge C\log(2q)\), and let \(|\zeta|=1\).
There is an index \(l\) with \(l/q\in J\), \(\mu_l\ne0\), and
\[
 \operatorname{Re}(\zeta\mu_l)\ge\tfrac18\sqrt{d/q}.
\]
\end{lemma}
\begin{proof}
Finite Fourier orthogonality gives, for every integer \(k\),
\[
 \sum_{l=0}^{q-1}z_le(kl/q)=\sqrt{q/d}\,e(pk^2/q),
\]
and
\[
 \sum_{l=0}^{q-1}z_l^2e(kl/q)
 =\frac1d e(pk^2/q)
       \sum_{n=0}^{q-1}e\bigl((2pn^2+2pkn)/q\bigr).
\]
For integer \(A,B\), expanding the square and summing the inner geometric
progression shows that
\[
 \left|\sum_{n=0}^{q-1}e\bigl((An^2+Bn)/q\bigr)\right|^2
 \le q\gcd(2A,q).
\]
Since \(\gcd(4p,q)\le4\), both preceding Fourier sums have modulus
at most \(2\sqrt q\).  Fourier completion over any interval of grid
indices, using the bound \(O(q\log(2q))\) for the sum of the absolute
Fourier coefficients of its indicator, consequently gives
\begin{equation}\label{en:local-moments}
 \left|\sum_{l/q\in J}z_l\right|
 +\left|\sum_{l/q\in J}z_l^2\right|
 \le C_0\sqrt q\log(2q).
\end{equation}
The number \(L\) of nonzero terms in this interval is
\(q|J|/d+O(1)\), by \eqref{en:gauss-magnitude}.  On these indices set
\(Y_l=\operatorname{Re}(\zeta z_l)\).  If the constant in the lemma is
large enough, \eqref{en:local-moments} implies
\[
 \left|\frac1L\sum Y_l\right|\le\frac1{16},
 \qquad
 \frac1L\sum Y_l^2
 =\frac12+\frac1{2L}\operatorname{Re}
       \left(\zeta^2\sum z_l^2\right)\ge\frac38.
\]
If every \(Y_l\le1/8\), then \(-1\le Y_l\le1/8\) and
\(Y_l^2\le1/8-7Y_l/8\).  Averaging contradicts the second bound.
Thus some \(Y_l>1/8\), as required.
\end{proof}

\subsubsection{A logarithmic expansion on a broad rational neighborhood}

We use the uniform quadratic-sum estimate
\begin{equation}\label{en:weyl}
 \sup_\beta\left|\sum_{n=1}^N e(\alpha n^2+\beta n)\right|
 \le C\left(\frac N{\sqrt q}+\sqrt q\right)
\end{equation}
for a reduced approximation \(|\alpha-p/q|\le q^{-2}\).
The estimate is uniform in the linear parameter; see
\cite[proof of Lemma~1, equations~(13)--(14)]{BMP}, or
Lemma~\ref{ar:quadratic}.  We shall use it also with the next convergent
denominator.

Fix a convergent \(p/q\), write \(q'\) for the next denominator,
and put
\[
 \delta=\alpha-p/q,\qquad
 B=\left\lfloor\tfrac1{10}\sqrt{qq'}\right\rfloor,
 \qquad U=B/q,\qquad r=\frac1{q\sqrt U}.
\]
The next lemma is used only when \(B\ge2q\).

\begin{lemma}\label{en:log-block}
If \(\beta=l/q+\theta\), \(|\theta|\le r\), and \(q\le M\le B\),
then
\begin{equation}\label{en:log-expansion}
 \sum_{n=q}^M\frac{e(\alpha n^2+\beta n)}n
 =\mu_l\log\left(\frac{\min\{M,|\theta|^{-1}\}}q\right)
     +O\left(\frac{\log(2q)}{\sqrt q}\right),
\end{equation}
where \(|\theta|^{-1}=\infty\) when \(\theta=0\).
The constant is uniform in all the parameters above.  The main logarithm
is nonnegative; when \(M=B\), it is at least \(\tfrac12\log U\).
\end{lemma}
\begin{proof}
The periodic sequence
\(a_n=e((pn^2+ln)/q)\) has mean \(\mu_l\).  Fourier completion of a
partial period, followed by \eqref{en:gauss-magnitude}, gives
\[
 \sup_K\left|\sum_{n=1}^K(a_n-\mu_l)\right|
 \le C\sqrt q\log(2q).
\]
Indeed, after removing the zero Fourier mode, the sum of the geometric
progression bounds over the remaining frequencies is \(O(q\log(2q))\),
while every normalized Gauss coefficient is at most \(\sqrt{2/q}\).
Periodicity extends this estimate to every \(K\).

For \(f(y)=e(\delta y^2+\theta y)/y\), its endpoint size and total
variation on \([q,M]\) are bounded by
\[
 C\left(\frac1q+|\delta|B+|\theta|\log(B/q)\right)\le\frac Cq.
\]
Here \(|\delta|<1/(qq')\), and
\(|\theta|\log U\le q^{-1}(\log U)/\sqrt U\).
Summation by parts therefore bounds the zero-mean contribution by
\(C\log(2q)/\sqrt q\).

In the mean contribution, removing \(e(\delta n^2)\) costs at most
\[
 C|\mu_l|\,|\delta|\sum_{n\le M}n
 \le Cq^{-1/2}|\delta|B^2\le Cq^{-1/2}.
\]
Finally, for \(|\theta|\le1/q\),
\[
 \sum_{n=q}^M\frac{e(\theta n)}n
 =\log\left(\frac{\min\{M,|\theta|^{-1}\}}q\right)+O(1).
\]
Below \(|\theta|^{-1}\), subtract the harmonic sum and use the linear
Taylor estimate.  Above that frequency, the geometric-sum bound and
summation by parts give a bounded remainder.  This proves
\eqref{en:log-expansion}.  The lower bound at \(M=B\) follows from
\(|\theta|^{-1}\ge q\sqrt U\) and \(B=qU\).
\end{proof}

\subsubsection{Discarding summably small blocks}

Assume now that \(\mathcal B(\alpha)=\infty\), and write
\(b_j=\log(q_{j+1})/\sqrt{q_j}\).  Call an index selected when
\begin{equation}\label{en:selected}
 b_j\ge q_j^{-1/8}.
\end{equation}
Since continued-fraction denominators grow geometrically every two steps,
\(\sum_jq_j^{-\epsilon}<\infty\) for every \(\epsilon>0\).
The sum of the selected \(b_j\) therefore diverges, and at these indices
\begin{equation}\label{en:large-gap}
 q_{j+1}\ge\exp(q_j^{3/8}).
\end{equation}
Finitely many initial indices will be discarded without further mention.
For each selected index use \(B_j,U_j,r_j\) from the preceding lemma.

The contribution of all unselected blocks \([q_j,q_{j+1})\) has a
summable sequence of uniform bounds, including every unfinished block.
Indeed, \eqref{en:weyl} with denominator \(q_j\) and summation by parts
give
\[
 \sup_\beta\sup_{q_j\le M<q_{j+1}}
 \left|\sum_{n=q_j}^M\frac{e(\alpha n^2+\beta n)}n\right|
 \le\frac{C(1+\log(q_{j+1}/q_j))}{\sqrt{q_j}}
 \le C(q_j^{-1/2}+b_j).
\]
This is summable over the unselected indices by \eqref{en:selected}.

At a selected index, the upper complement \((B_j,q_{j+1})\) also has
a summable bound.  The estimate \eqref{en:weyl} is valid also when the summation length
is smaller than the denominator. With $Q=q_{j+1}$, it bounds sums
on intervals of length at most $y$ by $C(y/\sqrt Q+\sqrt Q)$.
Summation by parts on $(B_j,M]$ therefore gives
$C((1+\log(M/B_j))/\sqrt Q+\sqrt Q/B_j)$.
In particular, uniformly for $B_j<M<q_{j+1}$,
\begin{equation}\label{en:upper-block}
 \sup_\beta
 \left|\sum_{n=B_j+1}^M\frac{e(\alpha n^2+\beta n)}n\right|
 \le C\left(\frac{\log q_{j+1}}{\sqrt{q_{j+1}}}
                  +\frac{\sqrt{q_{j+1}}}{B_j}\right)
 \le C\left(\frac{\log q_{j+1}}{\sqrt{q_{j+1}}}
                  +\frac1{\sqrt{q_j}}\right).
\end{equation}
Both terms are summable over \(j\).  Thus the only contributions that
can accumulate are the selected lower blocks \([q_j,B_j]\).
Their logarithmic heights satisfy, by \eqref{en:large-gap},
\begin{equation}\label{en:height}
 \frac{\log U_j}{\sqrt{q_j}}\ge c b_j
\end{equation}
for all sufficiently large selected indices.  The errors in
\eqref{en:log-expansion} are summable as well.

\subsubsection{Nested spatial choices and the nonlinear phase}

\begin{proposition}\label{en:projected-divergence}
Suppose \(\mathcal B(\alpha)=\infty\).  If \(b\) is continuous on
an open arc \(I\) and is not identically zero there, then some
\(\beta_*\in I\) satisfies
\[
 \operatorname{Re}\left(
 b(\beta_*)\sum_{n=1}^N\frac{e(\alpha n^2+\beta_* n)}n\right)
 \longrightarrow+\infty\qquad(N\to\infty).
\]
The same conclusion holds for the Abel sums as their radius tends to one.
\end{proposition}
\begin{proof}
Choose \(\beta_0\in I\) with \(b_0=b(\beta_0)\ne0\), and shrink to
an interval \(I_0\) compactly contained in \(I\) such that
\( |b(\beta)-b_0|\le |b_0|/16\) on \(I_0\).
Put \(\zeta=b_0/|b_0|\).

List all sufficiently large selected indices as \(j_1<j_2<\cdots\).
At the first one, apply Lemma~\ref{en:positive-phase} in the middle half
of \(I_0\), choosing the index sufficiently large that the closed
interval of radius \(r_{j_1}\) about the chosen grid point is contained
in \(I_0\).  Suppose such an interval \(I_k\), centered at
\(l_k/q_{j_k}\), has been chosen.  Apply the lemma at the next selected
denominator in the middle half of \(I_k\).  Its length hypothesis holds
for every sufficiently large selected index: since
\(q_{j_{k+1}}\ge q_{j_k+1}\) and \(\sqrt x/\log(2x)\) is eventually
increasing,
\[
 \frac{r_{j_k}\sqrt{q_{j_{k+1}}}}{\log(2q_{j_{k+1}})}
 \ge\frac{r_{j_k}\sqrt{q_{j_k+1}}}{\log(2q_{j_k+1})}
 \ge c\frac{q_{j_k+1}^{1/4}}
              {q_{j_k}^{3/4}\log(2q_{j_k+1})}\longrightarrow\infty.
\]
The last limit follows from \eqref{en:large-gap}.  Moreover,
\(r_{j_{k+1}}\le1/q_{j_{k+1}}\) and
\(q_{j_{k+1}}r_{j_k}\to\infty\).  The new closed interval, of radius
\(r_{j_{k+1}}\), therefore fits inside \(I_k\).

This constructs nested closed intervals with a common point
\(\beta_*\in I_0\).  At each selected index,
\[
 \operatorname{Re}(b_0\mu_{l_k})
       \ge\frac{|b_0|}{8}\sqrt{d/q_{j_k}},
\]
where \(d=\gcd(2,q_{j_k})\).  By
\eqref{en:gauss-magnitude} and the variation bound for \(b\),
\begin{equation}\label{en:positive-projection}
 \operatorname{Re}(b(\beta_*)\mu_{l_k})
       \ge\frac{|b_0|}{16}\sqrt{d/q_{j_k}}
       \ge\frac{|b_0|}{16\sqrt{q_{j_k}}}.
\end{equation}
Also \(|\beta_*-l_k/q_{j_k}|\le r_{j_k}\), so
Lemma~\ref{en:log-block} applies at every selected scale.

By \eqref{en:height} and \eqref{en:positive-projection}, each completed
selected lower block contributes at least a positive constant times
\(b_{j_k}\), apart from a summable error.  Every unfinished selected
lower block has a nonnegative main projection and only a summably bounded
negative error.  The unselected blocks and selected upper complements
are summably bounded by the preceding subsection.  Since
\(\sum_k b_{j_k}=\infty\), the projected real parts of all ordinary
partial sums tend to positive infinity.

Finally, the Abel sum is
\[
 (1-r)\sum_{N\ge1}r^N
       \sum_{n=1}^N\frac{e(\alpha n^2+\beta_*n)}n.
\]
The positive weights have asymptotic total mass one and lose their mass
on every finite initial segment as \(r\uparrow1\).  The same projected
divergence follows.
\end{proof}

\begin{proof}[Proof of Proposition~\ref{en:classification}]
The implication \(\mathcal B(\alpha)<\infty\Rightarrow F_\alpha\in C^0\)
is Proposition~\ref{ar:uniform}.  Continuity implies boundedness.
Suppose now that \(\mathcal B(\alpha)=\infty\).  Let
\(b\in C^\rho\), \(\rho>1/2\), and choose an arc \(I\) where it is
not identically zero.  Proposition~\ref{en:projected-divergence} supplies
\(\beta_*\in I\) at which the projected Abel sums diverge positively.
The Poisson commutator from Lemma~\ref{li:commutator}, after rescaling
from period \(2\pi\) to period one, gives
\[
 \|\mathsf P_r(bF_\alpha)-b\mathsf P_rF_\alpha\|_\infty\longrightarrow0.
\]
Hence the Poisson means of \(\operatorname{Re}(bF_\alpha)\) at
\(\beta_*\) tend to positive infinity.  If this function were essentially
bounded on \(I\), those means would stay bounded: the contribution from
\(I\) is bounded by its essential supremum, and the separated
\(L^1\) contribution from its complement tends to zero.  This proves the
local assertion.  Taking \(b=1\) also excludes \(F_\alpha\in L^\infty\),
which completes the equivalence.
\end{proof}

\section{Continuity at irrational times}
\label{gm:section}

The continuity criterion is not restricted to a single logarithmic
kernel. An atom of any finite complex measure forces the same global
obstruction, even in the presence of an arbitrary atomless part. The
proof extracts that atom by averaging spatial translates. It is
important that the averaging is stopped after finitely many terms:
the resulting error need only be continuous, with no uniform bound
as the number of terms increases.

We work first on $\T_1=\R/\Z$, fix an irrational $\alpha$, and write
\[
 F_a(\beta)=F_\alpha(\beta-a),\qquad
 T_\alpha\nu=\int_{\T_1}F_a\,d\nu(a)
\]
for a finite complex measure $\nu$. The integral is well defined in
$L^2(\T_1)$, since all translates of $F_\alpha$ have the same
$L^2$ norm.

\begin{proposition}\label{gm:measure}
Suppose $\nu$ has a nonzero atom, $1/2<\rho<1$, and
$b\in C^\rho(\T_1)$ satisfies $|b|=1$. Set
$U=\operatorname{Re}(bT_\alpha\nu)$. Then
\[
 U\in C^0(\T_1)
 \quad\Longleftrightarrow\quad U\in L^\infty(\T_1)
 \quad\Longleftrightarrow\quad \mathcal B(\alpha)<\infty.
\]
\end{proposition}

\begin{corollary}[The BV continuity dichotomy]\label{gm:BO}
Let $u_0$ be real, mean zero, and of bounded variation, with at least
one nonzero jump. At every irrational $\alpha=t/(2\pi)$,
\[
 S(t)u_0\in C^0(\T)
 \quad\Longleftrightarrow\quad S(t)u_0\in L^\infty(\T)
 \quad\Longleftrightarrow\quad
 \sum_j\frac{\log q_{j+1}}{\sqrt{q_j}}<\infty,
\]
where $q_j$ are the convergent denominators of $\alpha$.
There is no restriction on the jump locations or on the atomless
part of $Du_0$.
\end{corollary}

\subsection{A common growth bound and stability under small measures}

The nested construction of Section~\ref{en:section} gives a slightly
stronger estimate than divergence alone. We record it because it
allows us to remove a small measure without losing the obstruction.

\begin{lemma}\label{gm:envelope}
If $\mathcal B(\alpha)=\infty$, there is a nonnegative sequence
$L_N\to\infty$ and an absolute constant $C$ such that
\begin{equation}\label{gm:upper}
 \sup_\beta\left|\sum_{n=1}^N
          \frac{e(\alpha n^2+\beta n)}n\right|
 \leq C L_N+C_\alpha.
\end{equation}
For every $|\zeta|=1$ and every nonempty open arc $I$, there is
$\beta_*\in I$ for which, with an absolute $c_0>0$,
\begin{equation}\label{gm:lower}
 \operatorname{Re}\left(\zeta\sum_{n=1}^N
          \frac{e(\alpha n^2+\beta_*n)}n\right)
 \geq c_0 L_N-C_{\alpha,I}\qquad(N\geq1).
\end{equation}
The sequence $L_N$ is independent of $I$ and $\zeta$.
\end{lemma}
\begin{proof}
Use the selected indices \eqref{en:selected} and the cutoffs $B_j$
from Section~\ref{en:section}, discarding only finitely many initial
indices so that $B_j\geq2q_j$. Define
\[
 L_N=\sum_{\substack{j\ \mathrm{selected}\\q_j\leq N}}
 \frac1{\sqrt{q_j}}
       \log\frac{\min\{N,B_j\}}{q_j}.
\]
The heights \eqref{en:height} show that $L_N\to\infty$.
Summation by parts with denominator $q_j$ bounds each unfinished
lower block $[q_j,M]$, $M\leq B_j$, by
$C(1+\log(M/q_j))/\sqrt{q_j}$. The sum of the extra
$q_j^{-1/2}$ terms converges. The unselected blocks and the upper
complements are uniformly summable by the estimates preceding
\eqref{en:height}. This proves \eqref{gm:upper}.

Apply the nested construction with constant phase $\zeta$, starting
inside $I$. At its limiting point, the offset from the selected
rational grid satisfies $|\theta_j|^{-1}\geq\sqrt{q_jB_j}$.
For every $q_j\leq M\leq B_j$, therefore,
\[
 \log\frac{\min\{M,|\theta_j|^{-1}\}}{q_j}
 \geq\frac12\log(M/q_j).
\]
The Gauss coefficient has real projection at least
$1/(8\sqrt{q_j})$. Lemma~\ref{en:log-block}, including its
summable error, consequently gives \eqref{gm:lower}, for example
with $c_0=1/16$. A later starting index changes only the additive
constant, not the sequence $L_N$.
\end{proof}

\begin{lemma}[Stability of one atomic contribution]\label{gm:robust}
Suppose $\mathcal B(\alpha)=\infty$, $c\ne0$, and $a_0\in\T_1$.
For two finite complex measures $\mu_+,\mu_-$, let
$h_\pm(\beta,a)$ be measurable in $a$ and belong to $C^\rho$
in $\beta$, where $1/2<\rho<1$, with
\[
 \int\|h_\pm(\cdot,a)\|_{C^\rho}\,d|\mu_\pm|(a)<\infty.
\]
There is an absolute $\epsilon_0>0$ such that, if
\begin{equation}\label{gm:small}
 \int\|h_+(\cdot,a)\|_\infty\,d|\mu_+|(a)
 +\int\|h_-(\cdot,a)\|_\infty\,d|\mu_-|(a)
 <\epsilon_0|c|,
\end{equation}
then the real part of
\[
 H(\beta)=cF_{a_0}(\beta)
  +\int h_+(\beta,a)F_a(\beta)\,d\mu_+(a)
  +\int h_-(\beta,a)\overline{F_a(\beta)}\,d\mu_-(a)
\]
is essentially unbounded on every nonempty open arc.
The conclusion is unchanged by adding a continuous function to $H$.
\end{lemma}
\begin{proof}
In a prescribed arc choose $\beta_*$ using
Lemma~\ref{gm:envelope}, translated by $a_0$ and with
$\zeta=c/|c|$. Replace each kernel in $H$ by its first $N$
terms and freeze the coefficients at $\beta_*$. The target term
has real part at least $c_0|c|L_N-C_{\alpha,I}|c|$.
The absolute value of the other terms is at most
$\epsilon_0|c|(CL_N+C_\alpha)$, by \eqref{gm:upper}.
Choose $\epsilon_0<c_0/(2C)$. The real parts of all these
formal partial sums tend to positive infinity. Positive Abel
averaging gives the same conclusion with every kernel replaced
by its Poisson mean.

It remains to relate this expression to the Poisson mean of $H$.
Lemma~\ref{li:commutator}, rescaled to period one, gives
\[
 \|\mathsf P_r(h_\pm F_a)-h_\pm \mathsf P_rF_a\|_\infty
 \leq C_\rho(1-r)^{\rho-1/2}
       \|h_\pm(\cdot,a)\|_{C^\rho}\|F_\alpha\|_2,
\]
with the same estimate for $\overline{F_a}$. Integration makes
the difference tend uniformly to zero. Thus the actual Poisson
means of $\operatorname{Re}H$ diverge at $\beta_*$. Local
essential boundedness on the chosen arc would keep those means
bounded: its complement is separated from $\beta_*$ and contributes
$o(1)$ by $L^1$ integrability. A continuous additional term has
bounded Poisson means and does not change the conclusion.
\end{proof}

\subsection{Spatial translates and extraction of an atom}

For $k\geq0$, reindexing the series gives
\begin{equation}\label{gm:shift}
 F_a(\beta+2k\alpha)
 =e(ka-k\beta-\alpha k^2)F_a(\beta)+R_{k,a}(\beta),
\end{equation}
where $R_{k,a}$ is jointly continuous in $(\beta,a)$. Indeed,
after taking out the phase, the remainder is
\[
 -\sum_{m=1}^k\frac{e(\alpha m^2+m(\beta-a))}{m}
 +\sum_{m\geq k+1}\frac{k\,e(\alpha m^2+m(\beta-a))}{m(m-k)}.
\]
The second series converges absolutely and uniformly for each fixed
$k$, and $\sup_a\|R_{k,a}\|_\infty\leq2\sum_{m=1}^k1/m$.
The identity is first understood in $L^2$.

\begin{proof}[Proof of Proposition~\ref{gm:measure}]
Sufficiency follows from the uniform convergence of $F_\alpha$.
Suppose instead that $\mathcal B(\alpha)=\infty$ and
$U\in L^\infty$. Choose an atom $c=\nu(\{a_0\})\ne0$.
Set
\[
 A_k(\beta)=b(\beta+2k\alpha)e(-k\beta-\alpha k^2).
\]
Integrating \eqref{gm:shift} shows that
\[
 U(\beta+2k\alpha)
 =\operatorname{Re}\left(A_k(\beta)
             \int e(ka)F_a(\beta)\,d\nu(a)\right)+r_k(\beta),
 \qquad r_k\in C^0.
\]
Since $|A_k|=1$, the bounded function
\[
 H_K(\beta)=\frac2K\sum_{k=0}^{K-1}
          \overline{A_k(\beta)}e(-ka_0)U(\beta+2k\alpha)
\]
has the representation
\begin{equation}\label{gm:extract}
 H_K=\int D_K(a-a_0)F_a\,d\nu(a)
     +\int E_K(\beta,a+a_0)\overline{F_a}\,d\overline\nu(a)
     +C_K,
 \qquad C_K\in C^0,
\end{equation}
where
\[
 D_K(s)=\frac1K\sum_{k<K}e(ks),\qquad
 E_K(\beta,s)=\frac1K\sum_{k<K}
             \overline{A_k(\beta)}^2e(-ks).
\]
Here $\overline\nu$ is the conjugate measure. The first measure in
\eqref{gm:extract} is $c\delta_{a_0}+\eta_K$, with
\begin{equation}\label{gm:measuretail}
 \|\eta_K\|_{\mathrm{TV}}\longrightarrow0.
\end{equation}
Indeed, $D_K(0)=1$, $|D_K|\leq1$, and $D_K(s)\to0$ for
$s\notin\Z$. Dominated convergence against
$|\nu-c\delta_{a_0}|$ proves the assertion.

We also have
\begin{equation}\label{gm:wrongsign}
 \sup_{\beta,s}|E_K(\beta,s)|\longrightarrow0.
\end{equation}
To verify this, approximate $\overline b^2$ uniformly by a
trigonometric polynomial. Each term in the corresponding average
is, up to a factor of modulus one, a normalized quadratic sum with
irrational quadratic coefficient $2\alpha$ and arbitrary linear
coefficient $2\beta-s+2m\alpha$. These sums tend to zero uniformly
in their linear coefficient. For example, apply
Lemma~\ref{ar:quadratic} at a fixed convergent denominator of
$2\alpha$, first let $K\to\infty$, and then let that denominator
tend to infinity. The uniform approximation error is independent
of $K$, proving \eqref{gm:wrongsign}.

Choose and fix $K$ so large that
\[
 \|\eta_K\|_{\mathrm{TV}}
 +\|\nu\|_{\mathrm{TV}}\sup_{\beta,s}|E_K(\beta,s)|
 <\epsilon_0|c|.
\]
For this fixed $K$, the $C^\rho$ norms of
$E_K(\cdot,s)$ are bounded uniformly in $s$. Formula
\eqref{gm:extract} therefore satisfies all the assumptions of
Lemma~\ref{gm:robust}, with $h_+=1$, $\mu_+=\eta_K$,
$h_-=E_K(\beta,a+a_0)$, and $\mu_-=\overline\nu$.
Its real part cannot be bounded. This contradicts the boundedness
of $H_K$ and proves necessity.
\end{proof}

\begin{proof}[Proof of Corollary~\ref{gm:BO}]
Let $V_0=\partial_x^{-1}u_0$ and
$\nu=-(2\pi)^{-1}D(u_0e^{-iV_0})$, pushed forward to
$\T_1$ under $x\mapsto x/(2\pi)$. A jump of $u_0$ at $a$
gives the nonzero atom
$-[u_0]_a e^{-iV_0(a)}/(2\pi)$. The measure representation
and the reconstruction \eqref{eq:GKT} give, with $\beta=x/(2\pi)$,
\[
 u(t,x)=2\operatorname{Re}\left(
       i e^{iV(t,x)-it\|u_0\|_2^2}T_\alpha\nu(\beta)\right)
       +r_t(x),\qquad r_t\in C^0.
\]
The displayed multiplier is unimodular and belongs to $C^\rho$
for every $\rho<1$. Proposition~\ref{gm:measure} proves the stated
equivalence. A continuous remainder cannot change essential
boundedness.
\end{proof}

\begin{corollary}[Integer frequency compression]\label{gm:compression}
For every irrational $\alpha$ and every positive integer $m$,
\[
 \mathcal B(\alpha)<\infty
 \quad\Longleftrightarrow\quad \mathcal B(m^2\alpha)<\infty.
\]
In particular, the parameters $t/(2\pi)$ and $2t/\pi$ used
for the general solution and the compressed square-wave profile
give the same continuity criterion.
\end{corollary}
\begin{proof}
For $\nu=\sum_{\ell=0}^{m-1}\delta_{\ell/m}$, Fourier
orthogonality gives
\[
 T_\alpha\nu(\beta)
 =\sum_{\ell=0}^{m-1}F_\alpha(\beta-\ell/m)
 =F_{m^2\alpha}(m\beta).
\]
Proposition~\ref{gm:measure}, with $b=1$, characterizes boundedness
of the real part on the left by $\mathcal B(\alpha)<\infty$.
Proposition~\ref{en:classification} characterizes boundedness of
the real part on the right by $\mathcal B(m^2\alpha)<\infty$.
\end{proof}

\subsection{Local continuity and rotation classes}
\label{fcl:section}

A finite rotation class of atoms is enough to localize the endpoint
obstruction.  The measure may have infinitely many other atoms and an
arbitrary atomless part.  The proof uses observations returning to a
fixed arc and inverts a finite Gram matrix on the chosen class.

\begin{proposition}\label{fcl:local}
Let \(\alpha\) be irrational, and let \(\nu\) be a finite complex
measure on \(\T_1\).  Suppose one of its atoms belongs to a finite
class of atoms under
\[
 a\sim a'\quad\Longleftrightarrow\quad
 a-a'\in2\alpha\Z\pmod1.
\]
Let \(b\in C^\rho(\T_1)\), where \(1/2<\rho<1\), satisfy
\(|b|=1\), and define in \(L^2\)
\[
 U(\beta)=\operatorname{Re}\left[
 b(\beta)\int_{\T_1}F_\alpha(\beta-a)\,d\nu(a)\right].
\]
Then, for every nonempty open arc \(I\subset\T_1\),
\[
 U\in L^\infty(I)
 \quad\Longleftrightarrow\quad
 \mathcal B(\alpha)<\infty.
\]
When the arithmetic sum is finite, \(U\) has a continuous
representative on the whole circle.
\end{proposition}

\begin{proof}
Sufficiency follows from uniform convergence of the quadratic profile.
For necessity, assume that \(\mathcal B(\alpha)=\infty\), and choose
an atom in a finite rotation class.  List all atoms in that class as
\(a_1,\ldots,a_J\), with nonzero masses \(c_1,\ldots,c_J\).
Put \(s=2\alpha\), and write
\(a_j=a_*+sm_j\pmod1\), where the integers \(m_j\) are distinct.
Choose a nonnegative smooth function \(\chi\), supported in \(I\)
and positive on some open interval.

For \(1\le l,m\le J\), define
\[
 G_{lm}(\beta)=\int_{\T_1}\chi(y)
                   e\bigl((m_m-m_l)(y-\beta)\bigr)\,dy.
\]
The quadratic form associated with this Hermitian matrix is
\[
 z^*G(\beta)z=\int_{\T_1}\chi(y)
          \left|\sum_{m=1}^Jz_me(m_m(y-\beta))\right|^2\,dy.
\]
A nonzero trigonometric polynomial cannot vanish on an interval, so
\(G(\beta)\) is positive definite.  Its dependence on \(\beta\)
is diagonal unitary conjugation.  In particular, \(G^{-1}\) is smooth
and uniformly bounded on the circle.

Let
\[
 A_k(\beta)=b(\beta+sk)e(-\alpha k^2-k\beta).
\]
The translation identity \eqref{gm:shift}, integrated
against \(\nu\), gives
\[
 U(\beta+sk)=\operatorname{Re}\left[
 A_k(\beta)\int e(ka)F_\alpha(\beta-a)\,d\nu(a)\right]
       +r_k(\beta),
 \qquad r_k\in C^0(\T_1).
\]
Continuity of the remainder follows from joint continuity in \((\beta,a)\)
and the bound \(\|R_{k,a}\|_\infty\le2\sum_{n=1}^kn^{-1}\).
Fix one index \(j\), and form
\[
 H_{K,j}(\beta)=\frac2K\sum_{k=0}^{K-1}
 \chi(\beta+sk)\overline{A_k(\beta)}
 \left(\sum_{l=1}^J(G(\beta)^{-1})_{jl}e(-ka_l)\right)
 U(\beta+sk).
\]
If \(U\) is essentially bounded on \(I\), this function is globally
essentially bounded for every fixed \(K\).  Only values in \(I\)
appear with nonzero weight.

Expanding the real part in the preceding translation identity yields
\begin{equation}\label{fcl:extraction}
 H_{K,j}(\beta)=\int\Phi_{K,j}(\beta,a)F_\alpha(\beta-a)\,d\nu(a)
  +\int\Psi_{K,j}(\beta,a)\overline{F_\alpha(\beta-a)}\,
                   d\overline\nu(a)+C_{K,j}(\beta),
\end{equation}
where \(C_{K,j}\) is continuous and
\begin{align*}
 \Phi_{K,j}(\beta,a)
 &=\sum_{l=1}^J(G(\beta)^{-1})_{jl}
       \frac1K\sum_{k=0}^{K-1}\chi(\beta+sk)e(k(a-a_l)),\\
 \Psi_{K,j}(\beta,a)
 &=\sum_{l=1}^J(G(\beta)^{-1})_{jl}
       \frac1K\sum_{k=0}^{K-1}
       \chi(\beta+sk)\overline{A_k(\beta)}^{\,2}e(-k(a+a_l)).
\end{align*}

Uniform approximation of \(\chi\) by trigonometric polynomials and
finite geometric summation show that
\[
 \frac1K\sum_{k=0}^{K-1}\chi(\beta+sk)e(k(a_m-a_l))
 \longrightarrow G_{lm}(\beta)
\]
uniformly in \(\beta\).  Therefore
\(\Phi_{K,j}(\beta,a_m)\to\delta_{jm}\) uniformly.
For a point \(a\notin a_*+s\Z\pmod1\), the same argument shows
\(\Phi_{K,j}(\beta,a)\to0\), again uniformly in \(\beta\):
none of the finitely many frequencies in a trigonometric approximation
has a nonzero geometric mean.  The remaining points of the countable
orbit carry no \(\nu\)-mass, by the choice of the complete atomic
class.  The kernels \(\Phi_{K,j}\) are bounded uniformly in all their
variables and \(K\).  Dominated convergence thus gives
\begin{equation}\label{fcl:small-correct}
 \int\sup_\beta\left|
 \Phi_{K,j}(\beta,a)-\mathbf1_{\{a_j\}}(a)\right|
                    \,d|\nu|(a)\longrightarrow0.
\end{equation}

For the second kernel, each inner sum contains the quadratic factor
\(e(2\alpha k^2)\).  Uniform trigonometric approximation of
\(\chi\overline b^{\,2}\), followed by the quadratic Weyl estimate
uniform in its linear parameter, gives
\begin{equation}\label{fcl:small-wrong}
 \sup_{\beta,a}|\Psi_{K,j}(\beta,a)|\longrightarrow0.
\end{equation}
The uniformly bounded inverse Gram matrix does not change this limit.
For each fixed \(K\), the \(C^\rho\) norms in \(\beta\) of both
kernels are bounded uniformly in \(a\).

Equations \eqref{fcl:small-correct}--\eqref{fcl:small-wrong} express
\eqref{fcl:extraction}, after removal of its continuous remainder,
as the nonzero target \(c_jF_\alpha(\beta-a_j)\) plus positive and
negative frequency kernel terms of arbitrarily small total coefficient
variation.  Fix \(K\) large enough for the robust atom obstruction
of Lemma~\ref{gm:robust} to apply.  That lemma implies that
\(H_{K,j}-C_{K,j}\) cannot be globally essentially bounded.
For completeness, its application here freezes the kernel coefficients
at the nested spatial point, uses the common partial-sum growth envelope,
and then commutes the Poisson mean past these fixed-\(K\)
\(C^\rho\) coefficients.  The uniform bounds in \(a\) justify
integration of the commutator estimate against \(|\nu|\).
This contradicts the boundedness of \(H_{K,j}\) resulting from
\(U\in L^\infty(I)\), and proves necessity.
\end{proof}

\begin{corollary}[Finite jump sets with arbitrary atomless remainder]
\label{fcl:finite-BV}
Let \(u_0\in\operatorname{BV}(\T)\) be real and mean zero, and
suppose it has a finite nonempty set of genuine jumps.  No restriction
is imposed on the atomless part of \(Du_0\).  At every irrational
\(\alpha=t/(2\pi)\), and for every nonempty open arc \(I\),
\[
 u(t)\in L^\infty(I)
 \quad\Longleftrightarrow\quad
 u(t)\in C^0(\T)
 \quad\Longleftrightarrow\quad
 \mathcal B(\alpha)<\infty.
\]
More generally, the same conclusion holds whenever, at the time in
question, one genuine jump belongs to a finite class under
\(a-a'\in2t\Z\pmod{2\pi}\), even if there are infinitely many
other jumps.
\end{corollary}
\begin{proof}
Write $V(t)=\partial_x^{-1}u(t)$, $V_0=\partial_x^{-1}u_0$,
and $M=\|u_0\|_2^2$.  The gauge derivative measure has precisely
the genuine jump locations as its atoms, with masses
\(-[u_0]_a e^{-iV_0(a)}/(2\pi)\).  In the normalized variables
\(\beta=x/(2\pi)\) and \(\alpha=t/(2\pi)\), the reconstruction
reads $u(t,2\pi\beta)=2U(\beta)+r_t(2\pi\beta)$, where
$r_t\in C^0$ and $U$ has the form in Proposition~\ref{fcl:local},
with unimodular multiplier \(i e^{iV(t,2\pi\beta)-itM}\).
This multiplier lies in \(C^\rho\) for every \(\rho<1\).
Apply the theorem to a finite class of atoms.
\end{proof}

\begin{corollary}\label{fcl:countable}
For a fixed real mean-zero BV datum with a nonzero jump, the local
continuity dichotomy holds on every open arc at every irrational
time outside the countable set $\mathscr R(u_0)$ of
Corollary~\ref{lbv:countable}. Thus any exceptions to the local
endpoint classification are confined to that explicit resonance set.
\end{corollary}
\begin{proof}
Outside $\mathscr R(u_0)$ every class of genuine jumps is a
singleton. Apply Corollary~\ref{fcl:finite-BV}.
\end{proof}

\subsection{An orbit with one dominant atom}

The local obstruction also survives infinitely many atoms in the
relevant rotation class, provided one of them dominates the others.
The required dominance factor is absolute.

\begin{lemma}[A dominant atom gives local unboundedness]
\label{gm:localdominant}
Use the notation of Proposition~\ref{gm:measure}, and let
$\epsilon_0$ be the constant in Lemma~\ref{gm:robust}.
Set $K_*=4/\epsilon_0$, increasing it to at least one if necessary.
Suppose that $c=\nu(\{a_0\})\ne0$ and
\[
 |c|>K_*\,|\nu|\bigl((a_0+2\alpha\Z)\setminus\{a_0\}\bigr),
\]
where the orbit is taken modulo one.  If
$\mathcal B(\alpha)=\infty$, then
$U=\operatorname{Re}(bT_\alpha\nu)$ is essentially unbounded
on every nonempty open arc.  The same conclusion holds after
adding a continuous function to $U$.
\end{lemma}

\begin{proof}
Suppose that $U$ is essentially bounded on an open arc $I$.
Choose a nonzero nonnegative $w\in C_c^\infty(I)$ and write
$w_0=\int_{\T_1}w>0$.  With $A_k$ as in the proof of
Proposition~\ref{gm:measure}, the functions
\[
 H_L(\beta)=\frac2L\sum_{k=0}^{L-1}
 w(\beta+2k\alpha)\overline{A_k(\beta)}e(-ka_0)
 U(\beta+2k\alpha)
\]
are globally bounded for every fixed $L$.  The shift identity
\eqref{gm:shift} gives
\[
 H_L=\int D_L^w(\beta,a-a_0)F_a\,d\nu(a)
 +\int E_L^w(\beta,a+a_0)\overline{F_a}\,d\overline\nu(a)
 +C_L,
 \qquad C_L\in C^0,
\]
where
\begin{align*}
 D_L^w(\beta,s)&=\frac1L\sum_{k<L}w(\beta+2k\alpha)e(ks),\\
 E_L^w(\beta,s)&=\frac1L\sum_{k<L}w(\beta+2k\alpha)
                    \overline{A_k(\beta)}^{2}e(-ks).
\end{align*}

Writing $\widehat w(m)=\int_{\T_1}w(y)e(-my)\,dy$, the Fourier
series of $w$ is absolutely convergent, so
\[
 D_L^w(\beta,s)=\sum_{m\in\Z}\widehat w(m)e(m\beta)
                    D_L(s+2m\alpha).
\]
For fixed $s$, this converges uniformly in $\beta$ to
\[
 D_\infty^w(\beta,s)=
 \begin{cases}
  \widehat w(m)e(m\beta),&s+2m\alpha\in\Z
                              \text{ for some }m\in\Z,\\
  0,&\text{otherwise}.
 \end{cases}
\]
There is at most one resonant integer $m$.  Uniform convergence
for fixed $s$ follows by first truncating the absolutely convergent
Fourier series and then letting $L\to\infty$ in the finite sum.
The norms of these kernels are uniformly bounded by a constant
depending only on $w$.  Dominated convergence against $|\nu|$
therefore gives the stronger integrated statement
\begin{equation}\label{gm:weightedTV}
 \int\sup_\beta
 |D_L^w(\beta,a-a_0)-D_\infty^w(\beta,a-a_0)|\,d|\nu|(a)
 \longrightarrow0.
\end{equation}
In particular, the entire atomless part makes a vanishing
contribution to this limit, since the resonant orbit is countable.

At $a=a_0$ the limiting coefficient is $w_0$.  At every other
point of its orbit, its absolute value is at most $w_0$, because
$w\geq0$ implies $|\widehat w(m)|\leq w_0$.  Outside that orbit
it vanishes.  If
$\eta=|\nu|((a_0+2\alpha\Z)\setminus\{a_0\})$, it follows
from \eqref{gm:weightedTV} that
\[
 \int\sup_\beta
 \bigl|D_L^w(\beta,a-a_0)-w_0\mathbf1_{\{a_0\}}(a)\bigr|
       \,d|\nu|(a)
 \leq w_0\eta+o(1).
\]
The dominance hypothesis makes the first term less than
$\epsilon_0w_0|c|/4$.

Furthermore,
$\sup_{\beta,s}|E_L^w(\beta,s)|\to0$.
Indeed, approximate $w\overline b^{2}$ uniformly by a
trigonometric polynomial.  Each resulting term is a normalized
quadratic sum with irrational quadratic coefficient $2\alpha$,
and the convergence is uniform in the linear coefficient, exactly
as in \eqref{gm:wrongsign}.

Choose and fix $L$ large enough that the integrated correct-frequency
error and the wrong-frequency integral together have coefficient
norm less than $\epsilon_0w_0|c|$.  After removing the continuous
term $C_L$, the preceding representation has target atom
$w_0cF_{a_0}$.  All its remaining coefficient functions have
$C^\rho$ norms bounded uniformly in the integrated location,
because $L$ is fixed.  Thus it satisfies every hypothesis of
Lemma~\ref{gm:robust}.  Its real part cannot be bounded, a
contradiction.  The same argument applies to each observation arc.
A continuous added function is bounded and does not affect the
contradiction.
\end{proof}

\begin{proposition}[An infinite rotation class can cancel on an arc]
\label{fcl:counterexample}
For every irrational \(\alpha\) and every proper open arc
\(I\subset\T_1\) whose complement contains an open interval,
there is a nonzero finite complex atomic measure \(\nu\), supported
on a single infinite \(2\alpha\)-rotation class, such that
\[
 \int F_\alpha(\beta-a)\,d\nu(a)
\]
has a continuous representative on \(I\).  This remains true when
\(\mathcal B(\alpha)=\infty\).
\end{proposition}
\begin{proof}
Choose a nonzero smooth function \(P\) vanishing on \(I\), and write
\(P(\beta)=\sum_{m\in\Z}p_me(m\beta)\).  Fix \(a_0\in\T_1\)
and define
\[
 \nu=\sum_{m\in\Z}p_m e(\alpha m^2+ma_0)
                        \delta_{a_0+2\alpha m}.
\]
The support points are distinct.  The coefficients are absolutely
summable, and infinitely many are nonzero: otherwise \(P\) would be
a trigonometric polynomial vanishing on an interval.

The translation identity extends to all integer shifts and gives
\[
 F_\alpha(\beta-a_0-2\alpha m)
 =e(-\alpha m^2+m\beta-ma_0)F_\alpha(\beta-a_0)
       +R_m(\beta),
 \qquad \|R_m\|_\infty\le2\sum_{n=1}^{|m|}\frac1n.
\]
For negative shifts, reindexing produces finitely many nonpositive
indices and the coefficient difference
\(-|m|/[n(n+|m|)]\) at positive indices; the same harmonic-sum
bound follows.  Rapid decay of \(p_m\) therefore implies
\[
 \int F_\alpha(\beta-a)\,d\nu(a)
 =P(\beta)F_\alpha(\beta-a_0)+R(\beta),
 \qquad R\in C^0(\T_1).
\]
The first term vanishes on \(I\), proving the assertion.
\end{proof}

\begin{remark}
This example concerns complex measures.  It does not by itself provide
real Benjamin--Ono initial data: the atomic gauge masses of a real BV
function have phases constrained by its primitive.  It shows why a local
statement for arbitrary complex measures needs a condition on the
rotation classes, even though the global endpoint obstruction only
requires one nonzero atom.
\end{remark}

\section{Sharp H\"older and Besov regularity}\label{fe:section}

The sufficient estimates follow by integrating the quadratic kernel
against the initial gauge measure. We prove them first, then obtain
matching lower bounds from the atomic part of that measure.

\subsection{Upper bounds for BV data}
\begin{proposition}[Arithmetic regularity for nonlinear BO]
\label{ar:nonlinear}
Let $u_0\in\operatorname{BV}(\T)$ be real and have zero mean, and
let $u(t)$ be its periodic BO solution.  Put $\alpha=t/(2\pi)$.
If $\alpha$ is irrational and $\mathcal B(\alpha)<\infty$, then
$u(t)$ has a continuous representative.  If
\eqref{ar:finite-type} holds, then
\[
 u(t)\in C^{1/(\kappa+1)}(\T).
\]
There is a set of times of full Lebesgue measure, independent of
$u_0$, on which $u(t)\in C^\gamma(\T)$ for every $\gamma<1/2$.
\end{proposition}

\begin{proof}
Set $V_0=\partial_x^{-1}u_0$, $f_0=u_0e^{-iV_0}$, and
$\mu=Df_0$.  Since $u_0$ is bounded and $V_0$ is Lipschitz,
$f_0$ is of bounded variation and $\mu$ is a finite complex measure.
The identity $(e^{-iV_0})'=-if_0$ gives, for $n\geq 1$,
\[
 \widehat{\mathcal G(u_0)}(n)
 =-i\widehat f_0(n)
 =-\frac{1}{2\pi n}\int_\T e^{-iny}\,d\mu(y).
\]
The measure integral here is unnormalized.  Writing
$M=(2\pi)^{-1}\int_\T u_0^2$, the linear gauge profile therefore
has the representation
\begin{equation}\label{ar:measuregauge}
 w_L(t,x)=-\frac{e^{-itM}}{2\pi}
 \int_\T F_\alpha\!\left(\frac{x-y}{2\pi}\right)\,d\mu(y)
\end{equation}
whenever the series for $F_\alpha$ converges uniformly.  Indeed,
one may integrate its partial sums term by term, and the resulting
Fourier coefficients are precisely those of $w_L$.
Integration against a finite measure preserves continuity and every
H\"older exponent in $(0,1)$.

Also $u_0\in H^s(\T)$ for every $s<1/2$, as follows from
$|\widehat u_0(n)|\leq C/|n|$.  The BO flow preserves these spaces
\cite[Corollary~A.4]{GKT-wp}.  The smoothing theorem
\cite[Theorem~1.1]{GKT-smoothing} consequently yields
\[
 u(t)=2\operatorname{Re}\bigl(i a_t w_L(t)\bigr)+r_t,
 \qquad a_t=e^{i\partial_x^{-1}u(t)},
\]
where $a_t,r_t\in C^\rho(\T)$ for every $\rho<1$.
For $a_t$ this follows from $u(t)\in H^{1/2-}$ and Sobolev
embedding; for $r_t$ it follows by applying the smoothing theorem
with $s$ arbitrarily close to $1/2$, since $r_t\in H^{3s}$.
The first two conclusions now follow from
Propositions~\ref{ar:uniform} and~\ref{ar:holderkernel}.

For each fixed $\kappa>1$, the sets
$\{\alpha\in[0,1]:\|m\alpha\|<m^{-\kappa}\}$ have summable
measures.  Borel--Cantelli, with a change in the constant to cover
finitely many exceptional $m$, gives \eqref{ar:finite-type} for
almost every $\alpha$.  Intersecting these full-measure sets over
a sequence $\kappa\downarrow1$ proves the last assertion.
\end{proof}

\begin{remark}
The almost-everywhere H\"older conclusion has a close antecedent in
\cite[Lemma~4.4]{GKT-smoothing}, which treats initial data in
$B^1_{1,1}(\T)$.  Formula~\eqref{ar:measuregauge} extends that
argument to all BV data, including functions with jumps, and makes
the arithmetic sufficient conditions explicit.  In particular,
Remark~\ref{ar:liouvillecontinuous} supplies Liouville times at
which every such nonlinear solution is continuous.
\end{remark}

\subsection{The necessary H\"older bound}

We now prove the positive regularity part of Theorem~\ref{main:regularity}.
A single genuine jump supplies the required lower bound.  Neither
the locations of the jumps nor the continuous part of the derivative
measure need satisfy any arithmetic condition.

\begin{proposition}\label{fe:classification}
Let $u_0\in\operatorname{BV}(\T)$ be real and have zero mean.
Assume that $u_0$ has at least one nonzero jump, and let $u(t)$ be
the periodic BO solution.  Fix an irrational
$\alpha=t/(2\pi)$, with continued-fraction denominators $q_j$.
For every $0<\gamma\leq1/2$,
\begin{equation}\label{fe:iff}
 u(t)\in C^\gamma(\T)
 \quad\Longleftrightarrow\quad
 \sup_j\frac{q_{j+1}}{q_j^{1/\gamma-1}}<\infty.
\end{equation}
In particular, $u(t)\in C^{1/2}$ if and only if $\alpha$ has
bounded partial quotients.  If
\[
 \kappa(\alpha)=\limsup_{j\to\infty}
 \frac{\log q_{j+1}}{\log q_j},
\]
then the supremum of its positive H\"older exponents, with the
value zero assigned to an empty set, is
$1/(1+\kappa(\alpha))$.  This value is zero when
$\kappa(\alpha)=\infty$.
\end{proposition}

The proof combines frequency localization with the atomic part of a
finite measure.  It applies, in particular, to arbitrary finite-jump
data without requiring rational jump locations or a lower bound on
the separation of their rational translates.  It also permits
infinitely many jumps and a singular continuous derivative measure.

Fix a nonzero, nonnegative $\varphi\in C_c^\infty((1,2))$, and
let $P_N$ have Fourier symbol $\varphi(|n|/N)$.  Its kernel satisfies
\[
 |K_N(x)|\leq C_K N(1+N\operatorname{dist}(x,2\pi\Z))^{-K}
\]
for every $K$, and has zero mean.  Consequently,
\begin{equation}\label{fe:holderblock}
 \|P_Nf\|_\infty\leq C_\gamma N^{-\gamma}\|f\|_{C^\gamma},
 \qquad 0<\gamma<1.
\end{equation}

\begin{lemma}[Finite atomic packets]\label{fe:atomicpackets}
Let $a_1,\ldots,a_J$ be distinct points of $\T$, and let
$c_1,\ldots,c_J\in\C$ not all vanish.  Define
\[
 W_N(x)=\sum_{n\geq1}\frac{\varphi(n/N)}n
 e(\alpha n^2)\left(\sum_{j=1}^Jc_je^{-ina_j}\right)e^{inx},
 \qquad e(z)=e^{2\pi iz}.
\]
Then
\begin{equation}\label{fe:packetmass}
 \|W_N\|_2^2
 =\frac{I_\varphi}{N}\sum_{j=1}^J|c_j|^2+o(N^{-1}),
 \qquad
 I_\varphi=\int_0^\infty\frac{\varphi(s)^2}{s^2}\,ds>0.
\end{equation}
Suppose that $(p,q)=1$, $N\geq q$, and
$|\alpha-p/q|N^2\leq1$.  For $A\geq1$, let
\[
 \Omega=\bigcup_{j=1}^J\bigcup_{\ell=0}^{q-1}
 \left\{x:\operatorname{dist}
 \left(x,a_j+\frac{2\pi\ell}{q}\right)\leq\frac A N\right\}.
\]
Then $|\Omega|\leq 2JqA/N$, and, for every fixed $K>1$,
\begin{equation}\label{fe:concentration}
 \|\mathbf1_{\Omega^c}W_N\|_2
 \leq C_{\varphi,K}\left(\sum_j|c_j|\right)
 (1+A)^{1/2-K}N^{-1/2}.
\end{equation}
Here and below the $L^2$ norms use normalized measure; $|\Omega|$
denotes ordinary length.
\end{lemma}

\begin{proof}
Parseval's identity removes the unimodular quadratic phase.
The diagonal terms in the squared exponential polynomial give
\eqref{fe:packetmass} by a Riemann sum.  Each off-diagonal term is
\[
 \sum_n\frac{\varphi(n/N)^2}{n^2}e^{-in(a_j-a_k)},
 \qquad j\ne k,
\]
which is $o(N^{-1})$, and in fact decays faster than any fixed
power of $N$ after the natural $N^{-1}$ factor, by Poisson
summation.  Distinctness of $a_j$ and $a_k$ is the only spatial
condition used here.

Put $\delta=\alpha-p/q$.  The rational quadratic phase has the
finite expansion
\[
 e(pn^2/q)=\sum_{\ell=0}^{q-1}d_\ell e(-\ell n/q),
 \qquad
 d_\ell=\frac1q\sum_{m=0}^{q-1}e((pm^2+\ell m)/q).
\]
The elementary complete Gauss-sum estimate gives
$|d_\ell|\leq\sqrt{2/q}$.  Let
\[
 H_{N,\delta}(y)=\sum_{n\geq1}
 \frac{\varphi(n/N)}n e(\delta n^2)e^{iny}.
\]
Since $|\delta|N^2\leq1$, Poisson summation gives
\[
 |H_{N,\delta}(y)|\leq C_K
 (1+N\operatorname{dist}(y,2\pi\Z))^{-K},
\]
uniformly in $\delta$.  Thus the contribution from the atom at $a_j$
is a sum of $q$ such packets, with coefficients bounded by
$\sqrt{2/q}|c_j|$.

For one grid, write $d_q(x)$ for its distance from $x$.  Because
$N\geq q$, summing the decay estimate over that grid gives
\[
 \sum_{\ell=0}^{q-1}
 (1+N\operatorname{dist}(x,2\pi\ell/q))^{-K}
 \leq C_K(1+Nd_q(x))^{-K}.
\]
Squaring and integrating outside the $A/N$ neighborhoods of the
grid points bounds the squared $L^2$ norm by
$C_K|c_j|^2N^{-1}(1+A)^{1-2K}$.  If those neighborhoods cover
the circle, the same assertion holds trivially.  The triangle
inequality over the finitely many atoms proves
\eqref{fe:concentration}.  Overlap or near overlap between
different grids causes no difficulty.
\end{proof}

\begin{lemma}[Taking a real part preserves packet mass]
\label{fe:realmass}
Suppose $|a|=1$ and $a\in C^\rho(\T)$ for some $0<\rho<1$.
For $W_N$ as in Lemma~\ref{fe:atomicpackets},
\[
 \|2\operatorname{Re}(iaW_N)\|_2^2
 =2\|W_N\|_2^2+O_a(N^{-\rho}\|W_N\|_2^2).
\]
\end{lemma}

\begin{proof}
Expanding the square gives
\[
 \|2\operatorname{Re}(iaW_N)\|_2^2
 =2\|W_N\|_2^2-2\operatorname{Re}\int_\T a^2W_N^2\,
 \frac{dx}{2\pi}.
\]
The Fourier support of $W_N^2$ lies in $(2N,4N)$.  In the last
integral one may therefore replace $a^2$ by a smooth Fourier
cutoff to negative frequencies of size comparable to $N$,
equal to one on $[-4N,-2N]$.  Its sup norm is at most
$C_aN^{-\rho}$, by the zero-mean kernel estimate for a
H\"older function.  The last integral has absolute value at
most $C_aN^{-\rho}\|W_N\|_2^2$.
\end{proof}

\begin{lemma}[Atomless measures have negligible packet mass]
\label{fe:wiener}
Let $\nu$ be a finite complex measure on $\T$, with no atoms,
and use the unnormalized Fourier coefficients
$\widehat\nu(n)=\int_\T e^{-iny}\,d\nu(y)$.  Then
\[
 \frac1N\sum_{n=1}^N|\widehat\nu(n)|^2\longrightarrow0.
\]
Consequently, uniformly in $\alpha$,
\[
 \left\|\sum_{n\geq1}\frac{\varphi(n/N)}n
 e(\alpha n^2)\widehat\nu(n)e^{inx}\right\|_2
 =o(N^{-1/2}).
\]
For an arbitrary finite measure, the same norm is bounded by
$C_\varphi\|\nu\|_{\mathrm{TV}}N^{-1/2}$.
\end{lemma}

\begin{proof}
The averaged squared coefficients equal the integral, against
$d\nu(y)d\overline\nu(z)$, of
\[
 \frac1N\sum_{n=1}^N e^{-in(y-z)}.
\]
This expression has modulus at most one and tends to zero off
the diagonal.  Since the total variation measure has no atoms,
the diagonal has zero product measure.  Dominated convergence
proves the first assertion.  The second follows from Parseval
and $n\asymp N$ on the support of $\varphi(n/N)$; the last uses
$|\widehat\nu(n)|\leq\|\nu\|_{\mathrm{TV}}$.
\end{proof}

\begin{proposition}[Packet lower bound for a BV jump]
\label{fe:lower}
Under the assumptions of Proposition~\ref{fe:classification}, fix
$t$ and let $p_j/q_j$ be the convergents of
$\alpha=t/(2\pi)$.  With $N_j=\sqrt{q_jq_{j+1}}$, there is a
constant $c_t>0$ such that, for every sufficiently large $j$,
\[
 \|P_{N_j}u(t)\|_\infty\geq c_tq_j^{-1/2}.
\]
\end{proposition}

\begin{proof}
Set $V_0=\partial_x^{-1}u_0$, $f_0=u_0e^{-iV_0}$, and
\[
 \nu=-\frac1{2\pi}Df_0.
\]
Then $\nu$ is a finite complex measure and
$\widehat{\mathcal G(u_0)}(n)=\widehat\nu(n)/n$ for $n\geq1$.
Since $V_0$ is continuous, an atom at $a$ has coefficient
$-[u_0]_a e^{-iV_0(a)}/(2\pi)$.  At least one atom, denoted by
$c_*\delta_{a_*}$, is nonzero.

Choose a finite atomic measure
$\nu_F=\sum_{k=1}^Jc_k\delta_{a_k}$ containing that atom so
that
\[
 \nu=\nu_F+\nu_T+\nu_C,
 \qquad \|\nu_T\|_{\mathrm{TV}}<\varepsilon,
\]
where $\nu_T$ is the remaining atomic tail and $\nu_C$ has no
atoms.  Such a choice is possible because the atomic masses of
a finite measure are absolutely summable.  The corresponding
gauge packet can be written
\[
 P_Nw_L(t)=e^{-itM}(W_N+E_N),\qquad
 \|E_N\|_2\leq C\varepsilon N^{-1/2}+o(N^{-1/2}),
\]
where $M=(2\pi)^{-1}\int u_0^2$, $W_N$ is the finite atomic
packet, and Lemma~\ref{fe:wiener} controls the error.  Its
leading mass coefficient satisfies
$\sum_{k=1}^J|c_k|^2\geq|c_*|^2>0$.

The smoothing representation reads
\[
 u(t)=2\operatorname{Re}(ia_tw_L(t))+r_t,
 \qquad a_t=e^{i\partial_x^{-1}u(t)}.
\]
Since $u_0\in\operatorname{BV}\subset H^{1/2-}$, one may take
$a_t\in C^\rho$ with some $\rho>1/2$ and
$r_t\in H^s$ with some $s>1/2$.  The elementary commutator
estimate
\[
 \|[P_N,a_t]w_L(t)\|_2
 \leq C_tN^{-\rho}\|w_L(t)\|_2
\]
follows by integrating
$|K_N(y)|\,|a_t(x-y)-a_t(x)|$ and using Young's inequality.
Also $\|P_Nr_t\|_2=O_t(N^{-s})$.  Thus, with
$a=e^{-itM}a_t$,
\begin{equation}\label{fe:reconstruct}
 P_Nu(t)=2\operatorname{Re}(iaW_N)+\mathcal E_N,
 \qquad
 \|\mathcal E_N\|_2
 \leq C\varepsilon N^{-1/2}+o_t(N^{-1/2}).
\end{equation}

For $N=N_j$, the convergent inequalities give $N\geq q_j$ and
$|\alpha-p_j/q_j|N^2\leq1$.  Apply
Lemma~\ref{fe:atomicpackets} and form its set $\Omega$.
Lemma~\ref{fe:realmass} supplies a positive lower bound of size
$N^{-1/2}$ for the total norm of the real packet.  Its norm
outside $\Omega$ is at most twice that of $W_N$ there.
First choose $\varepsilon$ small relative to $|c_*|$, then
choose $A$ large in \eqref{fe:concentration}.  Equations
\eqref{fe:packetmass} and~\eqref{fe:reconstruct} show that
\[
 \|\mathbf1_\Omega P_Nu(t)\|_2\geq c_tN^{-1/2}
\]
for all sufficiently large $j$.  Since
$|\Omega|\leq2Jq_jA/N$, this gives
\[
 \|P_Nu(t)\|_\infty
 \geq c_tN^{-1/2}\left(\frac{N}{2Jq_jA}\right)^{1/2}
 \geq c_t' q_j^{-1/2}.
\]
The constants absorb the normalization of measure.  This proves
the proposition.
\end{proof}

\begin{proof}[Proof of Proposition~\ref{fe:classification}]
If $u(t)\in C^\gamma$, combine
\eqref{fe:holderblock} and Proposition~\ref{fe:lower} to obtain
\[
 c_tq_j^{-1/2}\leq C_t(q_jq_{j+1})^{-\gamma/2}.
\]
Rearranging proves $q_{j+1}\leq C_tq_j^{1/\gamma-1}$.
Finitely many initial denominators are absorbed into the
constant.

Conversely, put $\kappa=1/\gamma-1\geq1$. For
$q_j\leq m<q_{j+1}$, best approximation and the convergent
inequalities give
\[
 \|m\alpha\|_{\R/\Z}\geq\|q_j\alpha\|_{\R/\Z}
 >\frac1{q_j+q_{j+1}}.
\]
Thus $q_{j+1}\leq Cq_j^\kappa$ implies
$\|m\alpha\|\geq c m^{-\kappa}$. Conversely, taking $m=q_j$
and using $\|q_j\alpha\|<q_{j+1}^{-1}$ gives the denominator bound.
Proposition~\ref{ar:nonlinear}
therefore gives $u(t)\in C^\gamma$.  This proves
\eqref{fe:iff}.  The critical-exponent formula follows by taking
the limsup of the logarithms of the denominators.  Exponents
larger than $1/2$ are impossible: the packet lower bound and
$q_{j+1}\geq q_j$ already contradict
\eqref{fe:holderblock} for every such exponent below one.
\end{proof}

\begin{remark}[What holds for a general finite gauge expansion]
\label{fe:admissible}
The lower bound has a second application.  Suppose
$u_0\in H^{1/2-}_{r,0}(\T)$ and its initial gauge coefficients
have the form
\[
 \widehat{\mathcal G(u_0)}(n)
 =\frac1n\sum_{k=1}^Jc_ke^{-ina_k}+\rho_n,
 \qquad
 \sum_{n\geq1}n^{1+2\varepsilon}|\rho_n|^2<\infty,
\]
with distinct $a_k$ and at least one nonzero $c_k$.
The remainder packet has $L^2$ norm
$O(N^{-1/2-\varepsilon})$.  Sobolev persistence and the same
reconstruction argument therefore prove
Proposition~\ref{fe:lower}, without any geometric condition on
the locations $a_k$.  Consequently, the right-hand arithmetic
condition in \eqref{fe:iff} is necessary for $C^\gamma$
regularity of every such solution.

For this abstract gauge class, necessity should not be confused
with the full equivalence.  The weak Sobolev assumption on
$\rho$ alone does not supply the $C^\gamma$ upper bound for all
$\gamma\leq1/2$.  The full equivalence above uses the BV
structure of the original datum.  A separate sufficient
condition for the abstract class is
$\sum_{n\geq1}n^{2+2\varepsilon}|\rho_n|^2<\infty$,
which makes its linear evolution
H\"older continuous of every order at most $1/2$ at every time.
\end{remark}

\subsection{The Besov scale}
\label{bs:section}

The same packet argument gives more than the H\"older exponent.  It
also determines the dependence on spatial integrability and on the
summability of the frequency blocks.  We use the inhomogeneous periodic
Besov spaces, defined by a smooth dyadic partition of unity as
\[
 \|f\|_{B^s_{p,r}}=
 \|\Delta_0f\|_{L^p}
 +\big\|(2^{ks}\|\Delta_kf\|_{L^p})_{k\geq1}\big\|_{\ell^r}.
\]
We use $k$ for the dyadic index. At a dyadic frequency scale
$N=2^k$, we write $\Delta_{[N]}=\Delta_k$. The test cutoff
$P_N$ of the preceding subsection is defined for every real $N\geq1$
and need not be a member of this partition.
For $r=\infty$ the last norm is a supremum.  A different smooth
partition gives an equivalent norm.  In particular,
$B^s_{\infty,\infty}=C^s$ when $0<s<1$.

\begin{proposition}\label{bs:spectrum}
Let $u_0$ be real, mean zero, and of bounded variation, with at least
one nonzero jump.  Fix an irrational $\alpha=t/(2\pi)$, with
convergent denominators $q_j$.  Let $2<p\leq\infty$, set
$a_p=1/2-1/p$, and use $1/\infty=0$.
\begin{enumerate}[label=\textup{(\roman*)}]
\item If $1/p<s\leq1/2$ and $1\leq r<\infty$, then
\begin{equation}\label{bs:positive}
 u(t)\in B^s_{p,r}
 \quad\Longleftrightarrow\quad
 \sum_j
 \left(\frac{q_{j+1}^{s-1/p}}
 {q_j^{1-1/p-s}}\right)^{r/2}<\infty.
\end{equation}
For $r=\infty$, the corresponding criterion is boundedness of
the ratios in \eqref{bs:positive}.
\item At $s=1/p$ and $1\leq r<\infty$, the criterion is
\begin{equation}\label{bs:critical}
 u(t)\in B^{1/p}_{p,r}
 \quad\Longleftrightarrow\quad
 \sum_j\frac{\log(q_{j+1}/q_j)}{q_j^{a_pr}}<\infty.
\end{equation}
The solution always belongs to $B^{1/p}_{p,\infty}$, and at an
irrational time its normalized blocks tend to zero:
\[
 2^{k/p}\|\Delta_ku(t)\|_{L^p}\longrightarrow0.
\]
\item For each $p$ in this range, the supremum of the indices $s$
for which $u(t)\in B^s_{p,\infty}$ is
\begin{equation}\label{bs:criticalindex}
 \frac1p+\frac{1-2/p}{1+\kappa(\alpha)},
 \qquad
 \kappa(\alpha)=\limsup_j\frac{\log q_{j+1}}{\log q_j}.
\end{equation}
The second term is zero when $\kappa=\infty$.
\end{enumerate}
\end{proposition}

Thus the continuity sufficient condition has an exact interpretation:
\begin{equation}\label{bs:brjuno}
 u(t)\in B^0_{\infty,1}
 \quad\Longleftrightarrow\quad
 \sum_j\frac{\log q_{j+1}}{\sqrt{q_j}}<\infty.
\end{equation}
Indeed, the difference between the sums with $\log q_{j+1}$ and
$\log(q_{j+1}/q_j)$ is summable.  Corollary~\ref{gm:BO} shows that, within this BV-jump solution
class, the same criterion also characterizes continuity. This additional
conclusion requires a separate argument: the embedding
$B^0_{\infty,1}\subset C^0$ is strict for general functions.

\begin{proof}
We give the frequency bounds and the summation argument.  They also
explain the different form of the endpoint condition.
For the linear gauge profile, its measure representation and Parseval
give, for every smooth dyadic block at frequency $N$,
\[
 \|\Delta_{[N]}w_L(t)\|_2\leq C N^{-1/2}.
\]
When $q_j\leq N<q_{j+1}$, apply the uniform quadratic-sum bound
with $q_j$ and with $q_{j+1}$, followed by summation by parts.  It gives
\[
 \|\Delta_{[N]}w_L(t)\|_\infty
 \leq C\min\{q_j^{-1/2},q_{j+1}^{1/2}/N\}.
\]
The constants may depend on the total variation of the gauge measure
but not on $j$ or $N$.  Interpolation yields
\begin{equation}\label{bs:upperblock}
 \|\Delta_{[N]}w_L(t)\|_p
 \leq C N^{-1/p}
 \min\{q_j^{-a_p},q_{j+1}^{a_p}N^{-2a_p}\}.
\end{equation}

These bounds pass to Besov membership of the actual solution.
In fact, multiplication by a $C^\rho$ function is bounded on
$B^s_{p,r}$ whenever $0\leq s<\rho$ and $1\leq p,r\leq\infty$.
To recall the elementary dyadic argument, split the product into
low--high, high--low, and comparable-frequency terms.  The low--high
part is bounded by the sup norm of the multiplier.  In the other
two parts, its frequency-$2^k$ block is bounded by $C2^{-k\rho}$;
the resulting convolutions of block norms are summable because
$\rho>s$.  This argument includes $s=0$, where a harmless linear
factor in the dyadic index is absorbed by the geometric decay.
Use this fact in
\[
 u(t)=2\operatorname{Re}(ia_tw_L(t))+r_t,
 \qquad a_t,r_t\in C^\rho\quad(\rho<1).
\]
Choose $\rho>1/2$.  The remainder belongs to all the Besov spaces
in the statement, so \eqref{bs:upperblock} gives every asserted
sufficient condition.

For necessity, use the test packet $P_N$ of Section~\ref{fe:section}.
The proof of Proposition~\ref{fe:lower} applies uniformly throughout
\[
 q_j\leq N\leq L_j:=\sqrt{q_jq_{j+1}}.
\]
Indeed, these are precisely the two conditions $N\geq q_j$ and
$|\alpha-p_j/q_j|N^2\leq1$ used there.  The retained atomic set,
the discarded total variation, and the radius factor are fixed before
$j$ and $N$ tend to infinity.  Consequently there is a set $\Omega$
of length at most $Cq_j/N$ on which
\[
 \|\mathbf1_\Omega P_Nu(t)\|_2\geq cN^{-1/2}.
\]
H\"older's inequality on this set gives the lower bound
\begin{equation}\label{bs:lowerblock}
 \|P_Nu(t)\|_p
 \geq cN^{-1/p}q_j^{-a_p},
 \qquad q_j\leq N\leq L_j,
\end{equation}
for all sufficiently large $j$.  This includes $p=\infty$.
For every real $N\geq1$, the symbol of $P_N$ meets only a fixed
number of blocks $\Delta_k$ with $2^k\asymp N$. The multipliers
$P_N\Delta_k$ have uniformly bounded convolution-kernel $L^1$
norms. Consequently,
\[
 \|P_Nf\|_p\leq C\sum_{2^k\asymp N}\|\Delta_kf\|_p,
 \qquad 1\leq p\leq\infty.
\]
This comparison permits the real scales $N=L_j$ in the necessary
Besov estimates.

Suppose first that $s>1/p$.  The two expressions in
\eqref{bs:upperblock} agree at $N=L_j$.  After multiplication by
$N^s$, the lower-frequency side grows geometrically, with exponent
$s-1/p>0$, and the higher-frequency side decreases geometrically,
with exponent $s-1+1/p<0$.  Their dyadic $\ell^r$ sum is bounded
by a constant times the $r$th power of the peak
\[
 A_j=q_j^{-a_p}L_j^{s-1/p}
 =\left(\frac{q_{j+1}^{s-1/p}}
 {q_j^{1-1/p-s}}\right)^{1/2}.
\]
Conversely, \eqref{bs:lowerblock} at $N=L_j$ bounds the same
peak from below.  The scales $L_j$ have bounded overlap in dyadic
annuli, since $L_{j+2}\geq2L_j$.  Thus membership forces
$(A_j)\in\ell^r$, or boundedness for $r=\infty$.
This proves \eqref{bs:positive}.

At $s=1/p$, each lower-frequency block has normalized size
$q_j^{-a_p}$. There are $\tfrac12\log_2(q_{j+1}/q_j)+O(1)$
such dyadic scales. The upper side still decreases geometrically.  The estimates above consequently
give \eqref{bs:critical}; changing the number of blocks by a fixed
constant adds only $\sum_jq_j^{-a_pr}<\infty$.
They also show that the normalized blocks of $w_L$ tend to zero
when $j\to\infty$.  The same is true after multiplication by
$a_t$: one may first approximate $w_L$ by a smooth finite spectral truncation
in $B^{1/p}_{p,\infty}$ and use the bounded multiplier estimate.
The product of this finite sum with $a_t$, as well as $r_t$, has
normalized blocks tending to zero.  This proves the last assertion
of part \textup{(ii)}.

Finally, the bounded-ratio criterion gives \eqref{bs:criticalindex}
by taking logarithms.  The crude bounds
$\|\Delta_{[N]}w_L\|_2\leq CN^{-1/2}$ and
$\|\Delta_{[N]}w_L\|_\infty\leq C$ give membership for every
$s<1/p$.  No exponent above $1/2$ is possible, by the peak lower
bound.  These observations cover the limiting cases of the formula.
\end{proof}

\begin{remark}\label{bs:p2}
At $p=2$, a nonzero jump gives
$u(t)\in B^{1/2}_{2,\infty}$ and
$u(t)\notin B^{1/2}_{2,r}$ for every finite $r$, at every real time.
Indeed, the packet $L^2$ mass is asymptotic to a strictly positive
constant times $N^{-1}$, while the Fourier bound and the reconstruction
give the matching upper estimate.  This conclusion requires no
arithmetic assumption.
\end{remark}

\subsection{Fractal graphs}
\begin{lemma}[A graph-dimension criterion]\label{ar:graphlemma}
Let $f:\T\to\R$ be continuous, assume that $f\in C^\gamma(\T)$
for every $\gamma<1/2$, and suppose that $f\notin H^{1/2}(\T)$.
Then the upper Minkowski dimension of its graph is $3/2$.
\end{lemma}

Here the upper Minkowski dimension of a compact set $G$ is
$\limsup_{h\downarrow0}\log N_h(G)/\log(1/h)$, where $N_h(G)$
is the number of squares of side length $h$ in a fixed square grid
that meet $G$. Using a minimal covering gives the same dimension.

\begin{proof}
The H\"older bounds give upper Minkowski dimension at most
$2-\gamma$ for each $\gamma<1/2$, and hence at most $3/2$.
Suppose that the dimension is $D<3/2$, and choose $d$ with
$D<d<3/2$.  For a sufficiently small mesh $h=2\pi/m$, the graph
meets at most $C h^{-d}$ squares of the grid with side length $h$.
Let $\omega_j$ be the oscillation of $f$ over its $j$th horizontal
column.  Continuity implies that the vertical projection of the
graph in that column is an interval.  Counting the grid squares
in each column therefore gives
\[
 \sum_j\omega_j\leq C h^{1-d}.
\]
For $0<z\leq h$, the points $x$ and $x+z$ lie in at most two
adjacent columns.  Thus
\[
 \|f(\cdot+z)-f\|_{L^1}\leq C h\sum_j\omega_j
 \leq C h^{2-d}.
\]
Choosing $h$ comparable to $z$ yields the same bound with $z$ in
place of $h$.  Take $\gamma<1/2$ sufficiently close to $1/2$ that
$\gamma+2-d>1$.  Then
\[
 \|f(\cdot+z)-f\|_{L^2}^2
 \leq\|f(\cdot+z)-f\|_{L^\infty}
      \|f(\cdot+z)-f\|_{L^1}
 \leq C z^{\gamma+2-d}.
\]
It follows that
\[
 \int_0^1\frac{\|f(\cdot+z)-f\|_{L^2}^2}{z^2}\,dz<\infty.
\]
The difference-quotient characterization of $H^{1/2}(\T)$ now
contradicts the hypothesis.  Therefore $D=3/2$.
\end{proof}

\begin{corollary}[Sharp regularity and nonlinear fractal graphs]
\label{ar:fractal}
Under the assumptions of Proposition~\ref{ar:nonlinear}, suppose in
addition that $u_0\notin H^{1/2}(\T)$.  Then $u(t)$ does not belong
to $C^{1/2+\eps}(\T)$ for any $\eps>0$ at any time.  For almost
every time, its continuous representative belongs to every
$C^\gamma(\T)$ with $\gamma<1/2$, and its graph has upper
Minkowski dimension exactly $3/2$.
\end{corollary}

\begin{proof}
The global BO flow in $H^{1/2}$ agrees with the $L^2$ flow and is
invertible \cite[Corollary~A.4]{GKT-wp}.  If $u(t)$ belonged to
$H^{1/2}$ at one time, evolving backward would give
$u_0\in H^{1/2}$, a contradiction.  The embedding
$C^{1/2+\eps}\subset H^{1/2}$ rules out the stated H\"older
regularity.  The remaining assertions follow from
Proposition~\ref{ar:nonlinear} and Lemma~\ref{ar:graphlemma}.
\end{proof}

In particular, Corollary~\ref{ar:fractal} applies to every nonzero
multiple of the square wave.  More generally, it applies to
piecewise smooth data with a genuine jump.  Such a jump prevents
membership in $H^{1/2}$: the difference-quotient integral across
the two sides of the jump diverges logarithmically.  The conclusion
concerns the graph of the real nonlinear solution itself.

\section{Spatial localization}\label{sec:localization}

The global lower bound detects any jump. Local regularity also depends
on how jump locations interact under rotation by $2t$. We first treat
a finite jump set, then compute the limiting energy measure for a
general BV datum and construct data for which its density is positive
everywhere at every irrational time.

\subsection{Finite jump sets}
\label{loc:section}

For a finite set of jumps, the preceding arithmetic obstruction is
local in space.  Here $C^\gamma_{\mathrm{loc}}(I)$ means
H\"older regularity on every compact subarc of the open arc $I$.
It is a neighborhood regularity statement, not a claim about
pointwise H\"older exponents at individual points.

\begin{proposition}\label{loc:classification}
Let $u_0\in\operatorname{BV}(\T)$ be real and mean zero, and
suppose it has a finite, nonempty set of genuine jumps.  The
atomless part of its derivative may be arbitrary.  Fix an
irrational $\alpha=t/(2\pi)$ and its convergent denominators
$q_j$.  For every nonempty open arc $I\subset\T$ and every
$0<\gamma\leq1/2$,
\[
 u(t)\in C^\gamma_{\mathrm{loc}}(I)
 \quad\Longleftrightarrow\quad
 \sup_j\frac{q_{j+1}}{q_j^{1/\gamma-1}}<\infty.
\]
Thus a Liouville-time solution has no open arc on which it has
any positive H\"older regularity, even when the solution is
continuous.
\end{proposition}

\begin{lemma}[Local packet energy]\label{loc:energy}
Let $W_N$ be the finite atomic packet of
Lemma~\ref{fe:atomicpackets}, and fix $t=2\pi\alpha$ with
$\alpha$ irrational.  Partition the jump positions into
equivalence classes under
\[
 a_j\sim a_k\quad\Longleftrightarrow\quad
 a_j-a_k\in2t\Z\pmod{2\pi}.
\]
For each class choose a reference point $b_r$ and the distinct
integers $m_j$ satisfying $a_j\equiv b_r+2tm_j\pmod{2\pi}$.
Set
\[
 D_t(x)=\sum_r
 \left|\sum_{j\in r}c_j
 e^{-itm_j^2-im_jb_r}e^{im_jx}\right|^2.
\]
For every smooth function $\chi$,
\begin{equation}\label{loc:energylimit}
 N\int_\T\chi(x)|W_N(x)|^2\,\frac{dx}{2\pi}
 \longrightarrow I_\varphi\int_\T\chi(x)D_t(x)\,
 \frac{dx}{2\pi}.
\end{equation}
If at least one $c_j$ is nonzero, then the right-hand side is
strictly positive for every nonnegative, nonzero smooth $\chi$.
\end{lemma}

\begin{proof}
Expand the two Fourier sums in $|W_N|^2$ and put $k=n-m$.
The contribution from the jumps at $a_j,a_\ell$ contains the factor
\[
 c_j\overline{c_\ell}\,
 e^{-itk^2-ika_\ell}
 e^{in(2tk-a_j+a_\ell)}e^{ikx}.
\]
For each fixed $k$, the normalized inner sum has a nonzero
limit only if
$a_j-a_\ell\equiv2tk\pmod{2\pi}$; in that case the limit is
$I_\varphi$.  Otherwise smooth summation makes it tend to zero.
The normalized inner sums are uniformly bounded in $N$ and
$k$ by Cauchy--Schwarz.  Since the Fourier coefficients of
$\chi$ are absolutely summable, dominated convergence
justifies summing these limits over $k$.

Irrationality of $\alpha$ makes the resonant integer $k$
unique for each pair of jump points.  Within one equivalence class,
$k=m_j-m_\ell$, and its phase simplifies to
\[
 e^{-itk^2-ika_\ell}
 =e^{-it(m_j^2-m_\ell^2)-i(m_j-m_\ell)b_r}.
\]
The surviving terms therefore sum to the stated $D_t$.
Each nonzero class polynomial is a finite trigonometric
polynomial with distinct frequencies.  It cannot vanish on
an open interval.  Hence a nonnegative smooth function that
is positive somewhere has strictly positive integral
against $D_t$.
\end{proof}

\begin{proof}[Proof of Proposition~\ref{loc:classification}]
Only necessity requires an argument. Use the measure
\[
 \nu=-\frac1{2\pi}D(u_0e^{-i\partial_x^{-1}u_0}).
\]
Its atomic part is now finite, and its remaining part is
atomless.  Choose a nonnegative, nonzero smooth $\chi$
supported in a compact subarc of $I$.  By
Lemma~\ref{loc:energy}, the finite atomic packet satisfies
\[
 \|\sqrt\chi W_N\|_2\geq c_{t,\chi}N^{-1/2}
\]
for large $N$.  The proof of Lemma~\ref{fe:realmass} works
with the extra factor $\chi$: its error tests the
high-frequency portion of $\chi a_t^2$.  It follows that
\[
 \|\sqrt\chi\,2\operatorname{Re}
 (i e^{-itM}a_tW_N)\|_2
 \geq c_{t,\chi}N^{-1/2}.
\]
The atomless measure, the reconstruction commutator, and the
smoothing remainder contribute $o(N^{-1/2})$ in $L^2$, as
in Proposition~\ref{fe:lower}.

Take $N=\sqrt{q_jq_{j+1}}$ and use the same union $\Omega$
of rational-grid neighborhoods.  By taking its fixed radius
factor $A$ sufficiently large, its complement has too little
packet mass to absorb the preceding local lower bound.
Thus
\[
 \|\mathbf1_{\Omega\cap\operatorname{supp}\chi}
 P_Nu(t)\|_2\geq c_{t,\chi}N^{-1/2},
\]
and consequently
\begin{equation}\label{loc:blocklower}
 \sup_{x\in\operatorname{supp}\chi}|P_Nu(t,x)|
 \geq c_{t,\chi}q_j^{-1/2}.
\end{equation}

Suppose that $u(t)\in C^\gamma_{\mathrm{loc}}(I)$.
Choose a smooth cutoff $\eta$ supported in $I$, equal to one
on a neighborhood of $\operatorname{supp}\chi$.
Then $\eta u(t)\in C^\gamma(\T)$, so its packet has sup norm
$O(N^{-\gamma})$.  On $\operatorname{supp}\chi$, the packet
of $(1-\eta)u(t)$ is $O_K(N^{-K})$ for every fixed $K$,
by the spatial separation, the kernel decay of $P_N$, and
$u(t)\in L^2$.  Combining this upper bound with
\eqref{loc:blocklower} yields
$q_{j+1}\leq C_{t,\chi}q_j^{1/\gamma-1}$.
The converse follows from the global BV upper-regularity
theorem.
\end{proof}

\begin{remark}
The finiteness assumption in this local theorem ensures that the
class polynomials cannot vanish on an interval. The next subsection
identifies the corresponding energy density
and extends the local theorem whenever that density has positive
mass. It also gives classes of infinite jump configurations for
which this condition holds at every irrational time.
\end{remark}

\subsection{The high-frequency energy measure}
\label{lbv:section}

The restriction to finitely many jumps can be substantially relaxed.
The relevant quantity is a continuous density determined by the
jump coefficients and their orbits under rotation by $2t$.  It also
gives an exact high-frequency energy law for the nonlinear solution.

Fix a nonzero, nonnegative $\varphi\in C_c^\infty((1,2))$, put
\[
 P_Nf=\sum_{n\in\Z}\varphi(|n|/N)\widehat f(n)e^{inx},
 \qquad
 I_\varphi=\int_0^\infty\frac{\varphi(s)^2}{s^2}\,ds,
\]
and use normalized measure $dx/(2\pi)$ in the $L^2$ norm.
For a real mean-zero datum $u_0\in\operatorname{BV}(\T)$, let
$\mathcal J$ be its countable set of genuine jumps, and define
\begin{equation}\label{lbv:coefficients}
 c_a=-\frac{[u_0]_a}{2\pi}e^{-iV_0(a)},
 \qquad V_0=\partial_x^{-1}u_0,
 \qquad a\in\mathcal J.
\end{equation}
These coefficients are absolutely summable.

\begin{proposition}[The nonlinear high-frequency energy law]
\label{lbv:energy}
Let $u(t)$ be the BO solution with real mean-zero BV initial
datum $u_0$.  For every real $t$,
\begin{equation}\label{lbv:globalenergy}
 \lim_{N\to\infty}N\|P_Nu(t)\|_2^2
 =\frac{I_\varphi}{2\pi^2}
   \sum_{a\in\mathcal J}|[u_0]_a|^2.
\end{equation}
Suppose in addition that $\alpha=t/(2\pi)$ is irrational.
Partition $\mathcal J$ into equivalence classes under
$a\sim b$ if $a-b\in2t\Z$ modulo $2\pi$.  For each class
choose a representative $b_r$ and write its points uniquely as
$a=b_r+2tm$ modulo $2\pi$, putting $c_{r,m}=0$ when that
orbit point is not a jump.  Define
\begin{equation}\label{lbv:density}
 F_{r,t}(x)=\sum_{m\in\Z}c_{r,m}
    e^{-itm^2-imb_r}e^{imx},
 \qquad
 D_t(x)=\sum_r|F_{r,t}(x)|^2.
\end{equation}
The series for each $F_{r,t}$ and for $D_t$ converge uniformly.
The density does not depend on the choice of class representatives:
replacing $b_r$ by $b_r+2th$ and the index $m$ by $m-h$ gives
\[
 F_{r,t}^{\mathrm{new}}(x)
   =e^{i(th^2+hb_r-hx)}F_{r,t}(x),
\]
so its modulus is unchanged. Changing a real representative by an
integer multiple of $2\pi$ also leaves the series unchanged.
For every continuous $\chi:\T\to\C$,
\begin{equation}\label{lbv:weakenergy}
 \lim_{N\to\infty}N\int_\T\chi(x)|P_Nu(t,x)|^2\,
       \frac{dx}{2\pi}
 =2I_\varphi\int_\T\chi(x)D_t(x)\,\frac{dx}{2\pi}.
\end{equation}
In particular, $D_t$ is a nonnegative continuous function and
\[
 \int_\T D_t(x)\,\frac{dx}{2\pi}
 =\sum_{a\in\mathcal J}|c_a|^2.
\]
\end{proposition}

\begin{proof}
Set $\nu=-(2\pi)^{-1}D(u_0e^{-iV_0})$ and use
unnormalized measure Fourier coefficients.  Then
$\widehat{\mathcal G(u_0)}(n)=\widehat\nu(n)/n$ for $n\geq1$,
and the atoms of $\nu$ are precisely those in
\eqref{lbv:coefficients}.  Write
\[
 W_N(t,x)=\sum_{n\geq1}\frac{\varphi(n/N)}n
   e^{itn^2}\widehat\nu(n)e^{inx}.
\]
The bound $|\widehat\nu(n)|\leq\|\nu\|_{\mathrm{TV}}$ gives
$\|W_N(t)\|_2\leq C\|\nu\|_{\mathrm{TV}}N^{-1/2}$.

The reconstruction \eqref{eq:GKT} and the $L^2$ commutator estimate give
\begin{equation}\label{lbv:reconstruction}
 P_Nu(t)=2\operatorname{Re}(i aW_N(t))+E_N,
 \qquad |a|=1,
 \qquad \|E_N\|_2=o_t(N^{-1/2}),
\end{equation}
where $a=e^{-itM}e^{i\partial_x^{-1}u(t)}$,
$M=(2\pi)^{-1}\int u_0^2$, and $a\in C^\rho$ for some
$\rho>1/2$.  Indeed,
$\|[P_N,a_t]w_L\|_2=O_t(N^{-\rho})$, and the smoothing
remainder belongs to $H^s$ for some $s>1/2$.

For a real smooth $\chi$, the square of the real part satisfies
\[
 \int\chi\,|2\operatorname{Re}(iaW_N)|^2
 =2\int\chi\,|W_N|^2
  -2\operatorname{Re}\int\chi a^2W_N^2.
\]
The integrals in this proof use normalized measure.  Since
$W_N^2$ has Fourier support in $(2N,4N)$, the last integral
only tests a smooth frequency cutoff of $\chi a^2$ of size
comparable to $N$.  Its sup norm is $O_{t,\chi}(N^{-\rho})$.
After multiplication by $N$, this term tends to zero.
The error in \eqref{lbv:reconstruction} also changes the
weighted squared norm by $o(N^{-1})$.  Consequently,
\begin{equation}\label{lbv:reduceenergy}
 N\int\chi|P_Nu(t)|^2
 =2N\int\chi|W_N(t)|^2+o(1).
\end{equation}

The atomless part of $\nu$ contributes $o(N^{-1/2})$ to
$W_N$ in $L^2$.  This is the Wiener argument of
Lemma~\ref{fe:wiener}: the averages
$N^{-1}\sum_{n=1}^N|\widehat\nu_C(n)|^2$ tend to zero.
A measure of total variation $\varepsilon$ contributes at
most $C\varepsilon N^{-1/2}$.  It is therefore enough first
to prove the assertions for a finite collection of atoms and
then let the omitted atomic total variation tend to zero.

For finitely many atoms, Parseval and smooth summation give,
at every time,
\[
 N\|W_N(t)\|_2^2\longrightarrow
 I_\varphi\sum_a|c_a|^2.
\]
The quadratic time phase disappears from this squared norm.
Taking $\chi=1$ in \eqref{lbv:reduceenergy} proves
\eqref{lbv:globalenergy}, including rational times.

At an irrational time, expand $\int\chi|W_N|^2$ and put
$k=n-m$ in the two frequency sums.  For atoms at $a,b$, the
oscillating factor is
\[
 c_a\overline{c_b}
 e^{-itk^2-ikb}e^{in(2tk-a+b)}e^{ikx}.
\]
After multiplication by $N$, the inner sum tends to
$I_\varphi$ if $a-b\equiv2tk$ modulo $2\pi$, and to zero
otherwise.  The inner sums are bounded uniformly in $N$ and
$k$ by Cauchy--Schwarz, so absolute summability of the Fourier
coefficients of $\chi$ justifies passage to the limit.
For two points in one rotation class, $k=m_a-m_b$, and
\[
 e^{-itk^2-ikb}
 =e^{-it(m_a^2-m_b^2)-i(m_a-m_b)b_r}.
\]
Thus the limit is $I_\varphi\int\chi D_t$.

For infinitely many atoms, the sum
$\sum_r\|F_{r,t}\|_{\mathcal A}$ is bounded by
$\sum_a|c_a|$, where $\mathcal A$ denotes the Wiener algebra.
In particular, the defining series for $D_t$ is uniformly
convergent.  Removing atoms of total mass $\varepsilon$
changes $D_t$ in sup norm by at most
$2\varepsilon\sum_a|c_a|$.  It changes
$N\int\chi|W_N|^2$ by at most $C_\chi\varepsilon$,
uniformly in $N$.  The finite-atom result therefore passes
to the limit and proves \eqref{lbv:weakenergy} for smooth
$\chi$.  The uniformly bounded total masses on the left,
approximation of continuous test functions by smooth ones,
and complex linearity give the stated weak convergence.  The formula for
the integral of $D_t$ follows from Fourier orthogonality
within each class.
\end{proof}

\begin{proposition}[A local criterion for arbitrary BV jumps]
\label{lbv:local}
Suppose $u_0\in\operatorname{BV}(\T)$ is real and mean zero
and has a genuine jump.  Fix an irrational time parameter
$\alpha=t/(2\pi)$, and let $I$ be a nonempty open arc for
which $D_t$ is not identically zero on $I$.  Then, for every
$0<\gamma\leq1/2$,
\begin{equation}\label{lbv:localiff}
 u(t)\in C^\gamma_{\mathrm{loc}}(I)
 \quad\Longleftrightarrow\quad
 \sup_j\frac{q_{j+1}}{q_j^{1/\gamma-1}}<\infty.
\end{equation}
Here $C^\gamma_{\mathrm{loc}}(I)$ means H\"older regularity
on each compact subarc of $I$.
\end{proposition}

\begin{proof}
Choose a nonnegative smooth $\chi$ supported inside $I$ with
$\int\chi D_t>0$.  Proposition~\ref{lbv:energy} gives a local
packet norm bounded below by $c_{t,\chi}N^{-1/2}$.
Retain finitely many atoms of $\nu$, leaving an atomic tail
of total variation less than $\varepsilon$.  The atomless
part is $o(N^{-1/2})$ in packet norm, and the atomic tail is
$O(\varepsilon N^{-1/2})$.  Choose $\varepsilon$ small
relative to the fixed positive local mass.  The finite
retained packet thus carries a positive fraction of that
mass.

For a convergent $p_j/q_j$, take
$N=\sqrt{q_jq_{j+1}}$.  Concentrate the retained packet in
the union $\Omega$ of its rational-grid neighborhoods of
radius $A/N$, as in Lemma~\ref{fe:atomicpackets}.  With $A$
fixed sufficiently large, the omitted mass is small and
$|\Omega|\leq Cq_j/N$.  The reconstruction errors are
$o(N^{-1/2})$ in $L^2$.  Hence
\[
 \sup_{x\in\operatorname{supp}\chi}|P_Nu(t,x)|
 \geq c_{t,\chi}q_j^{-1/2}.
\]
If $u(t)\in C^\gamma_{\mathrm{loc}}(I)$, multiply it by a
smooth cutoff supported in $I$ and equal to one near
$\operatorname{supp}\chi$.  Its packet is $O(N^{-\gamma})$;
the contribution from outside the cutoff is $O_K(N^{-K})$
on $\operatorname{supp}\chi$, by kernel decay and spatial
separation.  Comparison of the two bounds gives the
necessary denominator condition.  Its sufficiency follows
from the global BV upper-regularity theorem.
\end{proof}

\begin{corollary}[Only countably many possible resonance exceptions]
\label{lbv:countable}
For a fixed real mean-zero BV datum with a genuine jump,
choose representatives of its jump points in $[0,2\pi)$ and
define the countable set
\[
 \mathscr R(u_0)=
 \left\{\frac{a-b+2\pi\ell}{2m}:
 a,b\in\mathcal J,\ a\ne b,\ \ell\in\Z,
 \ m\in\Z\setminus\{0\}\right\}.
\]
At every irrational time parameter with $t\notin\mathscr R(u_0)$,
\eqref{lbv:localiff} holds on every nonempty open arc.
If all jump locations lie in one coset of $2\pi\mathbb Q$,
the same assertion holds at every irrational time parameter.
\end{corollary}

\begin{proof}
Outside $\mathscr R(u_0)$, every rotation class contains a
single jump.  Thus
$D_t\equiv\sum_a|c_a|^2>0$, and
Proposition~\ref{lbv:local} applies to every arc.
If the jump locations lie in one coset of $2\pi\mathbb Q$,
every member of $\mathscr R(u_0)$ has rational ratio to
$2\pi$, so there are no irrational-time exceptions.
\end{proof}

\begin{corollary}[Further sufficient conditions at a resonance time]
\label{lbv:nonvanishing}
Fix an irrational time parameter.  The local equivalence
\eqref{lbv:localiff} holds on every nonempty open arc if at
least one nonempty rotation class satisfies any of the
following conditions:
\begin{enumerate}
 \item it contains only finitely many jumps;
 \item its nonzero indices $m$ are bounded below or bounded
       above;
 \item one jump has absolute size strictly larger than the
       sum of the absolute sizes of all other jumps in that
       class;
 \item for some $\eta>0$,
       $\sum_m|c_{r,m}|e^{\eta|m|}<\infty$.
\end{enumerate}
\end{corollary}

\begin{proof}
It suffices to show that the corresponding $F_{r,t}$ cannot
vanish identically on an open arc.  In the first case it is
a nonzero trigonometric polynomial.  In the third case the
triangle inequality gives a strictly positive pointwise
lower bound.  In the fourth case its Fourier series extends
analytically to an annulus, so the identity theorem applies.

In the second case, after multiplying by a Fourier character
and, if needed, taking a conjugate, one has a nonzero
function analytic in the unit disk and continuous on its
closure, with absolutely summable Taylor coefficients.  If
its boundary values vanished on an open arc, the Cauchy
integral formula, with that arc omitted, would analytically
continue the function across the arc.  The continuation
would vanish there and hence identically.  This contradicts
the presence of a nonzero coefficient.  Therefore $D_t$
has positive mass on every open arc in all four cases.
\end{proof}

\begin{corollary}[Jumps on a one-sided rotation orbit]
\label{lbv:semiorbit}
Suppose that the nonempty jump set of $u_0$ is contained in
$\{b+m\theta:m\geq0\}$ modulo $2\pi$, for some $b,\theta\in\R$.
Then the local equivalence \eqref{lbv:localiff} holds on every
nonempty open arc at every irrational time parameter.
\end{corollary}

\begin{proof}
If $\theta/(2\pi)$ is rational, there are only finitely many
jump locations, and Corollary~\ref{lbv:nonvanishing} applies.
Suppose that $\theta/(2\pi)$ is irrational and fix an irrational
$t/(2\pi)$.  Consider the subgroup
\[
 H=\{(d,k)\in\Z^2:\theta d-2tk\in2\pi\Z\}.
\]
Both coordinate projections of $H$ are injective.  Thus either
$H=\{0\}$ or $H=\Z(d_0,k_0)$ for some nonzero integers $d_0,k_0$:
indeed, project onto a subgroup of $\Z$ and lift its generator.
In the first case each rotation class contains at most one
jump.  In the second case, within any fixed class, the original
indices have the form $m_0+d_0\ell$, while their indices along
the $2t$ orbit have the form $k_0\ell$.  The restriction
$m_0+d_0\ell\geq0$ bounds $\ell$ on one side.  Consequently
the latter indices are also bounded on one side.  In either
case Corollary~\ref{lbv:nonvanishing} applies.
\end{proof}

This gives a fixed class of genuinely infinite jump sets for
which there are no irrational-time exceptions.  For example,
fix an irrational $\theta/(2\pi)$ and prescribe any nonzero
summable sequence of real jumps $r_m$, $m\geq0$, at the distinct
points $b+m\theta$ modulo $2\pi$.  The real measure
\[
 \sum_{m\geq0}r_m\delta_{b+m\theta}
 -\frac{\sum_{m\geq0}r_m}{2\pi}\,dx
\]
has total mass zero and is the derivative of a periodic BV
function, whose mean may be adjusted to zero.  The local theorem
therefore applies at every irrational time, without any
restriction on the decay beyond summability.

\begin{corollary}[The local Besov scale]
\label{lbv:besov}
Let $u_0$ be real, mean zero, and of bounded variation, with
at least one nonzero jump, and fix an irrational $t/(2\pi)$.
Let $I$ be a nonempty open arc with $\int_I D_t(x)\,dx>0$.
For $2<p\leq\infty$, $1/p\leq s\leq1/2$, and
$1\leq r\leq\infty$, one has
\[
 u(t)\in B^s_{p,r,\mathrm{loc}}(I)
 \quad\Longleftrightarrow\quad
 u(t)\in B^s_{p,r}(\T).
\]
Here local membership means that $\eta u(t)\in B^s_{p,r}(\T)$
for every $\eta\in C_c^\infty(I)$.  Thus the arithmetic
criteria in Proposition~\ref{bs:spectrum} also give the exact local
Besov scale on $I$.  At $s=1/p$, every smooth cutoff satisfies
\[
 2^{k/p}\|\Delta_k(\eta u(t))\|_p\longrightarrow0.
\]
The latter assertion holds on every arc, without an assumption
on $D_t$ there.
\end{corollary}

\begin{proof}
The global sufficient conditions imply their local counterparts
by multiplication with a smooth cutoff.  For necessity, choose
a nonnegative $\chi\in C_c^\infty(I)$ such that
$\int\chi D_t>0$, and choose $\eta\in C_c^\infty(I)$ equal
to one on a neighborhood of $\operatorname{supp}\chi$.
The proof of Proposition~\ref{lbv:local} gives, uniformly for all
sufficiently large convergent indices and every real scale
$q_j\leq N\leq\sqrt{q_jq_{j+1}}$, a concentration set
$\Omega_N$ such that
\[
 |\Omega_N|\leq Cq_j/N,
 \qquad
 \|\mathbf1_{\Omega_N\cap\operatorname{supp}\chi}
                  P_Nu(t)\|_2\geq cN^{-1/2}.
\]
Indeed, the localized energy limit holds as $N\to\infty$
without any restriction to a sequence.  The finite atomic
approximation, its discarded variation, and the radius factor
in the concentration argument are fixed using $\int\chi D_t$
before $j$ and $N$ vary.  The only arithmetic conditions used
are $N\geq q_j$ and $|\alpha-p_j/q_j|N^2\leq1$.

For $x\in\operatorname{supp}\chi$, the kernel of $P_N$
is evaluated a fixed positive distance away from its diagonal
when applied to $(1-\eta)u(t)$.  Its rapid decay and
$u(t)\in L^2$ therefore give
\[
 \sup_{x\in\operatorname{supp}\chi}
 |P_N((1-\eta)u(t))(x)|\leq C_KN^{-K}\|u(t)\|_2
\]
for every $K>0$.  Consequently the preceding localized lower
bound also holds for $P_N(\eta u(t))$, after reducing $c$.
H\"older's inequality on the concentration set yields
\[
 \|P_N(\eta u(t))\|_p
 \geq cN^{-1/p}q_j^{1/p-1/2},
 \qquad q_j\leq N\leq\sqrt{q_jq_{j+1}}.
\]
These are exactly the lower estimates used for necessity in
Proposition~\ref{bs:spectrum}, so the same real-scale sampling and
dyadic summation prove all the local necessary conditions.
Finally, the global little-Besov assertion in that theorem
is preserved by multiplication with $\eta$.  This follows
either from the same dyadic product estimates or by smooth
finite spectral truncation and approximation.
\end{proof}

\subsection{An open dense set of data}
\label{lp:section}

We now construct an open dense set of BV data for which the local
classification holds at every irrational time. The construction adds
two suitably placed jumps with different absolute sizes.
The argument uses genuine real BV perturbations and is independent of
the phase of their nonlinear gauge coefficients.

\begin{lemma}[Choice of two jump locations]
\label{lp:geometry}
Let $S\subset\T$ be countable.  There is a subset
$\mathcal A_S\subset\T^2$ of full Lebesgue measure, dense in every
product of nonempty open arcs, with the following property.  If
$(a,b)\in\mathcal A_S$, then $a,b$ are distinct and lie outside $S$.
For every irrational rotation angle $\theta/(2\pi)$, at least one
of the following holds:
\begin{enumerate}
 \item the rotation class of $a$ or of $b$ in $S\cup\{a,b\}$ is
       a singleton;
 \item $a$ and $b$ belong to the same rotation class, and that class
       contains exactly these two points.
\end{enumerate}
Here rotation classes are defined by equality modulo $\theta\Z$ on
$\T$.
\end{lemma}

\begin{proof}
Choose representatives of $S$ in $[0,2\pi)$ and let $G$ be the
rational linear span of these representatives and $2\pi$.  This is
a countable subgroup of $\R$.  Take representatives $a,b$ whose
images in the rational vector space $\R/G$ are linearly independent:
\begin{equation}\label{lp:independence}
 ra+sb\in G,\quad r,s\in\mathbb Q
 \quad\Longrightarrow\quad r=s=0.
\end{equation}
The excluded pairs lie on a countable union of straight lines in
$[0,2\pi)^2$, so the admissible set has full measure and is dense
in every open rectangle.  Condition \eqref{lp:independence} also
ensures that the two new points are distinct and do not belong to $S$.

If $a$ and $b$ lie in different rotation classes and both classes
meet $S$, write
\[
 a-x=m\theta+2\pi k,\qquad
 b-y=n\theta+2\pi\ell,
 \qquad x,y\in S.
\]
The integers $m,n$ are nonzero, since the new points do not belong
to $S$.  Eliminating $\theta$ gives $na-mb\in G$, contrary to
\eqref{lp:independence}.  Thus at least one of the two classes is
a singleton.

If $a$ and $b$ are in the same class, write
$a-b=k\theta+2\pi\ell$ with $k\ne0$.  If this class also contained
$x\in S$, one would have $a-x=m\theta+2\pi h$ with $m\ne0$.
Elimination gives $(m-k)a-mb\in G$, again a contradiction.
Therefore the class contains only $a$ and $b$.
\end{proof}

For $a\in\T$, let $J_a$ denote the mean-zero sawtooth with one
upward jump of size one at $a$ and slope $-1/(2\pi)$ elsewhere.
More explicitly, away from $a$,
\[
 J_a(x)=\frac12-\frac{(x-a)\bmod2\pi}{2\pi},
 \qquad
 DJ_a=\delta_a-\frac{dx}{2\pi}.
\]
Its sup norm is at most $1/2$ and its total variation is two.

\begin{proposition}[A dense class with a uniform energy density]
\label{lp:dense}
Let $u_0\in\operatorname{BV}(\T)$ be real and mean zero; a jump
is not required.  Let $\mathcal J$ be its jump set.  For every
$(a,b)\in\mathcal A_{\mathcal J}$ and every real $\varepsilon\ne0$,
set
\begin{equation}\label{lp:perturbation}
 u_{0,\varepsilon}=u_0+\varepsilon J_a+2\varepsilon J_b.
\end{equation}
Then $u_{0,\varepsilon}$ is real and mean zero, and
\[
 \|u_{0,\varepsilon}-u_0\|_\infty\leq\frac32|\varepsilon|,
 \qquad
 \operatorname{Var}(u_{0,\varepsilon}-u_0)\leq6|\varepsilon|.
\]
At every irrational time parameter $t/(2\pi)$, its energy density
from \eqref{lbv:density} satisfies the pointwise bound
\begin{equation}\label{lp:uniformdensity}
 D_{t,\varepsilon}(x)\geq\frac{\varepsilon^2}{4\pi^2}
 \qquad\hbox{for every }x\in\T.
\end{equation}
In particular, data satisfying \eqref{lp:uniformdensity} at all
irrational times, with a positive constant depending on the datum,
are dense in the real mean-zero BV space.
\end{proposition}

\begin{proof}
The displayed sawtooth bounds prove all assertions concerning the
initial datum.  Its original jumps retain their sizes, and its new
jumps at $a,b$ have sizes $\varepsilon,2\varepsilon$.  The genuine
gauge coefficients at these points therefore have absolute values
$|\varepsilon|/(2\pi)$ and $2|\varepsilon|/(2\pi)$, respectively.
Their phases depend on the full perturbed primitive, but do not
affect these absolute values.

Apply Lemma~\ref{lp:geometry} with $\theta=2t$.  In its first case,
the singleton class contributes a constant of size at least
$\varepsilon^2/(4\pi^2)$ to $D_{t,\varepsilon}$.  In its second
case, the corresponding orbit series has exactly two terms.
The reverse triangle inequality bounds its absolute value below
by the difference of the two coefficient magnitudes, namely
$|\varepsilon|/(2\pi)$, for every $x$.  Its squared magnitude
gives the same contribution.  All other class contributions are
nonnegative, proving \eqref{lp:uniformdensity}.
Taking $\varepsilon\to0$ proves density in the BV norm.
\end{proof}

\begin{lemma}[Stability uniform in the irrational time]
\label{lp:stability}
For every $R<\infty$ there is $C_R<\infty$ such that real mean-zero
BV data $u,v$ with $\|u\|_{\mathrm{BV}},\|v\|_{\mathrm{BV}}\leq R$
satisfy
\begin{equation}\label{lp:lipschitz}
 \sup_{t/(2\pi)\notin\mathbb Q}
 \|D_t^u-D_t^v\|_\infty
 \leq C_R\|u-v\|_{\mathrm{BV}}.
\end{equation}
Here $D_t^u$ denotes the density determined by initial datum $u$, and
$\|u\|_{\mathrm{BV}}=\|u\|_{L^1}+\operatorname{Var}(u)$.
\end{lemma}

\begin{proof}
Let $A(u)=(Du)_{\mathrm{at}}$ be the atomic part of the derivative
measure and set
\[
 C(u)=-\frac1{2\pi}e^{-iV_u}A(u),\qquad V_u=\partial_x^{-1}u.
\]
This is exactly the atomic gauge measure whose coefficients occur
in $D_t^u$.  Taking atomic parts is a linear contraction in total
variation.  The primitive estimate
$\|V_u-V_v\|_\infty\leq2\pi\|u-v\|_{L^1}$ therefore gives
\begin{align*}
 \|C(u)\|_{\mathrm{TV}}&\leq\frac{\operatorname{Var}(u)}{2\pi},\\
 \|C(u)-C(v)\|_{\mathrm{TV}}
 &\leq\frac{\operatorname{Var}(u-v)}{2\pi}
       +\operatorname{Var}(v)\|u-v\|_{L^1}.
\end{align*}
The $L^1$ norms here use normalized measure.

For a fixed irrational time, place the atoms of both measures in
the rotation classes of their union, and use common representatives
to define $F_r^u,F_r^v$.  Adding zero coefficients does not change
either density.  The absolutely convergent orbit series give
\[
 \sum_r\|F_r^u\|_\infty\leq\|C(u)\|_{\mathrm{TV}},\qquad
 \sum_r\|F_r^u-F_r^v\|_\infty
       \leq\|C(u)-C(v)\|_{\mathrm{TV}}.
\]
Consequently,
\[
 \|D_t^u-D_t^v\|_\infty
 \leq\bigl(\|C(u)\|_{\mathrm{TV}}+\|C(v)\|_{\mathrm{TV}}\bigr)
             \|C(u)-C(v)\|_{\mathrm{TV}}.
\]
The preceding variation estimates prove \eqref{lp:lipschitz}.
None of these bounds depends on the time or the orbit partition.
\end{proof}

\begin{proposition}[Open dense nonvanishing]
\label{lp:opendense}
In the real mean-zero BV space, the set
\[
 \mathcal U=\left\{u_0:
   \inf_{t/(2\pi)\notin\mathbb Q}\inf_{x\in\T}D_t^{u_0}(x)>0
                 \right\}
\]
is open and dense.  Every datum in $\mathcal U$ has a nonzero jump.
For each such datum, the exact local H\"older and Besov criteria
hold on every nonempty open arc at every irrational time parameter.
\end{proposition}

\begin{proof}
Lemma~\ref{lp:stability} shows that the infimum defining $\mathcal U$
is locally Lipschitz in the BV norm.  Its strict positivity is
therefore an open condition.  Proposition~\ref{lp:dense} proves density.
A datum without jumps has $D_t\equiv0$, so does not lie in
$\mathcal U$.  For a datum in $\mathcal U$, every open arc has
positive energy mass at every irrational time.  The local criteria
follow from Proposition~\ref{lbv:local} and Corollary~\ref{lbv:besov}.
\end{proof}

\begin{corollary}[Local sharp regularity without exceptional times]
\label{lp:local}
For every datum \eqref{lp:perturbation}, every irrational time
$t/(2\pi)$, and every nonempty open arc $I$, local membership in
$C^\gamma(I)$, $0<\gamma\leq1/2$, is governed by the same exact
denominator criterion as global membership.  More precisely, with
the compact-subarc interpretation of local H\"older regularity,
\[
 u_\varepsilon(t)\in C^\gamma_{\mathrm{loc}}(I)
 \quad\Longleftrightarrow\quad
 \sup_j\frac{q_{j+1}}{q_j^{1/\gamma-1}}<\infty.
\]
The same local-to-global equivalence holds in the entire Besov
range $2<p\leq\infty$, $1/p\leq s\leq1/2$, and
$1\leq r\leq\infty$ of Corollary~\ref{lbv:besov}.
The continuity endpoint is also local: if $\mathcal B(\alpha)$
is finite, the solution is globally continuous; if it diverges,
the solution is essentially unbounded on every nonempty open arc.
Moreover, for every nonnegative continuous $\chi$,
\[
 \lim_{N\to\infty}N\int_\T\chi(x)|P_Nu_\varepsilon(t,x)|^2
       \frac{dx}{2\pi}
 \geq\frac{I_\varphi\varepsilon^2}{2\pi^2}
       \int_\T\chi(x)\,\frac{dx}{2\pi}.
\]
\end{corollary}

\begin{proof}
Combine \eqref{lp:uniformdensity} with the weak energy law and
the local H\"older and Besov criteria.  At every irrational time,
the chosen points supply a nonempty finite rotation class, so
Proposition~\ref{fcl:local} gives the endpoint assertion.
The construction is fixed
before the time and the arc are chosen, so no countable or
uncountable intersection of generic-time statements is involved.
\end{proof}

\subsubsection{An open dense class with the full local continuity dichotomy}

The preceding density bound controls the positive regularity scale.
For the continuity endpoint, a stronger open condition follows from
Lemma~\ref{gm:localdominant}.  It involves only the sizes of the
initial jumps and their rotation classes.

Fix $K\geq\max\{1,K_*\}$, where $K_*$ is the absolute constant
in that lemma.  For a real BV function $u$, an irrational rotation
angle $\theta/(2\pi)$, and a jump point $a$, set
\[
 R_K(u,\theta,a)=|[u]_a|
  -K\sum_{\substack{b\in\mathcal J(u),\ b\ne a\\
                    b-a\in\theta\Z\ ({\rm mod}\ 2\pi)}}|[u]_b|,
 \qquad
 \Gamma_K(u,\theta)=\max\{0,\sup_a R_K(u,\theta,a)\}.
\]
If there are no jumps, put $\Gamma_K=0$.  Define
\[
 m_K(u)=\inf_{\theta/(2\pi)\notin\mathbb Q}\Gamma_K(u,\theta),
 \qquad
 \mathcal V_K=\{u\in\operatorname{BV}_0(\T;\R):m_K(u)>0\}.
\]

\begin{proposition}[An open dense class with complete local classification]
\label{lp:fullgeneric}
The set $\mathcal V_K$ is open and dense in the real mean-zero
BV space.  In fact,
\begin{equation}\label{lp:marginlip}
 |m_K(u)-m_K(v)|\leq K\operatorname{Var}(u-v).
\end{equation}
For every $u_0\in\mathcal V_K$ and every irrational
$\alpha=t/(2\pi)$, the following alternative holds:
\begin{enumerate}
 \item if $\mathcal B(\alpha)<\infty$, then $S(t)u_0$ is continuous
       on the whole circle;
 \item if $\mathcal B(\alpha)=\infty$, then $S(t)u_0$ is essentially
       unbounded on every nonempty open arc.
\end{enumerate}
At the same times, every local positive H\"older and Besov condition
in Propositions~\ref{lbv:local} and \ref{bs:spectrum} has exactly its
global arithmetic criterion.  Quantitatively,
\[
 D_t^{u_0}(x)\geq\frac{m_K(u_0)^2}{4\pi^2}
 \qquad(x\in\T,\ t/(2\pi)\notin\mathbb Q).
\]
\end{proposition}

\begin{proof}
For two data, use the union of their jump sets and give absent
jumps coefficient zero.  Adding such candidates does not change
$\Gamma_K$, because their individual margins are nonpositive.
For every fixed $\theta$ and $a$, the difference between the two
margins is bounded by
\[
 |[u]_a-[v]_a|
   +K\sum_{b\ne a}|[u]_b-[v]_b|
 \leq K\operatorname{Var}(u-v).
\]
The sum may be taken over the whole union; the atomic projection
of $D(u-v)$ is a total-variation contraction.  Taking suprema,
positive parts, and then infima in $\theta$ proves
\eqref{lp:marginlip}, hence openness.

To prove density, take chosen points $a,b$ for the jump set of
an arbitrary background datum $u$ and put
\[
 u_\varepsilon=u+\varepsilon J_a+(K+2)\varepsilon J_b.
\]
This tends to $u$ in BV.  At each irrational rotation, a singleton among the two chosen points has margin at least $|\varepsilon|$.  If the chosen
class consists of the two new points, its larger jump has margin
$((K+2)-K)|\varepsilon|=2|\varepsilon|$.  Consequently
$m_K(u_\varepsilon)\geq|\varepsilon|$, proving density.

For $u_0\in\mathcal V_K$ and $\theta=2t$, choose any
$0<\delta<m_K(u_0)$.  There is an atom whose margin exceeds
$\delta$.  The corresponding genuine gauge coefficient therefore
dominates the total modulus of the other coefficients in its orbit
by the factor required in Lemma~\ref{gm:localdominant}.  The lemma
and the BO reconstruction give the local unboundedness assertion
when $\mathcal B(\alpha)=\infty$.  If the arithmetic sum is finite,
the global continuity theorem applies.

The reverse triangle inequality on this same orbit gives a lower
bound $\delta/(2\pi)$ for its orbit series, since $K\geq1$.
Hence $D_t^{u_0}\geq\delta^2/(4\pi^2)$.  Letting
$\delta\uparrow m_K(u_0)$ proves the displayed density bound.
The local H\"older and Besov criteria now follow from the local
energy theorems.
\end{proof}

\subsubsection{The exact meaning of an arc where the energy density vanishes}

For clarity, write $b^{1/2}_{2,\infty}$ for the little Besov space:
its elements belong to $B^{1/2}_{2,\infty}$ and satisfy
$2^{k/2}\|\Delta_k f\|_2\to0$ for a smooth dyadic partition.
Local membership is defined using smooth compactly supported cutoffs.

\begin{corollary}[The critical local energy defect]
\label{lp:criticaldefect}
Let $u_0$ be any real mean-zero BV datum and let $t/(2\pi)$ be
irrational.  For every nonempty open arc $I$,
\[
 D_t\equiv0\ \hbox{on }I
 \quad\Longleftrightarrow\quad
 S(t)u_0\in b^{1/2}_{2,\infty,\mathrm{loc}}(I).
\]
More precisely, choose the high-frequency blocks in the form
$\Delta_k=\psi(2^{-k}|D|)$, where $|D|$ is the multiplier $|n|$
and $\psi$ is real, smooth, and
compactly supported in $(0,\infty)$, as part of a dyadic partition.
Then, for every smooth $\eta$,
\begin{equation}\label{lp:cutoffenergy}
 \lim_{k\to\infty}2^k\|\Delta_k(\eta S(t)u_0)\|_2^2
 =2I_\psi\int_\T|\eta(x)|^2D_t(x)\,\frac{dx}{2\pi},
 \qquad I_\psi=\int_0^\infty\frac{\psi(s)^2}{s^2}\,ds.
\end{equation}
In particular, the support of $D_t$ is exactly the set where
the solution fails to have this local little-Besov regularity.
Every arc on which $D_t$ has positive mass also fails local
$H^{1/2}$ regularity.
\end{corollary}

\begin{proof}
The proof of Proposition~\ref{lbv:energy} applies without change to
any fixed smooth annular symbol $\psi$, replacing $I_\varphi$ by
$I_\psi$.  Its kernel estimate gives, for a smooth cutoff,
\[
 \|[\Delta_k,\eta]S(t)u_0\|_2
 \leq C_\eta2^{-k}\|S(t)u_0\|_2.
\]
Since $\|\Delta_k S(t)u_0\|_2=O(2^{-k/2})$, the commutator
does not affect the limit of the weighted squared norm.  Apply the
weak energy law with test function $|\eta|^2$ to obtain
\eqref{lp:cutoffenergy}.

The same estimates give local $B^{1/2}_{2,\infty}$ membership.
All normalized blocks of every cutoff supported in $I$ tend to
zero exactly when their displayed limits vanish.  Because $D_t$
is continuous and nonnegative, this is equivalent to $D_t=0$
on $I$.  If a cutoff has a strictly positive limit instead,
the squared normalized blocks are not summable, excluding
$H^{1/2}$ for that cutoff.
\end{proof}

\begin{remark}
Arbitrarily small, explicitly realizable perturbations give a
uniform positive lower bound for $D_t$ at all irrational times.
This lower bound makes the conclusion stable in the BV norm.
For a fixed background datum and a fixed nonzero
$\varepsilon$, almost every pair of new jump locations gives the
conclusion; this follows directly from the countable-line
exceptional set in Lemma~\ref{lp:geometry}.
\end{remark}

\begin{proof}[Completion of the proofs of Theorems~\ref{main:continuity}
and~\ref{main:regularity}]
Corollary~\ref{gm:BO} gives the global continuity criterion.
Corollary~\ref{fcl:finite-BV} gives the local assertion for finitely
many jumps. Fix any $K$ as in Proposition~\ref{lp:fullgeneric}
and take $\mathcal V=\mathcal V_K$. That proposition proves
openness, density, and the local conclusions at every irrational time.
The global H\"older criterion is Proposition~\ref{fe:classification}.

For the graph statement, a BV function with a nonzero jump cannot
belong to $H^{1/2}$. Indeed, on sufficiently short intervals on
opposite sides of the jump its values differ by a fixed positive
amount. The corresponding part of the difference-quotient integral
is bounded below by a positive multiple of
$\int_0^\delta\int_0^\delta(x+y)^{-2}\,dx\,dy=\infty$.
Corollary~\ref{ar:fractal} therefore applies.
\end{proof}

\section{Singularities at rational times}\label{sec:rational-times}

At a rational time, Fourier inversion on a finite cyclic group gives
an exact formula for the linear gauge profile. We first record this
formula for arbitrary BV data. A finite gauge expansion then yields
ordinary logarithmic and jump singularities. The final part of the
section compares their coefficients with the linear BO evolution and
extends that comparison to arbitrary BV data through Poisson means.

\subsection{The rational gauge representation}\label{sec:BV}
Let $t=2\pi p/q$, where $p\in\Z$, $q\geq1$, and $(p,q)=1$.
Set
\begin{equation}\label{pre:gauss}
 d_\ell=\frac1q\sum_{m=0}^{q-1}
      e^{2\pi i(pm^2+\ell m)/q},\qquad 0\leq\ell<q.
\end{equation}
Finite Fourier inversion gives
\begin{equation}\label{pre:finiteinversion}
 e^{itn^2}=\sum_{\ell=0}^{q-1}d_\ell e^{-2\pi i\ell n/q},
 \qquad n\in\Z.
\end{equation}

\begin{proposition}[Rational representation for BV data]
\label{prop:BVrepresentation}
Let $u_0\in BV_{r,0}(\T)$, and set
$h_0=(e^{-i\partial_x^{-1}u_0})'=-iu_0e^{-i\partial_x^{-1}u_0}$.
At the rational time above,
\begin{equation}\label{intro-BV-rep}
 u(t,x)=2\operatorname{Re}\left(
 i a_t(x)e^{-itM}\sum_{\ell=0}^{q-1}d_\ell
 (\Pi_{>0}h_0)(x-2\pi\ell/q)\right)+r_t(x),
 \qquad r_t\in C^0(\T).
\end{equation}
The identity holds in $L^2$. No finiteness assumption on the jump
set is required.
\end{proposition}
\begin{proof}
The Lipschitz function $e^{-i\partial_x^{-1}u_0}$ has BV
derivative $h_0$, and $w_0=\Pi_{>0}h_0$ by definition of the
gauge. Insert \eqref{pre:finiteinversion} in the Fourier series of
$w_L(t)$, then apply \eqref{eq:GKT}. The remainder is continuous
because $u_0\in H^{1/2-}$.
\end{proof}

\subsection{A finite singular part}\label{sec:finite-singularities}
For $\xi\in\T$, let
\[
 L_\xi(x)=\sum_{n\geq1}\frac{e^{in(x-\xi)}}n
          =\mathcal C_\xi(x)+i\mathcal J_\xi(x)
\]
with convergence in $L^2$, where
\begin{equation}\label{pre:kernels}
 \mathcal C_\xi(x)=-\log\left|2\sin\frac{x-\xi}{2}\right|,
 \qquad \mathcal J_\xi(x)=\frac{\pi-y}2,
 \quad 0<y<2\pi,\quad y\equiv x-\xi\pmod{2\pi}.
\end{equation}
Equivalently, $L_\xi$ is the radial boundary value of the
principal logarithm $-\log(1-re^{i(x-\xi)})$ as $r\uparrow1$.
The kernel $\mathcal J_\xi$ has jump $\pi$ at $\xi$.

\begin{lemma}[Multiplication by a H\"older function]
\label{thm:holder-transfer}
If $a\in C^\rho(\T)$, $0<\rho\leq1$, then
$(a(x)-a(\xi))L_\xi(x)$ extends continuously across $\xi$.
Consequently, for a finite sum
$v=\sum_\xi\Gamma_\xi L_\xi+C$ with $C\in C^0$,
\[
 av=\sum_\xi a(\xi)\Gamma_\xi L_\xi+C_a,
 \qquad C_a\in C^0.
\]
\end{lemma}
\begin{proof}
Near $\xi$ the product is bounded by
$C d(x,\xi)^\rho(1+|\log d(x,\xi)|)$, which tends to zero.
Away from the finitely many centers every term is continuous.
\end{proof}

\begin{proposition}[Finite rational-time representation]
\label{lem:finite}
Suppose \eqref{eq:FE} holds. At a rational time $t=2\pi p/q$,
combine equal centers $\xi=a_j+2\pi\ell/q$ and set
\[
 \Gamma_\xi=e^{-itM}
       \sum_{a_j+2\pi\ell/q=\xi}c_jd_\ell.
\]
Then
\begin{equation}\label{eq:finite}
 u(t,x)=2\operatorname{Re}\sum_\xi
          i a_t(\xi)\Gamma_\xi L_\xi(x)+R_t(x),
 \qquad R_t\in C^0(\T).
\end{equation}
The coefficients of $\mathcal C_\xi$ and $\mathcal J_\xi$ are
\begin{equation}\label{eq:realcoeff}
 A_\xi=2\operatorname{Re}(i a_t(\xi)\Gamma_\xi),\qquad
 B_\xi=-2\operatorname{Im}(i a_t(\xi)\Gamma_\xi),\qquad
 A_\xi^2+B_\xi^2=4|\Gamma_\xi|^2.
\end{equation}
If at least one $c_j$ is nonzero, the solution has no continuous
representative at any rational time.
\end{proposition}

We refer to $B_\xi$ as the jump coefficient, meaning the
coefficient of $\mathcal J_\xi$, not the jump height.
When $A_\xi=0$, the actual jump height is $\pi B_\xi$.
When $A_\xi\ne0$, this is the jump of the bounded singular part
after subtracting $A_\xi\mathcal C_\xi$.

\begin{proof}
Proposition~\ref{prop:auto-regularity} allows us to apply
\eqref{eq:GKT} with a continuous remainder and a H\"older
multiplier. The error series in \eqref{eq:FE} evolves in
$H^{1/2+\eps}$, hence continuously in space. Inserting
\eqref{pre:finiteinversion} in the finite part gives
$w_L(t)=\sum_\xi\Gamma_\xi L_\xi+C_t$ with $C_t$ continuous.
Lemma~\ref{thm:holder-transfer} proves \eqref{eq:finite}.
The coefficient identity follows from $|a_t|=1$.

It remains to show that not all collected coefficients vanish.
For every integer $n\geq1$, Fourier inversion gives
\[
 \sum_\xi\Gamma_\xi e^{-in\xi}
 =e^{-itM}e^{itn^2}\sum_{j=1}^J c_je^{-ina_j}.
\]
If the left side vanished for every $n$, taking $n=1,\ldots,J$
would give a homogeneous Vandermonde system for the $c_j$.
Its determinant is
$(\prod_jz_j)\prod_{j<k}(z_k-z_j)$, where $z_j=e^{-ia_j}$,
and is nonzero because the $a_j$ are distinct. Thus some
$\Gamma_\xi\ne0$. At this point either $A_\xi\ne0$, giving
logarithmic divergence, or $B_\xi\ne0$, giving a nonzero jump.
All other terms are continuous in a neighborhood. Neither
obstruction can be removed by changing a function on a null set.
\end{proof}

In particular, Lemma~\ref{lem:piecewise-tail} proves the last
assertion of Theorem~\ref{main:rational}. Formula
\eqref{intro-BV-rep} remains the appropriate representation when
the derivative has infinitely many atoms or an atomless singular part;
in that case a finite singular expansion need not exist.
\subsection{Separation at large denominators}
\label{rs:section}

We next prove the quantitative part of Theorem~\ref{main:rational}. For
every fixed piecewise smooth datum with a jump, nonlinear and linear
logarithmic coefficients differ at every rational time whose reduced
denominator is sufficiently large. The threshold is uniform over the
numerator and hence over all real times. No smallness condition on the
initial datum is imposed.

Let $f$ be real, mean zero, piecewise $C^{1,\eta}$ on a finite partition,
$0<\eta\leq1$, and suppose that it has at least one nonzero jump.
Write its distinct genuine jump points as $a_1,\ldots,a_J$, and set
\begin{equation}\label{rs:data}
 b_j=-\frac{[f]_{a_j}}{2\pi},\qquad
 c_j=b_j e^{-iV_0(a_j)},\qquad V_0=\partial_x^{-1}f,
 \qquad M=\|f\|_2^2,
\end{equation}
\[
 B=\sum_j b_j^2>0,\qquad
 C=\sum_j b_j^2e^{-iV_0(a_j)}.
\]
Choose a positive integer $D$ divisible by the reduced denominator of
every rational number among $(a_j-a_k)/(2\pi)$, using any fixed real
representatives of the jump points. If there are no nonzero rational
differences, take $D=1$. Write
$a_t=e^{i\partial_x^{-1}S(t)f}$ and
$\langle a_t\rangle{}=(2\pi)^{-1}\int_\T a_t(x)\,dx$.
Angle brackets here denote the spatial mean; complex conjugation is written with a bar.

\begin{proposition}[Uniform separation of logarithmic coefficients]
\label{rs:main}
Let $t=2\pi p/q$, where $(p,q)=1$ and $q>0$. On the union of all
translated jump centers, denote the nonlinear and linear logarithmic
coefficients by $A_\xi^{\rm nl}(t)$ and $A_\xi^{\rm lin}(t)$,
respectively, assigning zero to an absent coefficient. Then
\begin{equation}\label{rs:variance}
 \sum_\xi\big(A_\xi^{\rm nl}(t)-A_\xi^{\rm lin}(t)\big)^2
 =4B-4\operatorname{Re}\big(e^{-itM}\langle a_t\rangle{}C\big)
       +O_f(q^{-1/2}).
\end{equation}
The error is uniform in $p$ and $q$. Moreover,
\begin{equation}\label{rs:gap}
 \kappa_f:=1-\sup_{t\in\R}|\langle a_t\rangle{}|>0.
\end{equation}
Consequently, there is an integer $Q_f$ such that, whenever $q>Q_f$,
\begin{equation}\label{rs:positive}
 \sum_\xi\big(A_\xi^{\rm nl}(t)-A_\xi^{\rm lin}(t)\big)^2
 \geq 2\kappa_f B>0.
\end{equation}
In particular, $S(t)f-E(t)f$ has a nonzero logarithmic singularity
at some center, and has no bounded or continuous representative on
any neighborhood of that center.

The same formula \eqref{rs:variance}, with the same leading term,
holds for the squared differences of the jump coefficients. In addition,
\begin{equation}\label{rs:individual}
 \sum_\xi(A_\xi^{\rm nl}(t))^2=2B+O_f(q^{-1/2}),
 \qquad
 \sum_\xi(A_\xi^{\rm lin}(t))^2=2B+O_f(q^{-1/2}).
\end{equation}
\end{proposition}

Thus, for each datum in the proposition, rational times at which the
nonlinear-minus-linear profile is locally bounded at every rational
center can have only bounded reduced denominators. Such times form a
locally finite subset of $\R$. In particular, this conclusion applies
to every nonzero square-wave amplitude, including all even integer
amplitudes.

We prove the mean-phase gap first. The compactness used here is a
property of the actual BO orbit, not a compactness assumption on a
family of possible gauge phases.

\begin{lemma}[A uniform gap from constant gauge phases]\label{rs:compact}
For every nonzero mean-zero datum $f\in H^s(\T)$ with $s>0$,
the quantity in \eqref{rs:gap} is strictly positive.
\end{lemma}
\begin{proof}
By \cite[Corollary~A.4]{GKT-wp}, the orbit of $f$ is bounded in
$H^s$ for all real times. Rellich compactness therefore makes its
closure $K$ compact in $L^2$. Conservation of the norm gives
$\|v\|_2=\|f\|_2>0$ for every $v\in K$.
The map
\[
 v\longmapsto\frac1{2\pi}\int_\T e^{i\partial_x^{-1}v(x)}\,dx
\]
is continuous on $L^2_0$, since $\partial_x^{-1}:L^2_0\to C^0$
is bounded. If its modulus were one at a member $v$ of $K$, equality
in the triangle inequality would make
$e^{i\partial_x^{-1}v}$ constant. Its real, continuous primitive
would then be constant as well, because a continuous lift of a fixed
circle value has connected image in a discrete set. Its mean is zero,
so the primitive and $v$ would both vanish. This contradicts the
conserved positive norm. The maximum on the compact set $K$ is
therefore strictly smaller than one.
\end{proof}

For the remaining argument write $e_q(z)=e^{2\pi iz/q}$ and
\[
 d_\ell=\frac1q\sum_{n=0}^{q-1}e_q(pn^2+\ell n),
 \qquad \ell\in\Z/q\Z.
\]
For a center $\xi$, let
\[
 \Gamma_\xi=e^{-itM}
       \sum_{a_j+2\pi\ell/q=\xi}c_jd_\ell,
 \qquad
 \Lambda_\xi=
       \sum_{a_j+2\pi\ell/q=\xi}b_jd_\ell.
\]
All center equalities are on the circle. The finite representation gives
\[
 A_\xi^{\rm nl}=-2\operatorname{Im}(a_t(\xi)\Gamma_\xi),
 \qquad A_\xi^{\rm lin}=-2\operatorname{Im}\Lambda_\xi.
\]

\begin{lemma}[Gauss coefficients with collisions]\label{rs:gauss}
The following identities and bounds hold for every coprime pair $(p,q)$.
First,
\begin{equation}\label{rs:unitary}
 \sum_\xi|\Gamma_\xi|^2=
 \sum_\xi|\Lambda_\xi|^2=B,
 \qquad
 \sum_\xi\Gamma_\xi\overline{\Lambda_\xi}=e^{-itM}C.
\end{equation}
Second, for every function $h$ in the Wiener algebra
$\mathcal A(\T)$,
\begin{equation}\label{rs:nonconjugated}
 \left|\sum_\xi h(\xi)\Gamma_\xi^2\right|
 +\left|\sum_\xi h(\xi)\Gamma_\xi\Lambda_\xi\right|
 +\left|\sum_\xi h(\xi)\Lambda_\xi^2\right|
 \leq\frac{C_f}{\sqrt q}\|h\|_{\mathcal A}.
\end{equation}
Finally,
\begin{equation}\label{rs:conjugated}
 \sum_\xi a_t(\xi)\Gamma_\xi\overline{\Lambda_\xi}
 =\langle a_t\rangle{}e^{-itM}C+O_f(q^{-1/2}),
\end{equation}
uniformly over all rational times.
\end{lemma}
\begin{proof}
Partition the original jump points into cosets of
$(2\pi/q)\Z$ in $\T$. Within each coset, place the coefficients
$c_j$ or $b_j$ in a vector of length $q$, assigning zero to unused
positions. Convolution by $(d_\ell)$ is unitary on this vector space:
its finite Fourier multiplier is $e_q(pn^2)$ and has modulus one.
The norm and inner-product identities on each coset give
\eqref{rs:unitary}, including all collisions and all parity cases.

For integers $r,s,k$, finite Fourier orthogonality gives
\begin{equation}\label{rs:productDFT}
 \sum_{\ell=0}^{q-1}d_{\ell+r}d_{\ell+s}e_q(k\ell)
 =\frac{e_q(pk^2-sk)}q
     \sum_{n=0}^{q-1}e_q\bigl(2pn^2+(2pk+r-s)n\bigr).
\end{equation}
The squared-modulus bound for a complete quadratic sum is
\[
 \left|\sum_{n=0}^{q-1}e_q(An^2+Bn)\right|^2
 \leq q\gcd(2A,q).
\]
Since $\gcd(4p,q)\leq4$, the left side of
\eqref{rs:productDFT} has modulus at most $2/\sqrt q$.
Expanding $h$ in its absolutely convergent Fourier series therefore
gives the same bound times $\|h\|_{\mathcal A}$ on each translated
grid. Expand the products of the collected coefficients in
\eqref{rs:nonconjugated}. A pair of original jumps contributes only
if its grids intersect, in which case this is precisely a product
of two shifted $d$ sequences on one complete grid. Summing the
finitely many pairs proves \eqref{rs:nonconjugated}; one may take
$C_f=6(\sum_j|b_j|)^2$.

For the conjugated product, the elementary shift identity
\begin{equation}\label{rs:heisenberg}
 d_{\ell+2pD}=e_q(-pD^2-\ell D)d_\ell
\end{equation}
is obtained by shifting the summation variable by $D$.
Consider one of the preceding cosets, choosing one of its original
jump points as the reference point $a_*$.
Every jump in it has the form $a_j=a_*+2\pi k_j/q$ and satisfies
$Dk_j/q\in\Z$, by the choice of $D$.
At the center $a_*+2\pi\ell/q$, the summands in $\Gamma$ and
$\Lambda$ are $c_jd_{\ell-k_j}$ and $b_jd_{\ell-k_j}$.
Equation \eqref{rs:heisenberg} multiplies all these summands by
the same unit factor under $\ell\mapsto\ell+2pD$, because
$e_q(Dk_j)=1$. Hence the product
$\Gamma_\xi\overline{\Lambda_\xi}$ is constant along each
orbit of this shift.

Put $g=\gcd(2D,q)$. Each such orbit contains $q/g$ points and
is an equally spaced spatial grid. Norm conservation gives, for
the periodic distance $d(x,y)$,
\[
 |a_t(x)-a_t(y)|
 \leq |V_t(x)-V_t(y)|
 \leq\sqrt{2\pi M\,d(x,y)}.
\]
Thus the average of $a_t$ on any of these grids differs from its
spatial mean by at most $C\sqrt{Mg/q}$, uniformly in time and
in the grid's translation. This follows by comparing each grid
value with the integral over its adjacent cell. Since $g\leq2D$,
\[
 \left|\sum_\xi
    (a_t(\xi)-\langle a_t\rangle{})\Gamma_\xi\overline{\Lambda_\xi}\right|
 \leq C\sqrt{MD/q}\sum_\xi|\Gamma_\xi||\Lambda_\xi|
 \leq C B\sqrt{MD/q}.
\]
Here the last step is Cauchy--Schwarz and \eqref{rs:unitary}.
Combining this estimate with the inner-product identity in
\eqref{rs:unitary} proves \eqref{rs:conjugated}.
\end{proof}

\begin{proof}[Proof of Proposition~\ref{rs:main}]
Set $Z_\xi=a_t(\xi)\Gamma_\xi-\Lambda_\xi$. By
\eqref{rs:unitary} and \eqref{rs:conjugated},
\[
 \sum_\xi|Z_\xi|^2
 =2B-2\operatorname{Re}\bigl(e^{-itM}\langle a_t\rangle{}C\bigr)
       +O_f(q^{-1/2}).
\]
For $m=0,1,2$, Fourier Cauchy--Schwarz and
$\partial_x(a_t^m)=im(S(t)f)a_t^m$ show that
\[
 \|a_t^m\|_{\mathcal A}
 \leq\left(\sum_{n\in\Z}(1+n^2)^{-1}\right)^{1/2}
        (1+m^2M)^{1/2}.
\]
These bounds are uniform in time. Applying
\eqref{rs:nonconjugated} to $h=a_t^2,a_t,1$ gives
$\sum_\xi Z_\xi^2=O_f(q^{-1/2})$.
The identity
\[
 4(\operatorname{Im}Z_\xi)^2
   =2|Z_\xi|^2-2\operatorname{Re}(Z_\xi^2)
\]
now proves \eqref{rs:variance}. For the jump coefficients the
corresponding identity has the opposite sign before
$2\operatorname{Re}(Z_\xi^2)$, so its leading term is the same.
Applying the same calculation separately to $a_t\Gamma$ and
$\Lambda$ proves \eqref{rs:individual}.

The datum belongs to $H^s$ for every $s<1/2$, so
Lemma~\ref{rs:compact} gives $\kappa_f>0$. Since $|C|\leq B$,
the leading term in \eqref{rs:variance} is at least $4\kappa_f B$.
Absorb the uniform error by increasing $Q_f$. This proves
\eqref{rs:positive}. The logarithmic representation modulo a
continuous function then gives the asserted local unboundedness.
\end{proof}

The conclusion concerns the rational-time singularities of the actual
real solution. The squared real coefficients are evaluated explicitly;
no inference is made from the size of a complex gauge profile alone.

The calculation also describes how the logarithmic energy is distributed
in space. In the next statement, all constants may depend on the fixed
datum and, in the local conclusions, on the fixed open arc. None depends
on the numerator of the rational time.

\begin{corollary}[Spatial distribution and number of logarithmic centers]
\label{rs:local}
For every real $\chi\in C^\infty(\T)$,
\begin{align}
 \sum_\xi\chi(\xi)(A_\xi^{\rm nl})^2
 &=2B\langle\chi\rangle+O_{f,\chi}(q^{-1/2}),
 \label{rs:weightedactual}\\
 \sum_\xi\chi(\xi)(A_\xi^{\rm nl}-A_\xi^{\rm lin})^2
 &=4B\langle\chi\rangle
 -4\operatorname{Re}\!\left(e^{-itM}C\langle\chi a_t\rangle\right)
 +O_{f,\chi}(q^{-1/2}),
 \label{rs:weighteddifference}
\end{align}
where $\langle h\rangle=(2\pi)^{-1}\int_\T h$.
There are constants $c_f>0$ and $Q_f$ such that every reduced
rational time with $q>Q_f$ has at least $c_f q$ distinct centers
satisfying
\[
 |A_\xi^{\rm nl}-A_\xi^{\rm lin}|\geq c_fq^{-1/2}.
\]
For every nonempty open arc $I$, there are constants $c_{f,I}>0$
and $Q_{f,I}$ such that, whenever $q>Q_{f,I}$, at least
$c_{f,I}q$ centers in $I$ satisfy
\[
 |A_\xi^{\rm nl}|\geq c_{f,I}q^{-1/2}.
\]
If $|C|<B$, the same local conclusion holds with $A_\xi^{\rm nl}$
replaced by $A_\xi^{\rm nl}-A_\xi^{\rm lin}$.
\end{corollary}
\begin{proof}
The shift in \eqref{rs:heisenberg} also leaves $|\Gamma|^2$ and
$|\Lambda|^2$ unchanged. Averaging $\chi$ over its shift orbits,
exactly as in the proof of \eqref{rs:conjugated}, gives
\[
 \sum_\xi\chi(\xi)|\Gamma_\xi|^2
 =\sum_\xi\chi(\xi)|\Lambda_\xi|^2
     +O_{f,\chi}(q^{-1/2})
 =B\langle\chi\rangle+O_{f,\chi}(q^{-1/2}).
\]
In fact the two separate errors here are $O_{f,\chi}(q^{-1})$;
the weaker bound suffices. Applied to $\chi a_t$, the same averaging
argument gives
\[
 \sum_\xi\chi(\xi)a_t(\xi)\Gamma_\xi\overline{\Lambda_\xi}
 =e^{-itM}C\langle\chi a_t\rangle+O_{f,\chi}(q^{-1/2}).
\]
Multiplication by the fixed smooth function $\chi$ preserves the
uniform Wiener bounds used above. The identities for the squares
of real and imaginary parts therefore give
\eqref{rs:weightedactual}--\eqref{rs:weighteddifference}.

Put $L=\sum_j|b_j|$. The complete Gauss sum bound gives
$|d_\ell|\leq\sqrt{2/q}$, and hence
\[
 |A_\xi^{\rm nl}|\leq\frac{2\sqrt2L}{\sqrt q},\qquad
 |A_\xi^{\rm nl}-A_\xi^{\rm lin}|
 \leq\frac{4\sqrt2L}{\sqrt q}.
\]
There are at most $Jq$ distinct centers. A positive lower bound
for the total squared coefficients, together with these upper bounds,
forces a positive multiple of $q$ coefficients to be at least a
positive multiple of $q^{-1/2}$: choose the latter multiple small
enough that all coefficients below that threshold contribute less
than half the lower bound. Formula \eqref{rs:positive} proves the
global assertion this way.

For the local assertions choose $0\leq\chi\leq1$ supported in $I$
with $\langle\chi\rangle>0$. Equation \eqref{rs:weightedactual}
gives a positive lower bound for the squared nonlinear logarithmic
coefficients in $I$. If $|C|<B$, then
\[
 4B\langle\chi\rangle
 -4\operatorname{Re}(e^{-itM}C\langle\chi a_t\rangle)
 \geq4(B-|C|)\langle\chi\rangle>0.
\]
The preceding counting argument applies in either case.
\end{proof}

For $f=\lambda s$ with $\lambda\ne0$, the constants simplify to
\[
 B=\frac{2\lambda^2}{\pi^2},\qquad
 C=\frac{2\lambda^2}{\pi^2}\cos\frac{\pi\lambda}{2},\qquad
 M=\lambda^2.
\]
Thus every nonzero square-wave amplitude has logarithmic separation
at a positive multiple of $q$ centers for all sufficiently large
denominators. If $\lambda$ is not an even integer, this separation
occurs at a positive multiple of $q$ centers in every fixed open arc,
once $q$ is sufficiently large for that arc. The latter conclusion
is not asserted for even integer amplitudes.

\subsubsection{Arbitrary BV data and logarithmic Poisson limits}

The finite-jump hypothesis is unnecessary if logarithmic coefficients
are defined through Poisson limits. This formulation allows a dense
countable jump set and an arbitrary atomless part of the derivative;
it does not assert an isolated logarithmic expansion at such points.
Write $\mathsf P_r$ for the Poisson operator on $\T$, with Fourier multiplier
$r^{|n|}$, and define the logarithmic Poisson coefficient of a function
$v\in L^2$ at $\xi$, whenever the limit exists, by
\[
 \mathscr L_\xi(v)
 =\lim_{r\uparrow1}\frac{(\mathsf P_rv)(\xi)}{\log(1/(1-r))}.
\]
For a finite logarithmic expansion this is exactly the coefficient
of $\mathcal C_\xi=-\log|2\sin((x-\xi)/2)|$ used above.

\begin{proposition}[Uniform rational separation for BV data]\label{rs:bv}
Let $f$ be real, mean zero, and of bounded variation, with at least
one nonzero jump. Enumerate its finite or countable jump set by $a_j$,
define $b_j,c_j,M$ by \eqref{rs:data}, and set
\[
 B=\sum_jb_j^2>0,\qquad C=\sum_jb_j^2e^{-iV_0(a_j)}.
\]
These series converge absolutely.
At every reduced rational time $t=2\pi p/q$, all logarithmic Poisson
coefficients of $S(t)f$ and $E(t)f$ exist. They vanish outside the
countable set
\[
 \mathcal X_q=\{a_j+2\pi\ell/q:j,\ 0\leq\ell<q\}.
\]
Denote them on this set by $A_\xi^{\rm nl}$ and $A_\xi^{\rm lin}$.
Uniformly over all coprime numerators $p$,
\begin{align}
 \sum_{\xi\in\mathcal X_q}(A_\xi^{\rm nl})^2
 &=\sum_{\xi\in\mathcal X_q}(A_\xi^{\rm lin})^2
       +O_f(q^{-1/2})
 =2B+O_f(q^{-1/2}),\label{rs:bvindividual}\\
 \sum_{\xi\in\mathcal X_q}
       (A_\xi^{\rm nl}-A_\xi^{\rm lin})^2
 &=4B-4\operatorname{Re}(e^{-itM}\langle a_t\rangle{}C)+o_f(1).
 \label{rs:bvvariance}
\end{align}
The gap $\kappa_f$ in \eqref{rs:gap} is positive. Consequently
there exist $c_f>0$ and $Q_f$ such that, for every $q>Q_f$,
at least $c_fq$ distinct centers satisfy
\[
 |\mathscr L_\xi(S(t)f-E(t)f)|\geq c_fq^{-1/2}.
\]
At each such center the difference is essentially unbounded in
every neighborhood.

For every real smooth $\chi$, the weighted identities
\eqref{rs:weightedactual}--\eqref{rs:weighteddifference} remain
valid with errors $o_{f,\chi}(1)$. Thus on every fixed nonempty
open arc, at every sufficiently large reduced denominator, the
actual solution has at least $c_{f,I}q$ distinct centers with
$|\mathscr L_\xi(S(t)f)|\geq c_{f,I}q^{-1/2}$.
If $|C|<B$, the same local assertion holds for $S(t)f-E(t)f$.
All thresholds are independent of the numerator.
\end{proposition}

\begin{proof}
Since $f\in BV$, the functions $f$ and $g=f e^{-iV_0}$ are
bounded and of bounded variation. Introduce the finite complex
measure and the finite real measure
\[
 \nu=-\frac1{2\pi}Dg,\qquad \nu^{\rm lin}=-\frac1{2\pi}Df.
\]
Their atoms at $a_j$ have masses $c_j$ and $b_j$. With the measure
Fourier transform defined as $\widehat\nu(n)=\int e^{-inx}\,d\nu(x)$,
the initial positive gauge has coefficients $\widehat\nu(n)/n$,
$n\geq1$. At the rational time, its linear gauge evolution is
therefore the logarithmic potential of
\[
 \nu_t=e^{-itM}\sum_{\ell=0}^{q-1}d_\ell
                   (\tau_{2\pi\ell/q})_*\nu,
 \qquad \tau_y(x)=x+y.
\]
The linear BO profile uses the same construction with
$\nu^{\rm lin}$ and without the scalar factor $e^{-itM}$.
Their atomic masses are precisely $\Gamma_\xi$ and
$\Lambda_\xi$ defined above, with all sums now countable.

Here is the elementary limit that identifies the actual real
coefficients. Let
\[
 L_r(z)=-\log(1-re^{iz})=\sum_{n\geq1}\frac{r^ne^{inz}}n.
\]
For every finite complex measure $\mu$ and every $\xi$,
\begin{equation}\label{rs:atomlimit}
 \lim_{r\uparrow1}
 \frac{\int L_r(\xi-y)\,d\mu(y)}{\log(1/(1-r))}
 =\mu(\{\xi\}).
\end{equation}
Indeed, the normalized integrand tends to one at $y=\xi$ and
to zero elsewhere, and is uniformly bounded for $r\geq1/2$.
Dominated convergence for the total variation measure proves
\eqref{rs:atomlimit}, including an arbitrary atomless part.

The BO gauge reconstruction gives
$S(t)f=2\operatorname{Re}(ia_tw_L)+r_t$, where $r_t$ is
continuous. For some $1/2<\rho<1$, one has $a_t\in C^\rho$,
and the Poisson commutator estimate gives
\[
 \|\mathsf P_r(a_tw_L)-a_t\mathsf P_rw_L\|_\infty
 \leq C_\rho(1-r)^{\rho-1/2}
                  \|a_t\|_{C^\rho}\|w_L\|_2.
\]
This estimate follows directly by applying Cauchy--Schwarz to
the Poisson kernel times $|a_t(x-y)-a_t(x)|$.
It tends to zero. Hence \eqref{rs:atomlimit} yields, at every
point,
\[
 A_\xi^{\rm nl}=-2\operatorname{Im}(a_t(\xi)\Gamma_\xi),
 \qquad A_\xi^{\rm lin}=-2\operatorname{Im}\Lambda_\xi.
\]
If a nonzero coefficient existed at a point where the real profile
were essentially bounded in some neighborhood, its Poisson means
would remain bounded there: the separated complement contributes
$O(1-r)$, by global $L^1$ integrability. This proves the asserted
local unboundedness.

The jump coefficients satisfy $L:=\sum_j|b_j|<\infty$.
Partition their countable support into cosets of the finite group
$(2\pi/q)\Z$. On each coset the same finite unitary convolution
as before applies, and summing its identities gives
\eqref{rs:unitary} without change. Absolute convergence of the
double sums in the Gauss-product calculation also gives
\eqref{rs:nonconjugated}, with the bound $6L^2q^{-1/2}$.
In particular the individual unweighted variances retain the
error $O_f(q^{-1/2})$ in \eqref{rs:bvindividual}.

For the conjugated and spatially weighted expressions, truncate
the atomic sequences after $N$ terms. If the resulting vectors
are $\Gamma^{(N)},\Lambda^{(N)}$, unitarity gives
\[
 \|\Gamma-\Gamma^{(N)}\|_{\ell^2}
 =\|\Lambda-\Lambda^{(N)}\|_{\ell^2}
 =\left(\sum_{j>N}b_j^2\right)^{1/2}.
\]
Thus, for any bounded weight $h$,
\[
 \left|\sum_\xi h(\xi)
  \big(\Gamma_\xi\overline{\Lambda_\xi}
    -\Gamma_\xi^{(N)}\overline{\Lambda_\xi^{(N)}}\big)\right|
 \leq2\|h\|_\infty\sqrt B
                 \left(\sum_{j>N}b_j^2\right)^{1/2}.
\]
The analogous estimate holds for $|\Gamma|^2$ and $|\Lambda|^2$.
For each fixed $N$, choose the integer $D_N$ for those finitely
many jump locations and apply the earlier orbit-averaging argument
to the full solution phase $a_t$. Its modulus of continuity remains
uniform in time by norm conservation. First choosing $N$ large and
then $q$ large proves
\[
 \sum_\xi\chi(\xi)|\Gamma_\xi|^2
 =B\langle\chi\rangle+o_{f,\chi}(1),\qquad
 \sum_\xi\chi(\xi)a_t(\xi)
             \Gamma_\xi\overline{\Lambda_\xi}
 =e^{-itM}C\langle\chi a_t\rangle+o_{f,\chi}(1),
\]
uniformly in all numerators. The same norm identity holds for
$\Lambda$. Combining these statements with the nonconjugated
Gauss-product bound proves \eqref{rs:bvvariance} and the weighted
identities. Every BV datum belongs to $H^s$ for $s<1/2$;
Lemma~\ref{rs:compact} therefore proves $\kappa_f>0$.

Finally, the pointwise Gauss bounds used in Corollary~\ref{rs:local}
remain valid. Although the set of centers may now be infinite,
Cauchy--Schwarz gives $\sum_\ell|d_\ell|\leq\sqrt q$, so
\[
 \sum_\xi|A_\xi^{\rm nl}|\leq2L\sqrt q,
 \qquad
 \sum_\xi|A_\xi^{\rm nl}-A_\xi^{\rm lin}|\leq4L\sqrt q.
\]
Coefficients of size less than $\delta/\sqrt q$ contribute at
most $4L\delta$ to the total squared difference. Choose $\delta$
smaller than a fixed fraction of its positive energy lower bound.
The remaining coefficients each have square at most $32L^2/q$,
so there must be a positive multiple of $q$ of them. The same
argument with a nonnegative smooth cutoff supported in $I$ proves
the two local counting assertions.
\end{proof}

\begin{proof}[Completion of the proof of Theorem~\ref{main:rational}]
Proposition~\ref{rs:bv} supplies all the Poisson limits and the
uniform count of nonzero coefficients. Its local unboundedness
conclusion gives the first part of the theorem. For piecewise
$C^{1,\eta}$ data, Lemma~\ref{lem:piecewise-tail} and
Proposition~\ref{lem:finite} give the finite singular representation
and its nonvanishing at every rational time.
\end{proof}

\section{The square wave}\label{sec:square}

\subsection{The square wave and the one-sided parameter}\label{sec:square-parameter}

We begin by locating every rational-time singularity of the unit square wave.
The calculation also explains the compressed parameter in its quadratic
profile and the singular grids in the numerical comparison below.
Its mean-zero primitive and gauge factor are
\[
  V_0(x)=|x|-\frac\pi2,
  \qquad
  g_0(x)={e}^{-{i} V_0(x)}={i}{e}^{-{i}|x|}.
\]
Since \(g_0\) is even,
\[
  \widehat g_0(n)=\frac{{i}}{\pi}
  \int_0^\pi{e}^{-{i} x}\cos(nx)\,\mathrm{d} x.
\]
Writing \(\cos(nx)=({e}^{{i} nx}+{e}^{-{i} nx})/2\) and using
\(\widehat{\mathcal G(s)}(n)={i} n\widehat g_0(n)\), we obtain
\begin{equation}\label{eq:squaregauge}
  \widehat{\mathcal G(s)}(1)=-\frac12,
  \qquad
  \widehat{\mathcal G(s)}(n)=
  \begin{cases}
    -\displaystyle\frac{2{i} n}{\pi(n^2-1)},
      & n\ge2\ \text{even},\\[6pt]
    0, & n\ge3\ \text{odd}.
  \end{cases}
\end{equation}
In particular,
\begin{equation}\label{square-gauge-tail}
  \widehat{\mathcal G(s)}(n)
  =-\frac{{i}(1+(-1)^n)}{\pi n}+O(n^{-3}).
\end{equation}
The leading term contains only even modes. The \(n=1\) term and the
summable remainder contribute a continuous function to \(w_L(t)\).
Since \(\|s\|_2^2=1\), it follows that
\begin{align*}
  w_L(t,x)
  &=-\frac{{i}{e}^{-{i} t}}{\pi}\sum_{n\ge1}
    \frac{(1+(-1)^n){e}^{{i} tn^2}{e}^{{i} nx}}{n}+C_t(x)\\
  &=-\frac{{i}{e}^{-{i} t}}{\pi}\left[
    F_{t/(2\pi)}\left(\frac{x}{2\pi}\right)
    +F_{t/(2\pi)}\left(\frac{x+\pi}{2\pi}\right)
  \right]+C_t(x),
  \qquad C_t\in C^0(\T).
\end{align*}
Setting \(n=2m\) in the first sum gives the equivalent form
\begin{equation}\label{square-even-profile}
  w_L(t,x)=-\frac{{i}{e}^{-{i} t}}{\pi}
  \sum_{m\ge1}\frac{{e}^{{i}4tm^2}{e}^{{i}2mx}}{m}+C_t(x).
\end{equation}
Thus the one-sided parameter is \(4t/(2\pi)\) after the spatial
rescaling \(x\mapsto2x\).

\begin{proposition}[Exact singular set for the square wave]
\label{cor:square-nonvanishing}
Let \(u(t)=S(t)s\), and write \(4t/(2\pi)=P/Q\) in
lowest terms, with \(P\in\Z\) and \(Q\in\N\).  The set of
points at which \(u(t)\) has no local continuous representative is
\begin{equation}\label{square-singular-set}
  \Sigma_t=
  \begin{cases}
    \{\pi k/Q:0\le k<2Q\},&Q\text{ odd},\\[2pt]
    \{2\pi k/Q:0\le k<Q\},&Q\equiv0\pmod4,\\[2pt]
    \{\pi(2k+1)/Q:0\le k<Q\},&Q\equiv2\pmod4,
  \end{cases}
\end{equation}
where \(k\) ranges over integers and points are taken modulo \(2\pi\).
Every point in \(\Sigma_t\) is either a logarithmic singularity or
a nonzero jump, and \(u(t)\) is continuous away from this set.
In particular,
\[
  \#\Sigma_t=
  \begin{cases}
    2Q,&Q\text{ odd},\\
    Q,&Q\text{ even}.
  \end{cases}
\]
At each \(\xi\in\Sigma_t\), the coefficients of
\(\mathcal C_\xi\) and \(\mathcal J_\xi\) in
\eqref{eq:finite} satisfy
\begin{equation}\label{square-coefficient-norm}
  A_\xi(t)^2+B_\xi(t)^2=
  \begin{cases}
    4/(\pi^2Q),&Q\text{ odd},\\
    8/(\pi^2Q),&Q\text{ even}.
  \end{cases}
\end{equation}
\end{proposition}

\begin{proof}
Set
\[
  D_\ell:=\frac1Q\sum_{m=0}^{Q-1}
  {e}^{{i}4tm^2}{e}^{2\pi{i}\ell m/Q},
  \qquad 0\le\ell<Q.
\]
The sequence \(m\mapsto{e}^{{i}4tm^2}\) is \(Q\)-periodic.
Finite Fourier inversion gives
\[
  {e}^{{i}4tm^2}
  =\sum_{\ell=0}^{Q-1}D_\ell{e}^{-2\pi{i}\ell m/Q}.
\]
To determine which coefficients vanish, square the defining sum and
set \(h=m-k\) modulo \(Q\).  Since \((P,Q)=1\), orthogonality
of the finite Fourier characters yields
\begin{align*}
  |D_\ell|^2
  &=\frac1{Q^2}\sum_{h=0}^{Q-1}
    {e}^{2\pi{i}(Ph^2+\ell h)/Q}
    \sum_{k=0}^{Q-1}{e}^{2\pi{i}(2Phk)/Q}\\
  &=\frac1Q\sum_{\substack{0\le h<Q\\Q\mid2h}}
    {e}^{2\pi{i}(Ph^2+\ell h)/Q}.
\end{align*}
For odd \(Q\), only \(h=0\) contributes.  For even \(Q\),
the second possible value is \(h=Q/2\); its contribution is
\((-1)^{PQ/2+\ell}\), and \(P\) is odd.  Thus
\begin{equation}\label{square-gauss-magnitudes}
  |D_\ell|^2=
  \begin{cases}
    1/Q,&Q\text{ odd},\\
    (1+(-1)^\ell)/Q,&Q\equiv0\pmod4,\\
    (1-(-1)^\ell)/Q,&Q\equiv2\pmod4.
  \end{cases}
\end{equation}

For \(\xi_\ell=\pi\ell/Q\), the Fourier identity
\[
  \sum_{m\ge1}\frac{{e}^{2{i} m(x-\xi_\ell)}}{m}
  =L_{\xi_\ell}(x)+L_{\xi_\ell+\pi}(x)
\]
holds in \(L^2(\T)\): the odd modes on the right cancel, while the
two even-mode contributions add to the series on the left.
Substituting the finite Fourier expansion into
\eqref{square-even-profile} therefore yields
\[
  w_L(t,x)=\sum_{\ell=0}^{Q-1}\Gamma_\ell
  \bigl(L_{\xi_\ell}(x)+L_{\xi_\ell+\pi}(x)\bigr)+C_t(x),
  \qquad
  \Gamma_\ell=-\frac{{i}{e}^{-{i} t}}{\pi}D_\ell,
\]
where \(C_t\in C^0(\T)\). The \(2Q\) points
\(\xi_\ell,\xi_\ell+\pi\), \(0\le\ell<Q\), are distinct modulo
\(2\pi\), so no coefficients need to be combined.

At either point \(\xi=\xi_\ell\) or \(\xi=\xi_\ell+\pi\),
\eqref{eq:GKT} and Lemma~\ref{thm:holder-transfer} give the
logarithmic and jump coefficients
\[
  A_\xi=2\operatorname{Re}\bigl({i} a_t(\xi)\Gamma_\ell\bigr),
  \qquad
  B_\xi=-2\operatorname{Im}\bigl({i} a_t(\xi)\Gamma_\ell\bigr).
\]
The primitive of the real-valued function \(u(t)\) is real-valued,
so \(|a_t(\xi)|=1\). Consequently,
\[
  A_\xi^2+B_\xi^2=4|\Gamma_\ell|^2
  =\frac4{\pi^2}|D_\ell|^2.
\]
If \(D_\ell\ne0\), both associated points are therefore actual
singularities, by the same local argument used in the proof of
Proposition~\ref{lem:finite}.  If \(D_\ell=0\), the two kernels
are absent, and all remaining terms are continuous near those points.
The representation is also continuous away from the candidate set.

Equation~\eqref{square-gauss-magnitudes} now gives the complete list:
all \(\ell\) survive for odd \(Q\), only even \(\ell\) survive
when \(Q\equiv0\pmod4\), and only odd \(\ell\) survive when
\(Q\equiv2\pmod4\).  The paired points
\(\xi_\ell,\xi_\ell+\pi\) therefore give exactly
\eqref{square-singular-set}, with the stated cardinality.
Substituting the nonzero magnitudes into the preceding coefficient
identity proves \eqref{square-coefficient-norm}.
\end{proof}

\subsection{Comparison of singular coefficients}
For real $\lambda$, set
\[
 u_\lambda(t)=S(t)(\lambda s),\quad
 v_\lambda(t)=E(t)(\lambda s),\quad
 V_\lambda(t)=\partial_x^{-1}u_\lambda(t),\quad
 a_\lambda(t)=e^{iV_\lambda(t)}.
\]
The primitive is real, mean zero, and absolutely continuous. At rational
times Proposition~\ref{lem:finite} gives
\begin{equation}\label{amp:reconstruction}
 u_\lambda(t)=2\operatorname{Re}
   \sum_\xi i a_\lambda(t,\xi)\Gamma_\xi L_\xi+R,
 \qquad R\in C^0(\T).
\end{equation}
Thus the real singular coefficients at $\xi$ are
\begin{equation}\label{amp:AB}
 A_\xi=2\operatorname{Re}(ia_\lambda(t,\xi)\Gamma_\xi),
 \qquad B_\xi=-2\operatorname{Im}(ia_\lambda(t,\xi)\Gamma_\xi),
 \qquad A_\xi^2+B_\xi^2=4|\Gamma_\xi|^2.
\end{equation}
Distinct singular kernels cannot cancel modulo continuous functions:
near a fixed center, a nonzero $\mathcal C$ coefficient is unbounded,
whereas a nonzero $\mathcal J$ coefficient has unequal one-sided limits.
Consequently, if a linear solution has coefficients $\Lambda_\xi$ in
the representation $2\operatorname{Re}\sum i\Lambda_\xi L_\xi$, then
\begin{equation}\label{amp:matching}
 u_\lambda(t)-v_\lambda(t)\in C^0(\T)
 \quad\Longrightarrow\quad
 a_\lambda(t,\xi)\Gamma_\xi=\Lambda_\xi
 \quad\text{at every center }\xi.
\end{equation}
Membership in $C^0$ throughout means existence of a continuous representative.

\subsection{Amplitude dependence at odd half-periods}
\begin{proposition}\label{amp:amplitude}
Let $T=(2m+1)\pi/2$, where $m\in\Z$, and put
$\sigma=(-1)^m$ and $\theta=\pi\lambda/2$. The only possible singular
centers of $u_\lambda(T)$ are $0$ and $\pi$, and their gauge coefficients are
\begin{align}\label{amp:exactGamma}
 \Gamma_0&=-\frac{i\lambda}{\pi}
     (\sigma\cos\theta+\sin\theta)\e^{-iT\lambda^2},\\
 \Gamma_\pi&=\frac{i\lambda}{\pi}
     (\sigma\cos\theta-\sin\theta)\e^{-iT\lambda^2}.
 \notag
\end{align}
In particular,
\begin{equation}\label{amp:norms}
 A_0^2+B_0^2=\frac{4\lambda^2}{\pi^2}(1+\sigma\sin\pi\lambda),
 \qquad
 A_\pi^2+B_\pi^2=\frac{4\lambda^2}{\pi^2}(1-\sigma\sin\pi\lambda).
\end{equation}
If $\lambda\notin\Z$, the difference $u_\lambda(T)-v_\lambda(T)$
has no local continuous representative at either center.
\end{proposition}
\begin{proof}
The initial primitive is $\lambda(|x|-\pi/2)$. The jump formula in Section~\ref{sec:foundation} gives
\[
 c_0=-\frac{\lambda}{\pi}\e^{i\theta},\qquad
 c_\pi=\frac{\lambda}{\pi}\e^{-i\theta},\qquad
 M=\|\lambda s\|_2^2=\lambda^2.
\]
Since $\e^{iTn^2}=1$ for even $n$ and $i\sigma$ for odd $n$,
\[
 \e^{iTn^2}=d_0+d_2(-1)^n,\qquad
 d_0=\frac{1+i\sigma}{2},\quad d_2=\frac{1-i\sigma}{2}.
\]
Collecting the terms at the two centers gives
$\Gamma_0=\e^{-iT\lambda^2}(c_0d_0+c_\pi d_2)$ and
$\Gamma_\pi=\e^{-iT\lambda^2}(c_0d_2+c_\pi d_0)$.
These are exactly \eqref{amp:exactGamma}; \eqref{amp:norms} follows from
\eqref{amp:AB}. On the other hand,
\begin{equation}\label{amp:linearhalf}
 v_\lambda(T)=\frac{2\sigma\lambda}{\pi}
   (\mathcal C_0-\mathcal C_\pi).
\end{equation}
The linear singular coefficient has length $2|\lambda|/\pi$ at
each center. If $\lambda\notin\Z$, \eqref{amp:norms} gives a different
length at both centers, which excludes \eqref{amp:matching} locally.
More explicitly, the length of the difference coefficient at $0$ is at least
\[
 \frac{2|\lambda|}{\pi}
   \left|\,|\sigma\cos\theta+\sin\theta|-1\,\right|,
\]
and the analogous lower bound at $\pi$ has a minus sign before
$\sin\theta$. Both bounds are strictly positive.
\end{proof}

\begin{corollary}[The leading nonlinear correction]\label{amp:small}
At $T=\pi/2$, the coefficients of $u_\lambda(T)$ satisfy
\[
 A_0=\frac{2\lambda}{\pi}+\lambda^2+O(\lambda^3),\qquad
 A_\pi=-\frac{2\lambda}{\pi}+\lambda^2+O(\lambda^3),\qquad
 B_0,B_\pi=O(\lambda^3).
\]
Thus, for all sufficiently small nonzero $\lambda$, the nonlinear-minus-linear
difference has a nonzero logarithmic singularity at both centers.
\end{corollary}
\begin{proof}
The periodic mean-zero $L^2$ flow is real analytic near the origin on any
fixed compact time interval; see \cite[Theorem 1.2(ii)]{GKTanalytic}.
Its differential at zero is the linear group. Hence
\[
 u_\lambda(T)=\lambda E(T)s+O_{L^2}(\lambda^2).
\]
The map $\partial_x^{-1}:L^2_0\to H^1\hookrightarrow C^0$ is bounded.
The mean-zero primitive of the even function in \eqref{amp:linearhalf}
vanishes at both $0$ and $\pi$: it is the uniformly convergent sine
series obtained by dividing its Fourier coefficients by $in$.
It follows that $V_\lambda(T,0),V_\lambda(T,\pi)=O(\lambda^2)$.
Substituting this fact and \eqref{amp:exactGamma} into \eqref{amp:AB} gives
the stated expansions. This calculation uses the exact singular coefficient
formula; it does not interchange a Taylor expansion with a high-frequency limit.
\end{proof}

\subsection{The unit square wave at every nonzero rational time}
\begin{proposition}\label{amp:allrational}
Let $u(t)=S(t)s$ and $v(t)=E(t)s$. For every nonzero time
$t$ with $t/(2\pi)\in\mathbb Q$, the difference $u(t)-v(t)$
has no continuous representative on $\T$.
\end{proposition}

We first record the phase restriction supplied by the conserved norm.
\begin{lemma}\label{amp:phasevariation}
Let $V=\partial_x^{-1}u(t)$ and $a=\e^{iV}$. Then
\begin{equation}\label{amp:TV}
 \int_0^{2\pi}|V'(x)|\,\mathrm{d}x=\int_0^{2\pi}|u(t,x)|\,\mathrm{d}x\leq2\pi.
\end{equation}
If $a$ takes alternating opposite values at $N$ cyclically ordered points,
where $N$ is even, then $N\leq2$. If $a(0)=-i$ and $a(\pi)=i$,
then $u(t)=s$ almost everywhere. If $a(0)=i$ and $a(\pi)=-i$,
then $u(t)=-s$ almost everywhere.
\end{lemma}
\begin{proof}
Norm conservation gives $\|u(t)\|_2=1$, and Cauchy--Schwarz gives
\eqref{amp:TV}. Opposite values of $a$ imply that the change in its real
lift $V$ is an odd multiple of $\pi$ on each of the $N$ intervening
arcs. Summing the absolute changes gives $N\pi\leq2\pi$.
For the last two assertions, put $I=V(\pi)-V(0)$. Since $u(t)$ has
mean zero,
\[
 I=\int_0^\pi u(t,x)\,\mathrm{d}x=\pi\langle u(t),s\rangle,
 \qquad |I|\leq\pi.
\]
Either prescribed pair of phases makes $I$ an odd multiple of $\pi$.
Equality in Cauchy--Schwarz therefore forces $u(t)=s$ or $-s$.
Their mean-zero primitives are $|x|-\pi/2$ and its negative,
so their phases at $0,\pi$ select the sign claimed.
\end{proof}

\begin{proof}[Proof of Proposition~\ref{amp:allrational}]
Write $t/(2\pi)=p/q$ in lowest terms, with $q\geq1$, and define
\[
 d_\ell=\frac1q\sum_{j=0}^{q-1}
   \exp\left(\frac{2\pi i(pj^2+\ell j)}q\right).
\]
The nonlinear initial gauge coefficients are $c_0=c_\pi=-i/\pi$.
The corresponding linear coefficients are $b_0=-1/\pi$ and
$b_\pi=1/\pi$: indeed $-i\Pi_{>0}s$ has positive Fourier
coefficients $((-1)^n-1)/(\pi n)$, and
$v(t)=2\operatorname{Re}(iE_+(t)(-i\Pi_{>0}s))$, where
$E_+(t)$ multiplies the positive mode $n$ by $\e^{itn^2}$.
Suppose, for a contradiction, that the difference is continuous.
We apply \eqref{amp:matching} to the translated jump grids.

If $q$ is odd, all $d_\ell$ are nonzero. The two grids
$2\pi\ell/q$ and $\pi+2\pi\ell/q$ are disjoint and interlace.
On the first grid, \eqref{amp:matching} requires $a=-i\e^{it}$;
on the second, it requires $a=i\e^{it}$. There are $2q$
alternating phases, so Lemma~\ref{amp:phasevariation} excludes $q\geq3$.

If $q\equiv2\pmod4$, the nonzero Gauss coefficients are exactly
those with odd $\ell$. Since $q/2$ is odd, translation by $\pi$
switches the parity of the grid index. The two nonzero jump grids are
therefore disjoint, and their union consists of $q$ equally spaced
points with the same alternating phases as above.
Lemma~\ref{amp:phasevariation} excludes $q\geq6$.

Suppose next that $4\mid q$. The nonzero Gauss coefficients have
even indices. At $x_\ell=2\pi\ell/q$, with $\ell$ even, the two
jump contributions coincide and give
\begin{equation}\label{amp:combined}
 \Gamma_\ell=-\frac{i\e^{-it}}\pi
       (d_\ell+d_{\ell+q/2}),\qquad
 \Lambda_\ell=\frac1\pi(-d_\ell+d_{\ell+q/2}).
\end{equation}
Shifting the summation variable by $q/4$ gives the useful exact identity
\begin{equation}\label{amp:Gaussratio}
 r_\ell:=\frac{d_{\ell+q/2}}{d_\ell}
   =\exp\left(-\frac{i\pi pq}{8}-\frac{i\pi\ell}{2}\right).
\end{equation}
If $8\mid q$, then $r_\ell\in\{1,-1\}$.
Exactly one of the two coefficients in \eqref{amp:combined} vanishes,
so matching is impossible at every such center. If $q\equiv4\pmod8$,
then $r_\ell\in\{i,-i\}$ and its sign alternates as $\ell$ increases
by two. Matching requires
\[
 a(x_\ell)=i\e^{it}\frac{r_\ell-1}{1+r_\ell},
\]
which alternates between $\e^{it}$ and $-\e^{it}$ at $q/2$
equally spaced points. Lemma~\ref{amp:phasevariation} excludes $q\geq12$.

It remains to treat $q=1,2,4$. For $q=4$, write
$t=(2m+1)\pi/2$ and $\sigma=(-1)^m$. By
\eqref{amp:exactGamma}, $\Gamma_0=\Gamma_\pi=-\sigma/\pi$, whereas
the linear profile is $(2\sigma/\pi)(\mathcal C_0-\mathcal C_\pi)$.
Matching forces $a(0)=i$ and $a(\pi)=-i$.
Lemma~\ref{amp:phasevariation} gives $u(t)=-s$, a bounded function,
which cannot differ continuously from this logarithmically unbounded
linear profile. For $q=2$, $t$ is an odd multiple of $\pi$ and
$v(t)=-s$. Both nonlinear gauge coefficients equal $i/\pi$.
Matching forces $a(0)=-i$, $a(\pi)=i$, and therefore $u(t)=s$.
But $s-(-s)=2s$ is not continuous, again a contradiction.

Finally, let $q=1$, so $t=2\pi m$ for a nonzero integer $m$.
Now $v(t)=s$, the gauge coefficients equal $-i/\pi$, and matching
again forces $u(t)=s$. We exclude this recurrence using the BO
Birkhoff coordinates. Write $\gamma_n=|\zeta_n(s)|^2$.
The trace formula, frequency formula, Sobolev regularity of the Birkhoff
map, and diagonal flow identity are, respectively,
\cite[equations (13), (14), (21), (23), and (24)]{GKT-smoothing}:
\[
 2\sum_{n\geq1}n\gamma_n=1,\qquad
 \omega_n=n^2-1+\delta_n,\qquad
 \delta_n=2\sum_{k>n}(k-n)\gamma_k,
 \qquad
 \zeta_n(u(t))=\e^{it\omega_n}\zeta_n(s).
\]
Infinitely many $\gamma_n$ are nonzero. Otherwise the Birkhoff
sequence would belong to every weighted sequence space, and the inverse
Sobolev regularity theorem would imply $s\in H^{1/2}$, contrary to
its explicit Fourier coefficients. Hence $\delta_n>0$ for every $n$,
whereas the trace formula gives $\delta_n\to0$. Choose a sufficiently
large index $n$ with $\zeta_n(s)\ne0$ and $0<|m|\delta_n<1$.
Then $\e^{it\omega_n}=\e^{2\pi im\delta_n}\ne1$, which excludes
$u(t)=s$. This completes the proof.
\end{proof}

\begin{remark}
The theorem concerns the actual nonlinear solution and includes the
unit-amplitude square wave used in the numerical comparison. Its proof does
not infer a phase mismatch from the size of the gauge coefficients.
At the amplitudes and times where those sizes agree, it uses the variation
of the real primitive and, only at complete linear revivals, the exact
Birkhoff frequencies. At $t=0$ the difference vanishes, so the time
restriction is necessary.
\end{remark}

\subsection{A family of amplitudes}

The separation theorem extends to an interval of amplitudes. The upper
endpoint below is the range in which the conserved norm fixes the winding
between the two prescribed phase values.

\begin{proposition}\label{amp:smallfamily}
For every $0<|\lambda|\leq1$ and every nonzero rational time parameter
$t/(2\pi)$, the function $u_\lambda(t)-v_\lambda(t)$ has no continuous
representative on $\T$.
\end{proposition}
\begin{proof}
The case $\lambda=1$ is Proposition~\ref{amp:allrational}.
Translation by $\pi$ changes $\lambda s$ into $-\lambda s$ and
commutes with both flows, so it suffices to consider $0<\lambda<1$.
Put $\theta=\pi\lambda/2$ and write $t/(2\pi)=p/q$ in lowest terms.
Assume that the difference is continuous. Matching the coefficients on
two nonintersecting jump grids requires the values
\begin{equation}\label{amp:smallphases}
 a_\lambda(t,x)=e^{it\lambda^2-i\theta}
 \quad\hbox{on the grid from }0,\qquad
 a_\lambda(t,x)=e^{it\lambda^2+i\theta}
 \quad\hbox{on the grid from }\pi.
\end{equation}
The shortest angular distance between these values is $\pi\lambda$.
On the other hand, norm conservation gives
\[
 \int_\T|\partial_xV_\lambda(t,x)|\,\mathrm{d}x
 \leq2\pi\lambda.
\]
Thus the values in \eqref{amp:smallphases} cannot alternate at more
than two cyclically ordered points. The grid counts in the proof of
Proposition~\ref{amp:allrational} exclude every odd $q\geq3$ and every
$q\equiv2\pmod4$ with $q\geq6$.

For $4\mid q$, use the same centers and the ratio $r_\ell$ from
\eqref{amp:Gaussratio}. The coefficients now are
\[
 \Gamma_\ell=\frac{\lambda e^{-it\lambda^2}d_\ell}{\pi}
       (-e^{i\theta}+r_\ell e^{-i\theta}),\qquad
 \Lambda_\ell=\frac{\lambda d_\ell}{\pi}(-1+r_\ell).
\]
If $q\equiv4\pmod8$, then $r_\ell=\pm i$.
The squared factors in the two absolute values are, respectively,
$2\mp2\sin(\pi\lambda)$ and $2$. They differ because
$0<\lambda<1$, so coefficient matching is impossible.
If $8\mid q$, choose an even index with $r_\ell=1$; such indices
exist because the sign alternates. Then $\Lambda_\ell=0$ while
$\Gamma_\ell\ne0$, since $\sin\theta\ne0$.

For $q=1$, matching at $0$ and $\pi$ gives
\[
 V_\lambda(t,\pi)-V_\lambda(t,0)\in\pi\lambda+2\pi\Z.
\]
Its absolute value is at most $\pi\lambda$, by the same
Cauchy--Schwarz argument as in Lemma~\ref{amp:phasevariation}.
Since $\lambda<1$, the difference must be $\pi\lambda$.
Equality forces $u_\lambda(t)=\lambda s$. For $q=2$, the two
jump grids are interchanged, so the primitive difference instead lies
in $-\pi\lambda+2\pi\Z$. The same argument gives
$u_\lambda(t)=-\lambda s$.

Neither possibility occurs at the indicated times. In fact, for any
nonzero amplitude $\lambda$, the datum $\lambda s$ cannot return to
itself at a time $2\pi m$, $m\in\Z\setminus\{0\}$.
To see this, use its infinitely many active Birkhoff coordinates and write
\[
 \omega_n=n^2-\lambda^2+\delta_n,
 \qquad \delta_n=2\sum_{k>n}(k-n)|\zeta_k(\lambda s)|^2>0,
 \qquad \delta_n\longrightarrow0.
\]
A return would imply
$e^{-2\pi im\lambda^2+2\pi im\delta_n}=1$ at every active index.
Taking the limit along those indices first gives
$e^{-2\pi im\lambda^2}=1$. A sufficiently large active index with
$0<|m|\delta_n<1$ then gives a contradiction.
This excludes the $q=1$ case directly. In the $q=2$ case,
$u_\lambda(t)=-\lambda s$ and translation invariance would give
$u_\lambda(2t)=\lambda s$, again a return at a nonzero integer
multiple of $2\pi$. All cases have been excluded.
\end{proof}

\begin{remark}
Proposition~\ref{amp:amplitude} gives separation at odd
half-periods for every nonintegral amplitude, regardless of its size.
\end{remark}

\subsection{The irrational-time profile}\label{li:section}

For the unit square wave, put $\vartheta=4\alpha=2t/\pi$.
Its compressed spatial coordinate is $2\beta=x/\pi$. In this
subsection $u(t)=S(t)s$.

For \(|z|<1\), put
\[
 \mathscr F_\vartheta(z):=\sum_{n\ge1}\frac{e^{2\pi i\vartheta n^2}z^n}{n},
 \qquad
 f_\vartheta(x):=\mathscr F_\vartheta(e^{2ix}).
\]
The holomorphic function $\mathscr F_\vartheta$ has boundary values
$F_\vartheta(y)$ at $z=e^{2\pi iy}$, $y\in\T_1$, in the notation of
Section~\ref{ar:section}. The expression defining $f_\vartheta$ denotes its \(L^2(\T)\) Fourier series.  The
square-wave gauge coefficients satisfy
\[
 \widehat{\mathcal G(s)}(n)
 =-\frac{i(1+(-1)^n)}{\pi n}+O(n^{-3}),\qquad n\to\infty.
\]
The gauge smoothing theorem \cite{GKT-smoothing} therefore gives
\begin{equation}\label{li:reconstruction}
 u(t,x)=\frac2\pi\operatorname{Re}\bigl(b_t(x)f_\vartheta(x)\bigr)
       +R_t(x),
 \qquad
 b_t(x)=e^{i(V(t,x)-t)},\qquad V(t)=\partial_x^{-1}u(t),
\end{equation}
where \(R_t\in C^0(\T)\), \(|b_t|=1\), and
\(b_t\in C^\gamma(\T)\) for every \(\gamma<1\).  Indeed,
the square wave belongs to $H^\sigma$ for every $0\leq\sigma<1/2$;
persistence of this regularity \cite{GKT-wp} and \eqref{eq:GKT},
with $\sigma$ sufficiently close to $1/2$, give these assertions.  Throughout this section, time
is fixed.  No estimate uniform in time is needed.

For \(N\ge1\), let
\(r_N=e^{-1/N}\) and \(\rho_N=e^{-1/(2N)}\).  The factor of two in
\(\rho_N\) accounts for the physical frequencies \(2n\) in
\(f_\vartheta\).  Applying Lemma~\ref{li:commutator} to
\eqref{li:reconstruction} gives
\begin{equation}\label{li:poisson-reconstruction}
 \mathsf P_{\rho_N}u(t,x)
 =\frac2\pi\operatorname{Re}
       \bigl(b_t(x)\mathscr F_\vartheta(r_Ne^{2ix})\bigr)
   +\mathsf P_{\rho_N}R_t(x)+o(1),
 \qquad N\to\infty,
\end{equation}
uniformly in \(x\).

\begin{corollary}[The nonlinear square-wave endpoint]\label{en:BO}
Let \(u(t)=S(t)s\), with the square wave defined above, and suppose
\(\vartheta=2t/\pi\) is irrational.  Then
\[
 u(t)\in C^0(\T)
 \quad\Longleftrightarrow\quad
 u(t)\in L^\infty(\T)
 \quad\Longleftrightarrow\quad
 \mathcal B(\vartheta)<\infty.
\]
If \(\mathcal B(\vartheta)=\infty\), the solution is essentially unbounded
on every nonempty open spatial arc.
\end{corollary}
\begin{proof}
The square-wave reconstruction reads
\[
 u(t,x)=\frac2\pi\operatorname{Re}
       \bigl(b_t(x)F_\vartheta(x/\pi)\bigr)+R_t(x),
 \qquad b_t=e^{i(V(t)-t)},
\]
with \(R_t\) continuous, \(|b_t|=1\), and
\(b_t\in C^\rho\) for every \(\rho<1\).
If \(\mathcal B(\vartheta)<\infty\), continuity follows immediately.
Otherwise choose any sufficiently small spatial subarc and use the local
compressed coordinate $y=x/\pi=2\beta$.  The proof of
Proposition~\ref{en:classification}, applied with the corresponding local
phase, makes the Poisson means of the real gauge contribution unbounded
at a point of that subarc.  The physical Poisson radius is \(\sqrt r\),
since the compressed series has frequencies \(2n\).  The same commutator
estimate and the bounded continuous remainder give local essential
unboundedness of \(u(t)\).  Since the subarc was arbitrary, the last
assertion follows.
\end{proof}

\begin{corollary}\label{en:examples}
Every nonempty time interval contains uncountably many times for
which, with the convergents of $\vartheta=2t/\pi$,
\[
 \frac{\log q_{j+1}}{\sqrt{q_j}}\longrightarrow0,
 \qquad \mathcal B(\vartheta)=\infty,
\]
and $S(t)s$ is essentially unbounded on every open spatial arc.
There are also dense Liouville parameters for which $S(t)s$ is
continuous.
\end{corollary}
\begin{proof}
Choose an arbitrary sufficiently long continued-fraction prefix
inside the desired interval of $\vartheta$ values. At each subsequent
step choose either
\[
 a_{j+1}=\left\lfloor\frac{\exp(\sqrt{q_j}/(j+1))}{q_j}\right\rfloor
 \quad\hbox{or}\quad
 a_{j+1}=\left\lfloor\frac{\exp(\sqrt{q_j}/(j+1))}{q_j}\right\rfloor+1.
\]
The denominators grow at least geometrically every two steps,
so these are positive for all sufficiently large $j$, and
$\log q_{j+1}/\sqrt{q_j}\sim1/(j+1)$ for either choice.
Distinct infinite choices give distinct irrational numbers.
Corollary~\ref{en:BO} proves the unboundedness assertion.
For the continuous examples, use the construction in
Remark~\ref{ar:liouvillecontinuous} and the same corollary.
\end{proof}

\subsection{Comparison with numerical profiles}\label{sec:compatibility}

Alama Bronsard and Laurens \cite[Section~1.2 and Figure~1]{BL}
compare finite-dimensional approximations to the linear and nonlinear
BO evolutions of the square wave. The three rational times in their
figure are $\pi/2$, $\pi/3$, and $\pi/6$; the fourth time is
$\sqrt2\,\pi$. Their convergence theorem is an $L^2$ statement.
At each finite truncation the displayed profiles are trigonometric
polynomials, so the plots alone do not determine whether a limiting
feature is a jump or an unbounded logarithm.

Proposition~\ref{cor:square-nonvanishing} determines the singular
locations without reference to the numerical resolution. With
$Q$ the reduced denominator of $2t/\pi$, it gives
\[
\begin{array}{c|c|c|c}
 t&Q&\Sigma_t&\#\Sigma_t\\ \hline
 \pi/2&1&\{0,\pi\}&2\\
 \pi/3&3&\{k\pi/3:0\leq k<6\}&6\\
 \pi/6&3&\{k\pi/3:0\leq k<6\}&6
\end{array}
\]
where all positions are modulo $2\pi$. Every listed point is an
actual singularity, because the logarithmic and jump coefficients
cannot vanish simultaneously. Proposition~\ref{amp:allrational}
also shows that the nonlinear-minus-linear profile has no continuous
representative at each of these times. Thus the difference in the
sharp features of the two curves is consistent with a change in the
singular coefficients, as well as with a different continuous
remainder.

At $t=\sqrt2\,\pi$, the number $\alpha=\sqrt2/2$ is a quadratic
irrational and has bounded continued-fraction coefficients.
Theorem~\ref{main:regularity} therefore gives the stronger conclusion
$S(t)s\in C^{1/2}(\T)$. The absence of a macroscopic jump in
the corresponding numerical panel agrees with this statement.
For other irrational times, Theorem~\ref{main:continuity} gives
the precise distinction: continuity is equivalent to
$\mathcal B(t/(2\pi))<\infty$, while divergence makes the
square-wave solution essentially unbounded on every open arc.
Finite-resolution plots cannot resolve this arithmetic distinction
by themselves.

\appendix

\section{Finite sums of quadratic profiles}
\label{fep:section}

The endpoint criterion remains sharp when several translated quadratic
profiles occur.  Their locations may be arbitrary, and no separation
between their rational revival grids is needed.  The useful observation
is that translating a quadratic profile by the irrational rotation
\(2\alpha\) changes it, up to a continuous function, by a scalar phase.
A finite collection of such translated observations recovers every coefficient.

\begin{proposition}\label{fep:kernel}
Let \(\alpha\) be irrational, let \(a_1,\ldots,a_J\) be distinct
points of \(\T_1=\R/\Z\), and let \(c_1,\ldots,c_J\) be nonzero
complex numbers.  Define
\[
 G(\beta)=\sum_{j=1}^Jc_jF_\alpha(\beta-a_j)
 \quad\text{in }L^2(\T_1).
\]
If \(b\in C^0(\T_1)\) and \(I\) is an open arc on which \(b\)
is not identically zero, then
\begin{equation}\label{fep:local-iff}
 \operatorname{Re}(bG)\in L^\infty(I)
 \quad\Longleftrightarrow\quad
 \mathcal B(\alpha)<\infty.
\end{equation}
When this arithmetic sum is finite, \(G\) and
\(\operatorname{Re}(bG)\) have continuous representatives on
\(\T_1\).  In particular, when the sum diverges,
\(\operatorname{Re}(bG)\) is essentially unbounded on every such arc.
\end{proposition}

\begin{lemma}[Translation of a quadratic profile]\label{fep:translation}
For each integer \(k\ge0\) and \(a\in\T_1\),
\begin{equation}\label{fep:translation-identity}
 F_\alpha(\beta+2k\alpha-a)
 =e(-\alpha k^2-k\beta+ka)F_\alpha(\beta-a)+R_{k,a}(\beta),
 \qquad R_{k,a}\in C^0(\T_1).
\end{equation}
One can take \(R_{0,a}=0\) and
\(\|R_{k,a}\|_\infty\le2\sum_{n=1}^k n^{-1}\) for \(k\ge1\).
\end{lemma}
\begin{proof}
Put \(m=n+k\) in the series on the left.  After factoring out the
phase in \eqref{fep:translation-identity}, the difference from
\(F_\alpha(\beta-a)\) is
\[
 \sum_{m=k+1}^\infty
   \frac{k\,e(\alpha m^2+(\beta-a)m)}{m(m-k)}
 -\sum_{m=1}^k\frac{e(\alpha m^2+(\beta-a)m)}m.
\]
The first series converges absolutely and uniformly, and the sums of the
absolute coefficients of both terms are bounded by the same harmonic
sum.  This proves the identity in \(L^2\) and the stated bounds.
\end{proof}

\begin{lemma}[A finite family of observations inside an arc]
\label{fep:frame}
Under the hypotheses of Proposition~\ref{fep:kernel}, there are an integer
\(K\ge1\), a continuous nonnegative function \(\chi\) supported in
\(I\), and a constant \(\lambda>0\) such that, for every
\(\beta\in\T_1\) and \(z=(z_1,\ldots,z_J)\in\C^J\),
\begin{equation}\label{fep:finite-frame}
 \frac1K\sum_{k=0}^{K-1}\chi(\beta+2k\alpha)
 \left|\operatorname{Re}\left[
 b(\beta+2k\alpha)e(-\alpha k^2-k\beta)
           \sum_{j=1}^Jz_je(ka_j)\right]\right|^2
 \ge\lambda\sum_{j=1}^J|z_j|^2.
\end{equation}
\end{lemma}
\begin{proof}
Choose \(\chi\ge0\) supported in a compact subarc of \(I\), so that
\(w=\chi|b|^2\) is strictly positive on some open interval.
Write \(s=2\alpha\), and partition the indices into classes under
\[
 j\sim\ell\quad\Longleftrightarrow\quad
 a_j-a_\ell\in s\Z\pmod1.
\]
For each class \(C\), choose a reference point \(a_C\) and integers
\(m_j\) such that \(a_j=a_C+sm_j\pmod1\).  The integers within a
class are distinct because \(s\) is irrational and the \(a_j\) are
distinct.

We first calculate the limit, uniformly in \(\beta\), of
\[
 \frac1K\sum_{k=0}^{K-1}
 w(\beta+sk)\left|\sum_{j=1}^Jz_je(ka_j)\right|^2.
\]
Approximating \(w\) uniformly by trigonometric polynomials reduces the
calculation to finite geometric sums.  A mixed term with indices
\(j,\ell\) has nonzero limiting mean only when
\(a_j-a_\ell+ns\in\Z\) for some integer \(n\), which occurs
precisely within the classes above.  The resulting limit is
\begin{equation}\label{fep:gram-limit}
 Q_\beta(z)=\sum_C\int_{\T_1}w(y)
       \left|\sum_{j\in C}z_je(m_j(y-\beta))\right|^2\,dy.
\end{equation}
All convergence here is uniform for \(\beta\in\T_1\) and
\(\sum_j|z_j|^2=1\): there are finitely many mixed terms, and the
uniform approximation error is bounded by
\(J\|w-w_{\mathrm{approx}}\|_\infty\sum_j|z_j|^2\).

For each class, a nonzero trigonometric polynomial with distinct
frequencies \(m_j\) cannot vanish on an interval.  Its integral in
\eqref{fep:gram-limit} is therefore a positive definite quadratic form
in its coefficients.  Replacing those coefficients by
\(z_je(-m_j\beta)\) is a diagonal unitary change of variables.
Consequently, there is \(\lambda_0>0\), independent of \(\beta\),
such that
\begin{equation}\label{fep:positive-gram}
 Q_\beta(z)\ge\lambda_0\sum_{j=1}^J|z_j|^2.
\end{equation}

It remains to pass from the absolute square to the square of the real
part.  Their difference, after the factor \(1/2\), is the real part
of a finite sum of expressions
\[
 \frac1K\sum_{k=0}^{K-1}
 h(\beta+sk)e\bigl(-2\alpha k^2+
                         (a_j+a_\ell-2\beta)k\bigr),
 \qquad h=\chi b^2.
\]
Each expression tends to zero uniformly in \(\beta\).  Indeed,
uniformly approximate \(h\) by a trigonometric polynomial and use
\[
 \sup_{\theta\in\R}
 \left|\frac1K\sum_{k=0}^{K-1}
                  e(-2\alpha k^2+\theta k)\right|\longrightarrow0.
\]
This last limit follows from the uniform quadratic estimate
\eqref{en:weyl}: for any convergent denominator \(q\) of
\(-2\alpha\), first let \(K\to\infty\) and then \(q\to\infty\).
Thus the left side of \eqref{fep:finite-frame}, with variable \(K\),
converges uniformly on the same unit sphere to
\(\tfrac12Q_\beta(z)\).  A sufficiently large finite \(K\) and
\(\lambda=\lambda_0/4\) prove the lemma.
\end{proof}

\begin{proof}[Proof of Proposition~\ref{fep:kernel}]
If \(\mathcal B(\alpha)<\infty\), the single-profile criterion
already proves all asserted continuity statements.  For the converse,
put \(U=\operatorname{Re}(bG)\), and suppose \(U\in L^\infty(I)\).
For the fixed finite set of integers supplied by
Lemma~\ref{fep:frame}, equation \eqref{fep:translation-identity} gives
\[
 U(\beta+2k\alpha)
 =\operatorname{Re}\left[
 b(\beta+2k\alpha)e(-\alpha k^2-k\beta)
       \sum_{j=1}^J c_jF_\alpha(\beta-a_j)e(ka_j)\right]
       +E_k(\beta).
\]
Here $E_k\in C^0(\T_1)$. These identities hold simultaneously
outside a null set. Apply
\eqref{fep:finite-frame} pointwise there with
\(z_j=c_jF_\alpha(\beta-a_j)\).  The weight \(\chi\) vanishes
unless the observed point lies in \(I\), where \(U\) is essentially
bounded, and the finitely many \(E_k\) are bounded everywhere.
It follows that
\[
 \sum_{j=1}^J|c_jF_\alpha(\beta-a_j)|^2\le C
 \quad\text{for almost every }\beta\in\T_1.
\]
Since some \(c_j\ne0\), the single-profile classification forces
\(\mathcal B(\alpha)<\infty\), as claimed.
\end{proof}

\begin{corollary}[The actual Benjamin--Ono solution]
\label{fep:BO}
Suppose \(u_0\) satisfies the Fourier expansion \eqref{eq:FE},
with at least one nonzero coefficient $c_j$.  Fix a time for which
\(\alpha=t/(2\pi)\) is irrational.  For every nonempty open arc
\(I\subset\T\), the following are equivalent:
\[
 \mathcal B(\alpha)<\infty,
 \qquad u(t)\in C^0(\T),
 \qquad u(t)\in L^\infty(I).
\]
If the arithmetic sum diverges, \(u(t)\) is essentially unbounded on
every nonempty open arc.  In particular, this classification applies
to every real, mean-zero, piecewise \(C^{1,\eta}\) datum with a
finite nonempty set of genuine jumps, for \(\eta>0\).
\end{corollary}
\begin{proof}
Set \(\beta=x/(2\pi)\) and \(\widetilde a_j=a_j/(2\pi)\).
The finite gauge expansion and the smoothing reconstruction give
\[
 u(t,2\pi\beta)
 =2\operatorname{Re}\left[
 b_t(\beta)\sum_jc_jF_\alpha(\beta-\widetilde a_j)\right]
       +R_t(\beta),
 \quad
 b_t(\beta)=i e^{iV(t,2\pi\beta)-itM},
 \quad R_t\in C^0(\T_1).
\]
Here \(b_t\) is continuous and has modulus one.  The remainder from
\(\rho_n\) is continuous because its Sobolev regularity is
\(H^{1/2+\varepsilon}\); the nonlinear smoothing remainder is
continuous as well.  Apply Proposition~\ref{fep:kernel} on the rescaled
arc.  The last assertion follows from the integration-by-parts
verification of \eqref{eq:FE} for piecewise \(C^{1,\eta}\) data.
\end{proof}

\subsection{A category statement for exceptional times}\label{sec:category}

The exceptional-time conclusion is not confined to the square wave.
For data satisfying \eqref{eq:FE}, one can obtain a strong qualitative result
without specifying an arithmetic approximation rate.

\begin{proposition}\label{cat:main}
Suppose \eqref{eq:FE} holds and at least one $c_j$ is nonzero. There is
a dense $G_\delta$ set of times $t$ such that $u(t)$ is essentially
unbounded on every nonempty open spatial arc. This set can be chosen to
consist of irrational time parameters $t/(2\pi)$ and to have Lebesgue
measure zero. Every such irrational parameter is Liouville.
In particular, the assertion applies to every piecewise $C^{1,\eta}$
datum with a genuine jump.
\end{proposition}

\begin{proof}
Let $\mathcal A(\T)$ denote the Wiener algebra, with norm
$\|f\|_{\mathcal A}=\sum_n|\widehat f(n)|$. Norm conservation and
$\partial_x(a_t^2)=2iu(t)a_t^2$ give
\begin{equation}\label{cat:wiener}
 \|a_t^2\|_{\mathcal A}
 \le \left(\sum_{n\in\Z}\frac1{1+n^2}\right)^{1/2}
       (1+4M)^{1/2}=:C_M
\end{equation}
uniformly for all real $t$.

If $q$ is odd and $(p,q)=1$, completing the square gives
\[
 d_\ell=d_0 e^{-2\pi i\nu\ell^2/q},\qquad
 |d_0|=q^{-1/2},\qquad 4p\nu\equiv1\pmod q.
\]
For every $h\in\mathcal A(\T)$ and every $a\in\T$,
\begin{equation}\label{cat:gauss}
 \left|\sum_{\ell=0}^{q-1}
 h(a+2\pi\ell/q)d_\ell^2\right|
 \le q^{-1/2}\|h\|_{\mathcal A}.
\end{equation}
To prove this, expand $h$ absolutely. Each Fourier mode leaves a complete
quadratic Gauss sum whose quadratic coefficient $-2\nu$ is coprime to
$q$, so its modulus is $\sqrt q$; the factor $|d_0|^2$ is $1/q$.

Fix $j$ with $c_j\ne0$, an open spatial arc $I$, and a nonzero smooth
function $\chi\ge0$ supported in $I$. Choose odd prime denominators $q$
large enough that none of the translated grids of distinct initial points overlap.
Such primes exist: a collision requires a nonzero difference
$(a_j-a_k)/(2\pi)$ to be rational with reduced denominator dividing $q$;
it suffices to take $q$ larger than all the finitely many such denominators.
At $x_\ell=a_j+2\pi\ell/q$, Proposition~\ref{lem:finite} gives
$A_\ell=2\operatorname{Re}(i e^{-itM}c_j a_t(x_\ell)d_\ell)$.
Consequently \eqref{cat:gauss} implies
\begin{align}\label{cat:weighted}
 \sum_{\ell=0}^{q-1}\chi(x_\ell)A_\ell^2
 &=\frac{2|c_j|^2}{q}\sum_{\ell=0}^{q-1}\chi(x_\ell)
       +\mathcal E_q,\\
 |\mathcal E_q|&\le
       \frac{2|c_j|^2}{\sqrt q}\|\chi a_t^2\|_{\mathcal A}
 \le \frac{2|c_j|^2C_M}{\sqrt q}\|\chi\|_{\mathcal A}.
 \notag
\end{align}
The leading Riemann sum tends to
$2|c_j|^2(2\pi)^{-1}\int_\T\chi>0$. Thus all sufficiently large
chosen prime denominators force a logarithmic singularity inside $I$,
uniformly in the coprime numerator $p$. Rational times with these
denominators are dense in $\R$.

Take a countable base $\mathcal I$ of open arcs. For $I\in\mathcal I$
and $N\in\N$, define
\[
 F_{I,N}=\{t:\|u(t)\|_{L^\infty(I)}\le N\}.
\]
This set is closed: if $t_k\to t$, the flow converges in $L^2$, and
a subsequence converges almost everywhere on $I$, preserving the bound.
It has empty interior by the dense rational logarithmic singularities
just proved. Baire's theorem shows that
$\bigcap_{I,N}(\R\setminus F_{I,N})$ is dense $G_\delta$.
At each of its times, every open arc has infinite essential supremum.

It remains to justify the assertions about measure and arithmetic for
the full class \eqref{eq:FE}. At a finite-type irrational time,
the quadratic kernel of Section~\ref{ar:section} is continuous.
The finite sum of its translates in $w_L$, together with the continuous
remainder furnished by \eqref{eq:FE}, makes $u(t)$ continuous by
\eqref{eq:GKT}. Almost every irrational has finite type; conversely,
every irrational that is not Liouville satisfies a finite-type inequality
with some finite exponent and positive constant. Hence the unbounded
times are null, and their irrational parameters must be Liouville.
Remove the countable set $2\pi\Q$ from the dense $G_\delta$ set above.
The remaining set is still dense $G_\delta$ and has all the asserted
properties.
\end{proof}

\begin{corollary}[Exceptional times for arbitrary BV jumps]\label{cat:BV}
The conclusion of Proposition~\ref{cat:main} also holds for every
real mean-zero BV datum with a nonzero jump, without the
Fourier assumption \eqref{eq:FE}.
\end{corollary}
\begin{proof}
For each open arc, Proposition~\ref{rs:bv} supplies logarithmic
Poisson coefficients of nonzero size there at a dense set of
rational times: every sufficiently large reduced denominator
is allowed, uniformly in its numerator. A nonzero coefficient
forces local essential unboundedness. The closed-set and Baire
argument in the preceding proof therefore applies without change.
The BV continuity theorem at finite-type irrational times shows
that the resulting irrational exceptional times are Liouville
and have measure zero.
\end{proof}

\section{Dependence on the amplitude}\label{sec:amplitudefamilies}
\label{rg:section}

The restriction on the amplitude in the square-wave comparison can be
removed outside a countable exceptional set. In fact, the same conclusion
holds along every amplitude family generated by a piecewise smooth datum
with a jump. The argument also shows that the difference usually contains
a logarithm, rather than merely a jump. Its main ingredient is the
quadratic dependence of the mass correction on the amplitude.

Fix a real, mean-zero function $f$ which is piecewise $C^{1,\eta}$,
$0<\eta\leq1$, on a finite partition of $\T$, and suppose that $f$ has
at least one nonzero jump. Write its distinct jump points as $a_1,\ldots,a_J$
and set
\[
 b_j=-\frac{[f]_{a_j}}{2\pi},\qquad
 v_j=(\partial_x^{-1}f)(a_j),\qquad M=\|f\|_2^2>0.
\]
For a fixed nonzero rational time $t=2\pi p/q$, with $(p,q)=1$ and $q>0$,
let $d_\ell$ be the Gauss coefficients in Proposition~\ref{lem:finite}.
For each distinct point in the finite union
$\{a_j+2\pi\ell/q:1\leq j\leq J,\ 0\leq\ell<q\}$, define
\begin{equation}\label{rg:polynomial}
 P_\xi(z)=\sum_{a_j+2\pi\ell/q=\xi}
              b_jd_\ell e^{-izv_j},\qquad
 \Lambda_\xi=P_\xi(0).
\end{equation}
We call $\xi$ active if $P_\xi$ is not identically zero. This definition
allows for collisions of translated jumps and for cancellations between
terms with the same value of $v_j$.

\begin{proposition}[Separation along every amplitude family]\label{rg:ray}
For each $f$ and $t$ as above, there is a closed discrete set
$Z_t(f)\subset\R$, containing zero, such that the following holds for
every $\lambda\notin Z_t(f)$. At every active center $\xi$, both
\[
 S(t)(\lambda f)
 \quad\hbox{and}\quad
 S(t)(\lambda f)-E(t)(\lambda f)
\]
have a nonzero logarithmic coefficient. In particular, both functions
are essentially unbounded on every neighborhood of that center.
There is at least one active center at each rational time.

Consequently, for each such $f$ there is an at most countable set
$Z(f)\subset\R$ such that these conclusions hold simultaneously at
every nonzero rational time for every $\lambda\notin Z(f)$.
\end{proposition}

Here and below a logarithmic coefficient refers to the representation
modulo a continuous function by the kernels $\mathcal C_\xi$ and
$\mathcal J_\xi$. A discrete exceptional set is finite on every compact
amplitude interval. The countable union over all rational times need
not be discrete.

We first prove the elementary fact which prevents an amplitude family
from remaining confined to a fixed real projection.

\begin{lemma}[Quadratic phase and exponential polynomials]\label{rg:winding}
Let $P(x)=\sum_{j=1}^L z_j e^{i\nu_jx}$ be a nonzero finite exponential
polynomial with real frequencies $\nu_j$. Suppose that $W:\R\to\R$
is continuous and $|W(x)|\leq C(1+|x|)$. If $\tau\ne0$, then
\[
 x\longmapsto\operatorname{Im}\big(e^{i(W(x)-\tau x^2)}P(x)\big)
\]
is not constant.
\end{lemma}
\begin{proof}
We use the fact that a nonzero real or imaginary part of $P$ has only
$O(R)$ distinct zeros on $[0,R]$. For completeness, either part extends
to a nonzero entire exponential polynomial $H$ satisfying
$|H(z)|\leq C_0e^{C_1|z|}$. Choose a real $x_0$ with $H(x_0)\ne0$.
Jensen's formula on the disk centered at $x_0$ of radius $4R$ bounds
the number of zeros in the concentric disk of radius $2R$ by $C_2R$,
for all sufficiently large $R$. This includes all real zeros in $[0,R]$.

Suppose that
$Q(x):=e^{i(W(x)-\tau x^2)}P(x)$ has constant imaginary part $c$.
First let $c=0$. Then $Q$ is real, and
\[
 P(x)=Q(x)e^{i\phi(x)},\qquad \phi(x)=\tau x^2-W(x).
\]
If $\operatorname{Re}P$ is not identically zero, it vanishes whenever
$\phi(x)\in\pi/2+\pi\Z$. Otherwise $\operatorname{Im}P$ is not
identically zero and vanishes whenever $\phi(x)\in\pi\Z$.
Since $\phi$ is continuous and
$\phi(R)-\phi(0)=\tau R^2+O(R)$, the intermediate value theorem
produces at least $c_0R^2$ distinct such points in $[0,R]$, for large
$R$. Different phase levels give different points. This contradicts
the preceding zero count; zeros of $Q$ cause no difficulty.

If $c\ne0$, then $Q$ lies in a fixed open half-plane and has a
continuous argument $\vartheta$ taking values in an interval of length
$\pi$. Thus
\[
 P(x)=|Q(x)|e^{i\psi(x)},\qquad
 \psi(x)=\tau x^2-W(x)+\vartheta(x).
\]
The same argument applies to $\psi$, since
$\psi(R)-\psi(0)=\tau R^2+O(R)$. Again one of the two nonzero
real components of $P$ would have at least $c_0R^2$ zeros. This is
impossible.
\end{proof}

\begin{proof}[Proof of Proposition~\ref{rg:ray}]
Put
\[
 V_\lambda(t)=\partial_x^{-1}S(t)(\lambda f),\qquad
 Q_\xi(\lambda)=
    e^{i(V_\lambda(t,\xi)-tM\lambda^2)}P_\xi(\lambda).
\]
The global analyticity theorem
\cite[Theorem~1.2(ii)]{GKTglobalanalytic}, with $s=0$, shows that
$\lambda\mapsto S(t)(\lambda f)$ is real analytic on all of $\R$
with values in $L^2_0$. Since evaluation after $\partial_x^{-1}$ is
a bounded real linear functional on $L^2_0$, each $Q_\xi$ is real
analytic. This use of analyticity is global in the amplitude; a
theorem restricted to a neighborhood of zero would not suffice.

Norm conservation and Fourier Cauchy--Schwarz give the uniform bound
\begin{equation}\label{rg:linearbound}
 |V_\lambda(t,\xi)|
 \leq\left(\sum_{n\ne0}\frac1{n^2}\right)^{1/2}
       \|S(t)(\lambda f)\|_2
 =\frac{\pi}{\sqrt3}|\lambda|\|f\|_2.
\end{equation}
For an active center, Lemma~\ref{rg:winding}, with $\tau=tM$, implies
that $\operatorname{Im}Q_\xi$ is not constant. Therefore each of
the two real analytic functions
\[
 \operatorname{Im}Q_\xi(\lambda),\qquad
 \operatorname{Im}Q_\xi(\lambda)-\operatorname{Im}\Lambda_\xi
\]
is nonzero as a function. Its zero set is closed and discrete.

The initial nonlinear gauge coefficient for $\lambda f$ is
$\lambda b_je^{-i\lambda v_j}$. The rational-time representation
therefore yields
\begin{align*}
 S(t)(\lambda f)
   &=2\operatorname{Re}\sum_\xi i\lambda Q_\xi(\lambda)L_\xi
       +R_{\lambda,t},\\
 E(t)(\lambda f)
   &=2\operatorname{Re}\sum_\xi i\lambda\Lambda_\xi L_\xi
       +\widetilde R_{\lambda,t},
 \qquad R_{\lambda,t},\widetilde R_{\lambda,t}\in C^0.
\end{align*}
Their logarithmic coefficients, and those of their difference, are
respectively
\[
 -2\lambda\operatorname{Im}Q_\xi(\lambda),\qquad
 -2\lambda\operatorname{Im}\Lambda_\xi,
 \qquad
 -2\lambda\big(\operatorname{Im}Q_\xi(\lambda)
                    -\operatorname{Im}\Lambda_\xi\big).
\]
Take the finite union of the zero sets just described and add zero.
This is a closed discrete set $Z_t(f)$ and has the required property.

To check that the active set is nonempty, let
$\mu=\sum_j b_j\delta_{a_j}$ and
$\nu=\sum_\ell d_\ell\,\tau_{2\pi\ell/q}\mu$, where
$\tau_h\delta_a=\delta_{a+h}$. The coefficients of $\nu$ are
$\Lambda_\xi$, and its Fourier coefficients satisfy
\[
 \widehat\nu(n)=e^{itn^2}\widehat\mu(n),\qquad n\in\Z.
\]
If all $\Lambda_\xi$ vanished, uniqueness of the Fourier coefficients
of a finite measure would give $\mu=0$, contrary to the existence of
a nonzero jump. Finally set
$Z(f)=\bigcup_{t\in2\pi\Q\setminus\{0\}}Z_t(f)$.
\end{proof}

The same reasoning gives a statement about perturbations of the initial
profile. Fix a finite partition $\mathcal P$ of the circle and
$0<\eta\leq1$. Let $X_{\mathcal P,\eta}$ be the real Banach space
of mean-zero functions which belong to $C^{1,\eta}$ on each closed
piece, with independent one-sided endpoint values and the sum of the
piecewise $C^{1,\eta}$ norms. Assume that this space contains a function
with a nonzero jump.

\begin{corollary}[Generic initial profiles]\label{rg:generic}
For each nonzero rational time $t$, the set of
$f\in X_{\mathcal P,\eta}$ for which both $S(t)f$ and
$S(t)f-E(t)f$ have a nonzero logarithmic coefficient somewhere on
$\T$ is open and dense. The set for which this holds simultaneously
at every nonzero rational time is a dense $G_\delta$ subset of
$X_{\mathcal P,\eta}$.

If $Y\subset X_{\mathcal P,\eta}$ is any finite-dimensional real
linear subspace containing a function with a nonzero jump, the
simultaneous conclusion also holds for Lebesgue-almost every $f\in Y$.
\end{corollary}
\begin{proof}
For a fixed $t$, use the union of the translated partition endpoints
as a common finite set of centers; a vanishing jump simply contributes
zero. Each logarithmic coefficient is a real analytic function of
$f\in X_{\mathcal P,\eta}$. Indeed, jump evaluation and evaluation
of the primitive at an endpoint are continuous linear functionals,
$f\mapsto\|f\|_2^2$ is quadratic, and the remaining dependence is
through the globally real analytic $L^2$ flow.

Thus the sum of the squared logarithmic coefficients of $S(t)f$,
and the analogous sum for $S(t)f-E(t)f$, are nonnegative real analytic
functions, say $H_t$ and $K_t$. Their common positivity is an open
condition. Profiles with at least one nonzero jump are dense in
$X_{\mathcal P,\eta}$, since the subspace on which all endpoint
jumps vanish is proper and closed. For any such profile $f$,
Proposition~\ref{rg:ray} supplies amplitudes arbitrarily close to one
for which both sums are positive. This proves density. The simultaneous
statement follows from Baire's theorem.

On $Y$, neither $H_t$ nor $K_t$ is identically zero: apply the
amplitude theorem to any member of $Y$ with a jump. The zero set of
a nonzero real analytic function on a finite-dimensional real vector
space has Lebesgue measure zero. One can see this by choosing a
direction with a nonzero Taylor coefficient and applying the
one-dimensional isolated-zero property and Fubini's theorem locally.
Taking the countable union of these zero sets over rational $t\ne0$
proves the final assertion.
\end{proof}

For each fixed BV datum with a jump, Proposition~\ref{rs:bv}
excludes complete matching of the singular coefficients at all
sufficiently large reduced denominators.
The openness assertion above concerns one fixed time and one fixed
partition; the simultaneous set over rational times is residual.

\subsection{Repeated changes of singularity type as the amplitude grows}

The exceptional amplitudes in Proposition~\ref{rg:ray} are unavoidable if
one asks that the nonlinear logarithmic coefficient itself never vanish.
At each active center, the singularity changes between a pure jump and
a pure logarithm at least quadratically often as the amplitude increases.
This is a statement about one fixed center; it imposes no simultaneous
cancellation at different centers.

\begin{proposition}\label{rg:transitions}
Use the notation and hypotheses of Proposition~\ref{rg:ray}, and fix one
active center $\xi$. Let $N_{\mathrm{jump}}(R)$ count the distinct
$\lambda\in(0,R]$ at which the nonlinear profile $S(t)(\lambda f)$
has zero logarithmic coefficient and nonzero jump coefficient at $\xi$.
Let $N_{\mathrm{log}}(R)$ count those at which it has a nonzero logarithmic
coefficient and zero jump coefficient there. Then, for all $R\geq1$,
\[
 N_{\mathrm{jump}}(R)\geq\frac{|t|\|f\|_2^2}{\pi}R^2-C_{f,t,\xi}(1+R),
 \qquad
 N_{\mathrm{log}}(R)\geq\frac{|t|\|f\|_2^2}{\pi}R^2-C_{f,t,\xi}(1+R).
\]
Both counting functions are finite at every finite $R$.
At a counted pure-jump amplitude, the solution has a bounded
representative on a punctured neighborhood of $\xi$ with unequal
one-sided limits. At a counted pure-logarithm amplitude, its singular
part there is a nonzero multiple of $\mathcal C_\xi$.
\end{proposition}

\begin{proof}
Write $P=P_\xi$. The real zeros of $P$ in $[0,R]$ number at most
$C_P(1+R)$ by the exponential-type zero estimate in the proof of
Lemma~\ref{rg:winding}. On the complement of these zeros, put
\[
 Z(x)=\frac{P(x)}{\overline{P(x)}}.
\]
This map extends continuously to a circle-valued function on the whole
real line. Indeed, at a real zero $x_0$ of order $m$, the entire function
$P$ has the local form $(z-x_0)^m h(z)$ with $h(x_0)\ne0$, and for
real $x\ne x_0$ the ratio above equals $h(x)/\overline{h(x)}$.
The factor carrying the zero cancels exactly. Choose a continuous lift
\[
 Z(x)=e^{2i\Theta(x)},\qquad x\in\R.
\]

We claim that $|\Theta(R)-\Theta(0)|\leq C_P(1+R)$ for $R\geq0$.
If $\operatorname{Im}P$ vanishes identically on the real line, then
$Z\equiv1$ and its continuous lift is constant. Otherwise
$\operatorname{Im}P$ has only $O(1+R)$ zeros in $[0,R]$.
Whenever $\Theta$ reaches a level in $\pi\Z$, either $P=0$ or
$Z=1$ and $P$ is real. In both cases $\operatorname{Im}P=0$.
The intermediate value theorem, applied between $0$ and $R$, therefore
bounds the number of integer multiples of $\pi$ between the endpoint
values of $\Theta$ by $C_P(1+R)$. Different levels require different
points. This proves the claim even when $\Theta$ is not monotone.

Recall that
\[
 Q_\xi(\lambda)=e^{i(V_\lambda(t,\xi)-tM\lambda^2)}P(\lambda),
 \qquad M=\|f\|_2^2.
\]
Set
\[
 \Phi(\lambda)=V_\lambda(t,\xi)-tM\lambda^2+\Theta(\lambda).
\]
This is a continuous real function, and wherever $P(\lambda)\ne0$,
\[
 \frac{Q_\xi(\lambda)}{\overline{Q_\xi(\lambda)}}=e^{2i\Phi(\lambda)}.
\]
The preceding bound for $\Theta$ and \eqref{rg:linearbound} give
\[
 \Phi(R)-\Phi(0)=-tMR^2+O_{f,t,\xi}(1+R).
\]
There are consequently at least
$|t|MR^2/\pi-O_{f,t,\xi}(1+R)$ distinct points in $[0,R]$
where $\Phi\in\pi\Z$. Remove zero and all real zeros of $P$;
this removes only $O_P(1+R)$ points. At each remaining amplitude,
$Q_\xi$ is real and nonzero. The coefficient formulas
\[
 A_\xi(\lambda)=-2\lambda\operatorname{Im}Q_\xi(\lambda),\qquad
 B_\xi(\lambda)=-2\lambda\operatorname{Re}Q_\xi(\lambda)
\]
show that the singularity is a pure nonzero jump. Counting instead the
levels in $\pi/2+\pi\Z$ makes $Q_\xi$ purely imaginary and nonzero,
and gives the same lower bound for pure logarithms.

Finally, $\operatorname{Im}Q_\xi$ is a nonzero real analytic function
by Lemma~\ref{rg:winding}. The same lemma applied to $iP$ shows that
$\operatorname{Re}Q_\xi$ is not identically zero either. Thus their
zero sets are finite on compact intervals, proving the finiteness of
both counting functions. The asserted local profiles follow from the
rational-time representation and the continuous remainder.
\end{proof}

\section{An explicit Liouville example at a prescribed spatial point}
\label{fp:section}

The continuity classification does not require the point of divergence
to be fixed in advance. The following construction gives a simple
example in which the complex Abel means diverge at $\beta=0$.

\begin{proposition}\label{fp:example}
There is a Liouville number $\alpha$ such that
\[
 \sup_{0<r<1}\left|
       \sum_{n\geq1}\frac{e(\alpha n^2)}n r^n\right|=\infty.
\]
In particular, $F_\alpha$ has no essentially bounded representative.
\end{proposition}
\begin{proof}
Let $p/q$ be reduced with $q$ odd, and put
\[
 \mu_{p,q}=\frac1q\sum_{a=0}^{q-1}e(pa^2/q).
\]
The change of variables $(a,b)\mapsto(a-b,a+b)$ is invertible
modulo $q$. Finite Fourier orthogonality therefore gives
\[
 q^2|\mu_{p,q}|^2
 =\sum_{a,b\bmod q}e\bigl(p(a^2-b^2)/q\bigr)
 =\sum_{h,k\bmod q}e(phk/q)=q,
\]
so $|\mu_{p,q}|=q^{-1/2}$. The sequence
$e(pn^2/q)-\mu_{p,q}$ is $q$-periodic with zero mean, and its
partial sums have modulus at most $2q$. Summation by parts with
the decreasing weights $r^n/n$ consequently yields
\begin{equation}\label{fp:rational}
 \sum_{n\geq1}\frac{e(pn^2/q)}n e^{-n/N}
 =\mu_{p,q}\log N+O(q),\qquad N\geq2,
\end{equation}
with an absolute implied constant. Here we used
$\sum_{n\geq1}e^{-n/N}/n=-\log(1-e^{-1/N})=\log N+O(1)$.

Suppose $|\alpha-p/q|\leq e^{-4q^2}$ and take $N=e^{q^2}$.
Then
\begin{align*}
 \left|\sum_{n\geq1}
  \frac{e(\alpha n^2)-e(pn^2/q)}n e^{-n/N}\right|
 &\leq C|\alpha-p/q|\sum_{n\geq1}n e^{-n/N}\\
 &\leq Ce^{-4q^2}N^2=Ce^{-2q^2}.
\end{align*}
Combining this with \eqref{fp:rational} gives
\begin{equation}\label{fp:growth}
 \left|\sum_{n\geq1}\frac{e(\alpha n^2)}n e^{-n/N}\right|
 \geq q^{3/2}-Cq-Ce^{-2q^2}.
\end{equation}

Choose increasing positive integers $M_j$ so rapidly that
\[
 3^{-M_{j+1}}<\exp(-5\cdot3^{2M_j}),
 \qquad
 \alpha=\sum_{j\geq1}3^{-M_j}.
\]
Its base-three expansion has infinitely many digits equal to one,
separated by unbounded strings of zeros. It is neither terminating
nor eventually periodic, so $\alpha$ is irrational. Write
\[
 \frac{p_j}{q_j}=\sum_{\ell=1}^j3^{-M_\ell},\qquad
 q_j=3^{M_j}.
\]
Since $p_j\equiv1\pmod3$, these fractions are reduced and their
denominators are odd. Moreover,
\[
 0<\alpha-\frac{p_j}{q_j}
 \leq2\cdot3^{-M_{j+1}}<e^{-4q_j^2}.
\]
These approximations make $\alpha$ Liouville. Applying
\eqref{fp:growth} with $N_j=e^{q_j^2}$ proves unboundedness
along $r_j=e^{-1/N_j}\uparrow1$.

The displayed Abel means are the Poisson means of the $L^2$
function $F_\alpha$ at $\beta=0$. Essential boundedness of that
function would bound all these means by $\|F_\alpha\|_\infty$,
which is impossible.
\end{proof}

This example concerns the complex quadratic profile. A conclusion
about the actual real BO solution uses the separate gauge reconstruction
and real-projection arguments proved in the main text.

\end{document}